\documentclass[11pt,a4paper]{amsart}
\usepackage[utf8]{inputenc}
\usepackage[english]{babel}
\usepackage[left=2cm,right=2cm,top=2cm,bottom=2cm]{geometry}
\usepackage{graphicx}
\usepackage{amssymb}
\usepackage{amsmath}
\usepackage{amsthm}
\usepackage{algorithm}
\usepackage{algorithmic}
\usepackage{comment}
\usepackage{pdfpages}
\usepackage{appendix}
\usepackage{stmaryrd}
\usepackage{tikz}
\usepackage{tikz-cd} 

\usepackage{hyperref}
\hypersetup{
  colorlinks   = true, 
  linkcolor    = black,
  citecolor    = blue      
}

\providecommand{\ed}{\mathrm e}
\providecommand{\prob}{\mathbb P}

\providecommand{\N}{\mathbb N}
\providecommand{\Z}{\mathbb Z}
\providecommand{\R}{\mathbb R}

\providecommand{\C}{\mathbb C}

\newcommand\norm[1]{\left\lVert#1\right\rVert}
\providecommand{\Ocal}{\mathcal{O}}
\providecommand{\Gr}{\mathbb Gr}
\providecommand{\diff}{\mathrm d}

\numberwithin{equation}{section}
\newtheorem{theorem}{Theorem}[section]
\newtheorem{proposition}[theorem]{Proposition}
\newtheorem{corollary}[theorem]{Corollary}
\newtheorem{lemma}[theorem]{Lemma}
\newtheorem{definition}[theorem]{Definition}

\newtheorem{remark}[theorem]{Remark}

\title{Positive formula for the product of conjugacy classes on the unitary group}
\author
{Quentin François}
\address
{Quentin François: CEREMADE, UMR CNRS 7534, Universit\'e
Paris-Dauphine, PSL Research university, Place du Mar\'echal
de Lattre de Tassigny 75016 Paris, France.}
\email{\href{mailto:quentin.francois@dauphine.psl.eu}{quentin.francois@dauphine.psl.eu}}

\author
{Pierre Tarrago}
\address{Pierre Tarrago: Laboratoire de Probabilit\'es, Statistique et Mod\'elisation, UMR
CNRS 8001, Sorbonne Universit\'e, 4 Place Jussieu, 75005 Paris, France.}
\email{\href{mailto:pierre.tarrago@sorbonne-universite.fr}{pierre.tarrago@sorbonne-universite.fr}}
\begin{document}
\begin{abstract}
The convolution product of two conjugacy classes of the unitary group $U_n$ is described by a probability distribution on the space of central measures. Relating this convolution to the quantum cohomology of Grassmannians and using recent results describing the structure constants of the latter, we give a manifestly positive formula for the density of the probability distribution for the product of generic conjugacy classes. In the same flavor as the hive model of Knutson and Tao, this formula is given in terms of a subtraction-free sum of volumes of explicit polytopes. As a consequence, this expression also provides a positive and explicit formula for the volume of $SU_n$-valued flat connections on the three-holed two dimensional sphere, which was first given by Witten in terms of an infinite sum of characters.
\end{abstract}
\maketitle

\section{Introduction}
 
Irreducible characters and conjugacy classes of a compact group are two dual facets of its space of central measures. Thanks to the underlying group multiplication, this space is given a convolution product which allows to expand the product of two characters along the basis of irreducible characters. The same holds in the case of convolution of conjugacy classes, for which the expansion is described by a central measure on the group. It is then a classical fact that both problems involve only non-negative quantities, and it is in general very hard to provide manifestly positive formulas for those expansions (see \cite{murnaghan1937direct} and \cite{JACKSON1988363} for the two corresponding open problems in the case of the symmetric group $S_n$). 

In the framework of compact Lie groups, the problem of finding positive combinatorial formulas for the expansion of characters has been solved in the case of the unitary group $U_n$ by Littlewood and Richardson \cite{littlewood1934group} and in full generality by Luzstig \cite{lusztig2010introduction}, Kashiwara \cite{Kashiwara1994CrystalGF} and Littelmann \cite{littelmann1994littlewood}. The case of conjugacy classes is still open in full generality, and the goal of the present paper is to address this problem in the case of $U_n$. Conjugacy classes of $U_n$ are indexed by the symmetrized torus $\mathcal{H}=[0,1[^n/S_n$, and the decomposition of the product of two conjugacy classes $\alpha$ and $\beta$ is described by a probability distribution $\mathbb{P}[\cdot\vert\alpha,\beta]$ on $\mathcal{H}$. As for any compact Lie group, $\mathbb{P}[\cdot\vert\alpha,\beta]$ can be expressed in a weak sense as a complex weighted sum of characters, each of which being seen as a function on $\mathcal{H}$. When the conjugacy classes $\alpha$ and $\beta$ have maximal dimension, in which case $\alpha$ and $\beta$ are called regular, $\mathbb{P}[\cdot\vert\alpha,\beta]$ has a density $\diff \mathbb{P}[\cdot\vert\alpha,\beta]$ with respect to the Lebesgue measure on $\mathcal{H}$. The present paper gives then a positive formula for this density as a subtraction-free sum of volumes of some explicit polytopes.

The convolution of conjugacy classes of a compact Lie group $G$ is also intimately related to the moduli spaces of $G$-valued flat connections on punctured Riemannian surfaces. Moduli spaces of flat connections on Riemann surfaces are a central object in algebraic geometry \cite{NarasiSesha_1965} and mathematical physic \cite{Atiyah_bott_1983}, which appear in particular as the semiclassical limit of two dimensional Yang-Mills measures, \cite{Forman_1993,levy2003yang}. In \cite{Witten_quantum_gauge,witten1992two}, Witten first proposed a general expression of the volume of such moduli spaces as an infinite sum of characters on the corresponding Lie group $G$. This formula has then been then proved by Jeffrey and Kirwan \cite{jeffrey1998intersection} using localisation techniques from symplectic geometry. The involved infinite series have been later simplified \cite{meinrenken1999moduli} to an alternating sum of volumes of coadjoint orbits, see also \cite{dooley1993harmonic} for a similar result using the wrapping map. As it has been observed in \cite{witten1992two}, the study of moduli spaces of flat connections on punctured Riemann surfaces with prescribed holonomies can be reduced to the case of the three punctured sphere, see also \cite{meinrenken1999moduli}.  The formula obtained in this paper for the convolution of conjugacy classes directly translates into a simple and manifestly positive expression for the volume of the moduli space of flat connections on the three-punctured sphere for $G=SU_n$. This is up to our knowledge the first expression of those volumes as the volume of explicit polytopes.

Let us briefly explain the conceptual path leading to the previous formulas. Conjugacy classes on $U_n$ are naturally related through the exponential map to co-adjoint orbits on the underlying Lie algebra $\mathfrak{u}_n$ of Hermitian matrices, which are indexed by the quotient space $\mathbb{R}^n/S_n$. In the same way as the multiplicative structure yields a convolution product on conjugacy classes of the Lie group $U_n$, the additive structure yields a convolution product on co-adjoint orbits of $\mathfrak{u}_n$. For co-adjoint orbits of maximal dimension, a positive formula for the density of the convolution product has already been obtained following the work of Knutson and Tao, see \cite{Knutson_1999,Coquereaux_2018}. Besides their apparent beauty, having positive formulas for such convolution products offered new tools to tackle difficult probabilistic problems concerning invariant measures on the Lie algebras. For example, the positive formulas for the convolution of coadjoint orbits of $U_n$ lie at the heart of the recent results of Narayanan and Sheffield \cite{narayanan2024large} on large deviations for the spectrum of sums of conjugation invariant Hermitian random matrices.

The convolution structure on co-adjoint orbits can be seen as a semi-classical limit of the ring of characters of the Lie group \cite{Kirillov_2004}. Thanks to this relation, the formula for the positive density in the $\mathfrak{u}_n$ case is up to an explicit factor a limit of Littlewood-Richardson coefficients with growing partitions. Such a coefficient had been expressed by Berenstein and Zelevinsky \cite{berenstein2001tensor} as the number of integers points in a convex polytope whose boundaries depend on the involved partitions. A reformulation of this expression by Knutson and Tao turned this convex polytope into a simple convex body, called discrete hive. Using the discrete hive model, the semiclassical limit yields then a positive formula for the density of convolutions of regular co-adjoint orbits, which the consists of the volume of certain polytopes called continuous hives.   

When replacing co-adjoint orbits on $\mathfrak{u}_n$ by conjugacy classes of $U_n$, the convolution product has to be seen instead as a semi-classical limit of the fusion ring of the Lie algebra with growing level \cite{Witten_quantum_gauge} (see also \cite{Defosseux_2016} for a more probabilistic approach). For each given level, the fusion ring is a specialization of the ring of characters at roots of unity and the structure coefficients  of the multiplication of characters in this quotient ring are non-negative integers called fusion coefficients. There exist no general effective combinatorial expression for those coefficients, but in the case of $U_n$ this fusion ring is isomorphic to the quantum cohomology ring of Grassmannians on $\mathbb{P}^1$ (see \cite{witten1993verlinde} for a geometric explanation of this fact). Through this isomorphism, characters of the fusion ring are Schubert classes of the quantum cohomology. 

Using a reinterpretation of the quantum cohomology of Grassmannians in terms of complex bundles on $\mathbb{P}^1$ \cite{Buch2003_solo}, Buch, Kresch and Tamvakis \cite{Buch_2003} related the coefficients appearing in the multiplication of Schubert classes, called quantum Littelwood-Richardson coefficients, to the structure coefficients of the cohomology ring of the two-step flag variety. Starting from this relation and proving a conjectural formula of Knutson on the two-step flag variety, those authors and Purbhoo \cite{puzzle_conj_two_step} gave a positive formula for the quantum Littlewood-Richardson coefficients in terms of certain puzzles (see also \cite{buch2015mutations} for an equivariant version). Such puzzles are generalizations of puzzles that already appeared in the work of Knutson, Tao and Woodward \cite{knutson2004honeycomb} in a reformulation of the hive model. The puzzle approach has since then been systematically used to obtain combinatorial expressions of structure coefficients either in the cohomology or in the K-theory of some partial flag manifolds, see \cite{knutson2017schubert}.

In this paper, we translate the program achieved in the co-adjoint case to the conjugacy one. One of the main difficulties in the present framework is the absence of convex formulations for the quantum Littlewood-Richardson coefficients. Based on the puzzle formulation of \cite{puzzle_conj_two_step}, we first express the structure constants of the two-step flag manifold as the counting of integer points in a finite union of convex polytopes indexed by a smaller tiling model, see Theorem \ref{thm:bijection_puzzle_dual_hive}. In order to operate an asymptotic counting, we then turn those convex polytopes into convex bodies, whose boundaries are not fully explicit. Fortunately, the asymptotic integer counting then simplifies the boundaries leading to a formula for the convolution of two regular conjugacy classes as a finite sum of volumes of explicit convex bodies, see Theorem \ref{thm:main}.

Before this manuscript, several results have been achieved in the description on the convolution of conjugacy classes of $U_n$. One of the most important results, independently obtained by Belkale \cite{belkale2008quantum} and Agnihotri and Woodward \cite{agnihotri1998eigenvalues}, is the description of the support of the convolution of two non necessarily regular conjugacy classes (see also \cite{teleman2003parabolic}). This support happens to be a convex set with boundaries described by certain quantum Littlewood-Richardson coefficients, in the same vein as the solution to the Horn problem given by Klyachko \cite{klyachko1998stable} and Knutson and Tao \cite{Knutson_1999}. The regularity of the convolution of two conjugacy classes has also been investigated, with for example results of \cite{wright2011sums} showing that the convolution of two conjugagy classes has an $L^2$ density when the image is an open set in $\mathcal{H}$.

\subsection*{Acknowledgement}
Both authors are supported by the Agence Nationale de la Recherche funding CORTIPOM ANR-21-CE40-0019. We thank David Garcia-Zelada, C\'{e}dric Lecouvey and Thierry L\'{e}vy for numerous fruitful discussions surrounding this work.
\section{Notations and statement of the result}\label{Sec:Notation_statement}

Fix $n\geq 3$ throughout this manuscript (the case $n=2$ can be handled by direct computation) and denote by $U_n$ the unitary group of size $n$. Then, recall that the set of conjugacy classes of $U_n$ is homeomorphic to the quotient space $\mathcal{H}=[0,1[^n/S_n$, where the symmetric group $S_n$ acts on $[0,1[^n$ by permutation of the coordinates. This quotient space is described by the set of non-increasing sequences of $[0,1[^n$. For $\theta=(\theta_1\geq \theta_2\geq \dots\geq\theta_n) \in \mathcal{H}$, denote by $\mathcal{O}(\theta)$ the corresponding conjugacy class defined by
$$\mathcal{O}(\theta):=\left\lbrace Ue^{2i\pi\theta}U^*, \,U\in U_n\right\rbrace\text{, where } e^{2i\pi\theta}=\begin{pmatrix} e^{2i\pi \theta_1}&0&\dots&\\0&e^{2i\pi\theta_2}&&\\\vdots&&\ddots&\\&&&e^{2i\pi\theta_n}\end{pmatrix}.$$
The product structure on $U_n$ translates into a convolution product $\ast:\mathcal{M}_1(\mathcal{H})\times \mathcal{M}_1(\mathcal{H})\rightarrow \mathcal{M}_1(\mathcal{H})$ on the space of probability distributions on $\mathcal{H}$ such that for $\theta,\theta'\in \mathcal{H}$, $\delta_{\theta}\ast \delta_{\theta'}$ is the distribution of $p(U_{\theta}U_{\theta'})$, where $U_{\theta}$ (resp. $U_{\theta'}$) is sampled uniformly on $\mathcal{O}(\theta)$ (resp. $\mathcal{O}(\theta')$) and $p:U_n\rightarrow \mathcal{H}$ maps an element of $U_n$ to its conjugacy class in $\mathcal{H}$.  \\

\noindent
Let us denote by $\mathcal{H}_{reg}=\{\theta\in \mathcal{H},\theta_1>\theta_2>\ldots>\theta_n\}$ the set of regular conjugacy classes of $U_n$, namely the ones of maximal dimension in $U_n$. For $\alpha,\beta\in \mathcal{H}_{reg}$, $\delta_{\alpha}\ast \delta_{\beta}$ admits a density $\diff \mathbb{P}[\cdot\vert\alpha,\beta]$ with respect to the Lebesgue measure on $\{\gamma\in\mathcal{H},\sum_{i=1}^n\alpha_i+\sum_{i=1}^n\beta_i-\sum_{i=1}^n\gamma_i \in \mathbb{N}\}$ (see Section \ref{Sec:link_quantum_cohomolgy} for a concrete proof of this classical result).

\subsection*{The toric hive cones \texorpdfstring{$\mathcal{C}_g$}{Lg} } The main result of the present manuscript is a positive formula for $\diff \mathbb{P}[\cdot\vert\alpha,\beta]$ in terms of the volume of polytopes similar to the hive model of Knutson and Tao \cite{Knutson_1999}. For $0\leq d\leq n$, define the toric hive $R_{d,n}$ as the set $$R_{d,n}=\left\{(v_1,v_2)\in \llbracket 0,n\rrbracket^2, d\leq v_1+v_2\leq n+d\right\},$$
which can be represented as a discrete hexagon through the map $(v_1,v_2)\mapsto v_1+v_2e^{i\pi/3}$, see Figure \ref{Fig:R_dn} for a particular case and its hexagonal representation.

\begin{figure}[H]
\begin{tikzpicture}
\node (P10) at (1,0) {$\bullet$};
\node (P20) at (2,0) {$\bullet$};
\node (P30) at (3,0) {$\bullet$};

\node (P01) at (0+1/2,1*0.866) {$\bullet$};
\node (P11) at (1+1/2,1*0.866) {$\bullet$};
\node (P21) at (2+1/2,1*0.866) {$\bullet$};
\node (P31) at (3+1/2,1*0.866) {$\bullet$};

\node (P02) at (0+2/2,2*0.866) {$\bullet$};
\node (P12) at (1+2/2,2*0.866) {$\bullet$};
\node (P22) at (2+2/2,2*0.866) {$\bullet$};

\node (P03) at (0+3/2,3*0.866) {$\bullet$};
\node (P13) at (1+3/2,3*0.866) {$\bullet$};
\end{tikzpicture}
\caption{The set $R_{1,3}$ represented through the map $(v_1,v_2)\mapsto v_1+v_2e^{i\pi/3}$. \label{Fig:R_dn}}
\end{figure}
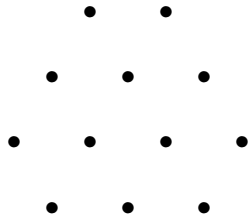

\subsubsection*{Boundary of the toric hive}
For any set $S$ and any function $f:R_{d,n}\rightarrow S$, we denote by $f^A$ (resp $f^B$, $f^C$) the vector $(f((d-i)\vee 0, (n+d-i) \wedge n)_{0\leq i\leq n}$ (resp. $(f(n+d-i \wedge n,i))_{0 \leq i \leq n}$, resp. $(f(n-i,i+d-n\vee 0))_{0\leq i\leq n}$). The vectors $f^A,f^B$ and $f^C$ correspond respectively to the north-west, east and south-west boundaries of $R_{d,n}$ through the hexagonal representation, see Figure \ref{Fig:boundary_toric_hive}.
 \begin{figure}[H]
\scalebox{0.5}{\begin{tikzpicture}
\node at (1,0) {$\bullet$};
\node at (2.7,0) {$\circ$};
\node at (3,0) {$\circ$};
\node at (3.3,0) {$\circ$};
\node at (5,0) {$\bullet$};
\node at (6,0) {$\bullet$};

\node at (6+1*1/2,1*0.866) {$\bullet$};
\node at (6+2.7*1/2,2.7*0.866) {$\circ$};
\node at (6+3*1/2,3*0.866) {$\circ$};
\node at (6+3.3*1/2,3.3*0.866) {$\circ$};
\node at (6+5*1/2,5*0.866) {$\bullet$};
\node at (6+6*1/2,6*0.866) {$\bullet$};

\node at (5+7*1/2,7*0.866) {$\bullet$};
\node at (3.3+8.7*1/2,8.7*0.866) {$\circ$};
\node at (3+9*1/2,9*0.866) {$\circ$};
\node at (2.7+9.3*1/2,9.3*0.866) {$\circ$};
\node at (1+11*1/2,11*0.866) {$\bullet$};
\node at (0+12*1/2,12*0.866) {$\bullet$};

\node at (-1+12*1/2,12*0.866) {$\bullet$};
\node at (-2.7+12*1/2,12*0.866) {$\circ$};
\node at (-3+12*1/2,12*0.866) {$\circ$};
\node at (-3.3+12*1/2,12*0.866) {$\circ$};
\node at (-5+12*1/2,12*0.866) {$\bullet$};
\node at (-6+12*1/2,12*0.866) {$\bullet$};

\node at (-6+11*1/2,11*0.866) {$\bullet$};
\node at (-6+9.3*1/2,9.3*0.866) {$\circ$};
\node at (-6+9*1/2,9*0.866) {$\circ$};
\node at (-6+8.7*1/2,8.7*0.866) {$\circ$};
\node at (-6+7*1/2,7*0.866) {$\bullet$};
\node at (-6+6*1/2,6*0.866) {$\bullet$};

\node at (-5+5*1/2,5*0.866) {$\bullet$};
\node at (-3.3+3.3*1/2,3.3*0.866) {$\circ$};
\node at (-3+3*1/2,3*0.866) {$\circ$};
\node at (-2.7+2.7*1/2,2.7*0.866) {$\circ$};
\node at (-1+1*1/2,1*0.866) {$\bullet$};
\node at (0+0*1/2,0*0.866) {$\bullet$};

\node at (1,0-0.4) {$f^C_{n-d-1}$};
\node at (5,0-0.4) {$f^C_1$};
\node at (6+0.4,0-0.4) {$f^C_0=f^B_0$};

\node at (0.4+6+1*1/2,1*0.866) {$f^B_1$};
\node at (0.5+6+5*1/2,5*0.866) {$f^B_{d-1}$};
\node at (0.4+6+6*1/2,6*0.866) {$f^B_d$};

\node at (0.6+5+7*1/2,7*0.866) {$f^B_{d+1}$};
\node at (0.6+1+11*1/2,11*0.866) {$f^B_{n-1}$};
\node at (0.6+0+12*1/2,0.5+12*0.866) {$f^B_n=f^A_0$};

\node at (-1+12*1/2,0.4+12*0.866) {$f^A_1$};
\node at (-5+12*1/2,0.4+12*0.866) {$f^A_{d-1}$};
\node at (-0.4+-6+12*1/2,0.4+12*0.866) {$f^A_{d}$};

\node at (-0.6+-6+11*1/2,11*0.866) {$f^A_{d+1}$};
\node at (-0.6+-6+7*1/2,7*0.866) {$f^A_{n-1}$};
\node at (-1+-6+6*1/2,6*0.866) {$f^A_{n}=f^C_{n}$};

\node at (-0.6+-5+5*1/2,5*0.866) {$f^C_{n-1}$};
\node at (-0.8+-1+1*1/2,1*0.866) {$f^C_{n-d+1}$};
\node at (-0.4+0*1/2,-0.4+0*0.866) {$f^C_{n-d}$};

\end{tikzpicture}}
\caption{The set boundary vectors $f^{A}$, $f^{B}$ and $f^{C}$. \label{Fig:boundary_toric_hive}}
\end{figure}
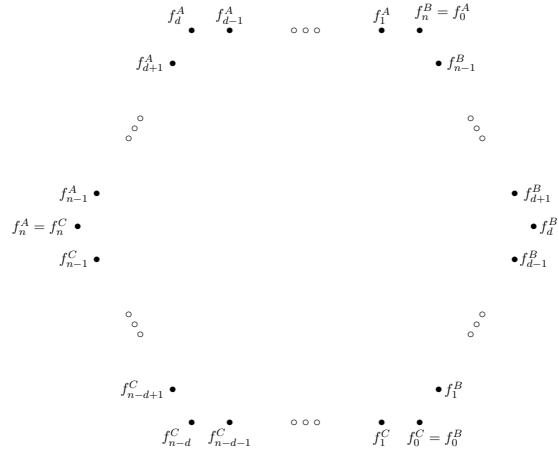

\subsubsection*{Toric rhombus concavity}Let us call a lozenge of $R_{d,n}$ any sequence $(v^1,v^2,v^3,v^4)\in (R_{d,n})^4$ corresponding to one of the three configurations of Figure \ref{Fig:possible_lozenges} in the hexagonal representation (in which $\vert v^i-v^{i+1}\vert=1$ for $1\leq i\leq 3$). 

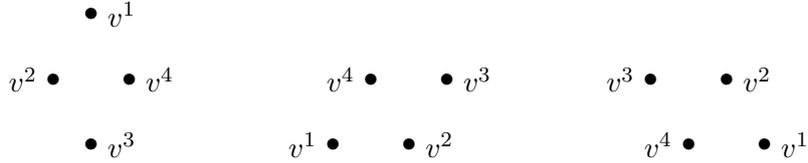
\begin{figure}[h!]
\begin{tikzpicture}
\node (P10) at (1,0) {$\bullet$};

\node (P01) at (0+1/2,1*0.866) {$\bullet$};
\node (P11) at (1+1/2,1*0.866) {$\bullet$};

\node (P02) at (0+2/2,2*0.866) {$\bullet$};

\node (P10) at (0.4+1,0) {$v^3$};

\node (P01) at (-0.4+0+1/2,1*0.866) {$v^2$};
\node(P11) at (0.4+1+1/2,1*0.866) {$v^4$};

\node (P02) at (0.4+0+2/2,2*0.866) {$v^1$};
\end{tikzpicture}
\hspace{1cm}
\begin{tikzpicture}
\node (P10) at (1,0) {$\bullet$};
\node (P20) at (2,0) {$\bullet$};

\node (P11) at (1+1/2,1*0.866) {$\bullet$};
\node (P21) at (2+1/2,1*0.866) {$\bullet$};
\node (P10) at (-0.4+1,0) {$v^1$};
\node (P20) at (0.4+2,0) {$v^2$};

\node (P11) at (-0.4+1+1/2,1*0.866) {$v^4$};
\node(P21) at (0.4+2+1/2,1*0.866) {$v^3$};
\end{tikzpicture}
\hspace{1cm}
\begin{tikzpicture}
\node (P10) at (1,0) {$\bullet$};
\node (P20) at (2,0) {$\bullet$};

\node (P01) at (0+1/2,1*0.866) {$\bullet$};
\node (P11) at (1+1/2,1*0.866) {$\bullet$};

\node (P10) at (-0.4+1,0) {$v^4$};
\node (P20) at (0.4+2,0) {$v^1$};

\node (P01) at (-0.4+0+1/2,1*0.866) {$v^3$};
\node (P11) at (0.4+1+1/2,1*0.866) {$v^2$};
\end{tikzpicture}
\caption{The three possible lozenges $(v^1,v^2,v^3,v^4)$ (beware of the position of the vertices which can not be permuted). \label{Fig:possible_lozenges}}
\end{figure}
\begin{definition}[Regular labeling]
\label{def:regular_labeling}
A labeling $g:R_{d,n}\rightarrow \mathbb{Z}_3$ is called regular whenever 
\begin{itemize}
\item $g^A_i=n+i[3]$, $g^B_i=i[3]$ and $g^C_i=i[3]$,
\item on any lozenge $\ell=(v^1,v^2,v^3,v^4)$, $$(g(v^2)=g(v^4))\Rightarrow \{g(v^1),g(v^3)\}=\{g(v^2)+1,g(v^2)+2\}.$$
\end{itemize}
A lozenge $(v^1,v^2,v^3,v^4)$ for which $(g(v^1), g(v^2), g(v^3), g(v^4)) =(a,a+1,a+2,a+1)$ for some $a \in \{0,1,2\}$ is called \textit{rigid}. \\
\noindent
The \textit{support} of a regular labeling $g:R_{d,n}\rightarrow \mathbb{Z}_3$ is the subset $Supp(g)\subset R_{d,n}$ of vertices of $R_{d,n}$ which are not a vertex $v_4$ of a rigid lozenge $(v^1,v^2,v^3,v^4)$.
\end{definition}
By the boundary condition of a regular labeling, any vertex $v_4$ of a rigid lozenge of $g$ can not be on the boundary of $R_{d,n}$, so that the latter is always contained in $Supp(g)$. 

\begin{remark}Although given above in a compact form, there may be better ways of considering  a regular labeling for growing $n$. By seeing $R_{d,n}$ through its hexagonal representation, a regular embedding is equivalent to a tiling of $R_{d,n}$ with either blue or red equilateral triangles of size $1$ or lozenges of size $1$ with alternating colors on its boundaries, such that the six boundary edges of $R_{d,n}$ are alternatively colored red and blue, starting with the color red on the south edge $\{(v_1,v_2)\in R_{d,n}, \, v_2=0\}$. The bijection from the former representation to the latter is given by assigning the red (resp. blue) color to any edge of the form $(v,v+e^{2i\pi\ell}), \, 0\leq \ell\leq 2$ along which the labels of $g$ increase by $1$ (resp. decrease by $1$),  see Figure \ref{Fig:label_color} for an example with $n=4, d=1$ and Proposition \ref{prop:regular_label_color} for a proof of this fact.
\end{remark}
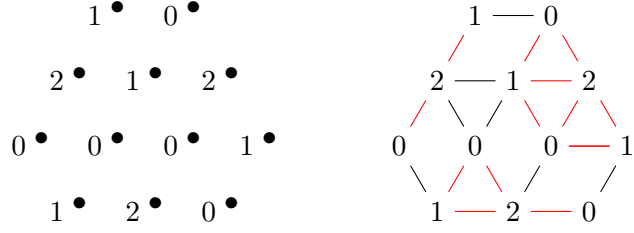
\begin{figure}[h!]
\begin{tikzpicture}
\node (P10) at (1,0) {$\bullet$};
\node (P20) at (2,0) {$\bullet$};
\node (P30) at (3,0) {$\bullet$};

\node (P01) at (0+1/2,1*0.866) {$\bullet$};
\node (P11) at (1+1/2,1*0.866) {$\bullet$};
\node (P21) at (2+1/2,1*0.866) {$\bullet$};
\node (P31) at (3+1/2,1*0.866) {$\bullet$};

\node (P02) at (0+2/2,2*0.866) {$\bullet$};
\node (P12) at (1+2/2,2*0.866) {$\bullet$};
\node (P22) at (2+2/2,2*0.866) {$\bullet$};

\node (P03) at (0+3/2,3*0.866) {$\bullet$};
\node (P13) at (1+3/2,3*0.866) {$\bullet$};

\node (P10) at (-0.3+1,-0.1+0) {$1$};
\node (P20) at (-0.3+2,-0.1+0) {$2$};
\node (P30) at (-0.3+3,-0.1+0) {$0$};

\node (P01) at (-0.3+0+1/2,-0.1+1*0.866) {$0$};
\node (P11) at (-0.3+1+1/2,-0.1+1*0.866) {$0$};
\node (P21) at (-0.3+2+1/2,-0.1+1*0.866) {$0$};
\node (P31) at (-0.3+3+1/2,-0.1+1*0.866) {$1$};

\node (P02) at (-0.3+0+2/2,-0.1+2*0.866) {$2$};
\node (P12) at (-0.3+1+2/2,-0.1+2*0.866) {$1$};
\node (P22) at (-0.3+2+2/2,-0.1+2*0.866) {$2$};

\node (P03) at (-0.3+0+3/2,-0.1+3*0.866) {$1$};
\node (P13) at (-0.3+1+3/2,-0.1+3*0.866) {$0$};
\end{tikzpicture}
\hspace{1cm}
\begin{tikzpicture}
\node (P10) at (1,0) {$1$};
\node (P20) at (2,0) {$2$};
\node (P30) at (3,0) {$0$};

\node (P01) at (0+1/2,1*0.866) {$0$};
\node (P11) at (1+1/2,1*0.866) {$0$};
\node (P21) at (2+1/2,1*0.866) {$0$};
\node (P31) at (3+1/2,1*0.866) {$1$};

\node (P02) at (0+2/2,2*0.866) {$2$};
\node (P12) at (1+2/2,2*0.866) {$1$};
\node (P22) at (2+2/2,2*0.866) {$2$};

\node (P03) at (0+3/2,3*0.866) {$1$};
\node (P13) at (1+3/2,3*0.866) {$0$};

\draw [color=red] (P10)--(P20);
\draw [color=red] (P20)--(P30);
\draw [color=red] (P10)--(P11);
\draw [color=red] (P11)--(P20);
\draw [color=red] (P21)--(P31);
\draw [color=red] (P01)--(P02);
\draw [color=red] (P21)--(P12);
\draw [color=red] (P21)--(P22);
\draw [color=red] (P21)--(P31);
\draw [color=red] (P31)--(P22);
\draw [color=red] (P02)--(P03);
\draw [color=red] (P12)--(P13);
\draw [color=red] (P12)--(P22);
\draw [color=red] (P22)--(P13);

\draw (P10)--(P01);
\draw (P30)--(P31);
\draw (P20)--(P21);
\draw (P11)--(P02);
\draw (P11)--(P12);
\draw (P02)--(P12);
\draw (P03)--(P13);

\end{tikzpicture}
\caption{A regular labeling on $R_{d,n}$ and its colored representation \label{Fig:label_color}}
\end{figure}
\begin{definition}[Toric hive cone]
A function $f:R_{d,n}$ is called rhombus concave with respect to a regular labeling $g:R_{d,n}\rightarrow \mathbb{Z}_3$ when $f(v_2)+f(v_4)\geq f(v_1)+f(v_3)$ on any lozenge $\ell=(v^1,v^2,v^3,v^4)$,  with equality if $\ell$ is rigid with respect to $g$.

For any regular labeling $g$, the toric hive cone $\mathcal{C}_{g}$ with respect to $g$ is the cone 
$$\mathcal{C}_g=\left\{f_{\vert Supp(g)},\text{$f:R_{d,n}\rightarrow  \mathbb{R}$ toric rhombus concave with respect to $g$}\right\}.$$
\end{definition}

As we will see later, for any regular labeling $g$, $Supp(g)$ has cardinal $(n-1)(n-2)/2+3n$ and so is the dimension of $\mathcal{C}_g$. As such, we recover the usual dimension of the classical hive cone from \cite{Knutson_1999}. The latter is then a particular case of toric hive cone for $d=0$. An example of a toric rhombus concave function in the case $n=3,d=1$ is given in Figure \ref{Fig:toric_hive}.
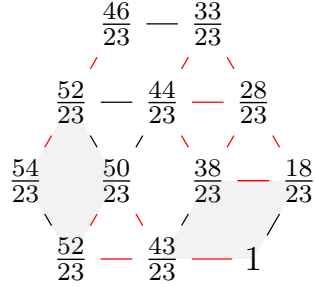
\begin{figure}[h!]
\scalebox{1.2}{\begin{tikzpicture}
\node (P10) at (1,0) {$\frac{52}{23}$};
\node (P20) at (2,0) {$\frac{43}{23}$};
\node (P30) at (3,0) {$1$};

\node (P01) at (0+1/2,1*0.866) {$\frac{54}{23}$};
\node (P11) at (1+1/2,1*0.866) {$\frac{50}{23}$};
\node (P21) at (2+1/2,1*0.866) {$\frac{38}{23}$};
\node (P31) at (3+1/2,1*0.866) {$\frac{18}{23}$};

\node (P02) at (0+2/2,2*0.866) {$\frac{52}{23}$};
\node (P12) at (1+2/2,2*0.866) {$\frac{44}{23}$};
\node (P22) at (2+2/2,2*0.866) {$\frac{28}{23}$};

\node (P03) at (0+3/2,3*0.866) {$\frac{46}{23}$};
\node (P13) at (1+3/2,3*0.866) {$\frac{33}{23}$};

\draw [color=red] (P10)--(P20);
\draw [color=red] (P20)--(P30);
\draw [color=red] (P10)--(P11);
\draw [color=red] (P11)--(P20);
\draw [color=red] (P21)--(P31);
\draw [color=red] (P01)--(P02);
\draw [color=red] (P21)--(P12);
\draw [color=red] (P21)--(P22);
\draw [color=red] (P21)--(P31);
\draw [color=red] (P31)--(P22);
\draw [color=red] (P02)--(P03);
\draw [color=red] (P12)--(P13);
\draw [color=red] (P12)--(P22);
\draw [color=red] (P22)--(P13);

\draw (P10)--(P01);
\draw (P30)--(P31);
\draw (P20)--(P21);
\draw (P11)--(P02);
\draw (P11)--(P12);
\draw (P02)--(P12);
\draw (P03)--(P13);

\fill[color=gray,opacity=0.1] (1,0)--(1/2,1*0.866)--(2/2,2*0.866)--(1+1/2,1*0.866)--cycle;
\fill[color=gray,opacity=0.1] (2,0)--(2+1/2,1*0.866)--(3+1/2,1*0.866)--(3,0)--cycle;
\end{tikzpicture}
}
\caption{A toric rhombus convave function for $n=3,d=1$\label{Fig:toric_hive}: the regular labeling is depicted through colored edges and the shaded lozenge are the rigid ones yielding the equality cases in the toric rhombus concavity.}
\end{figure}

\begin{definition}[Polytope $P^g_{\alpha,\beta,\gamma}$]
\label{def:p_g_alpha_beta_gamma}
    Let $n\geq 3$ and let $\alpha,\beta,\gamma\in \mathcal{H}_{reg}$ be such that $\sum_{i=1}^n\alpha_i+\sum_{i=1}^n\beta_i=\sum_{i=1}^n\gamma_i+d$ with $d\in\mathbb{N}$. Let $g$ be a regular labeling on $R_{d,n}$. Then, $P^g_{\alpha,\beta,\gamma}$ is the polytope of $\mathbb{R}^{Supp(g)\setminus \partial R_{d,n}}$ consisting of functions in $\mathcal{C}_{g}$ such that
    
    \begin{equation*}
    f^A=\left(\sum_{s=1}^n \beta_s+\sum_{s=1}^{i}\alpha_s\right)_{0\leq i\leq n},\,
f^B=\left((d-i)^{+}+\sum_{s=1}^i\beta_s\right)_{0\leq i\leq n},\,
f^C=\left(d+\sum_{s=1}^i\gamma_s\right)_{0\leq i\leq n}.
    \end{equation*}
 
\end{definition}

\noindent
An example of an element of $P^g_{\alpha,\beta,\gamma}$ for $n=3$ and $d=1$ is depicted in Figure \ref{Fig:toric_hive}, for $\alpha=\left(\frac{13}{23}\geq \frac{6}{23}\geq\frac{2}{23}\right)$, $\beta=\left(\frac{18}{23}\geq \frac{10}{23}\geq \frac{5}{23}\right)$ and $\gamma=\left(\frac{20}{23}\geq \frac{9}{23}\geq \frac{2}{23}\right)$.

\subsubsection*{Statement of the results:}
Our main result gives then a formula  for the density of the convolution of regular conjugacy classes as a sum of volumes of polytopes coming from $\mathcal{C}_g$ for regular labeling $g$.

\begin{theorem}[Probability density for product of conjugacy classes]
\label{thm:main}
Let $n\geq 3$ and let $\alpha,\beta,\gamma\in \mathcal{H}_{reg}$ be such that $\sum_{i=1}^n\alpha_i+\sum_{i=1}^n\beta_i=\sum_{i=1}^n\gamma_i+d$ with $d\in\mathbb{N}$. Then,
\begin{equation}
    \label{eq:main_result}
    \diff \mathbb{P}[\gamma\vert\alpha,\beta]=\frac{(2 \pi)^{(n-1)(n-2)/2}\prod_{k=1}^{n-1}k!\Delta(\ed^{2i \pi \gamma}) }{n!\Delta(\ed^{2i \pi \alpha} )\Delta(\ed^{2i \pi \beta})}\sum_{g:R_{d,n}\rightarrow \mathbb{Z}_3 \text{ regular}}Vol_g(P^g_{\alpha,\beta,\gamma}),
\end{equation}
%$$\diff \mathbb{P}[\gamma\vert\alpha,\beta]=\frac{\Delta(\ed^{2i \pi \gamma}) }{\Delta(\ed^{2i \pi \alpha} )\Delta(\ed^{2i \pi \beta})}\sum_{g:R_{d,n}\rightarrow \mathbb{Z}_3 \text{ regular}}Vol_g(P^g_{\alpha,\beta,\gamma}),$$
where $\Delta(e^{2i\pi\theta})=2^{n(n-1)/2}\prod_{i<j}\sin\left(\pi(\theta_i-\theta_j)\right)$ for $\theta\in\mathcal{H}$ and $Vol_g$ denotes the volume with respect to the Lebesgue measure on $\mathbb{R}^{Supp(g)\setminus \partial R_{d,n}}$.
\end{theorem}
\noindent
Note that the case $n=2$ admits explicit formulas which do not need such a machinery. Numerical experiments for $n=3$ suggest that there are $\alpha,\beta\in\mathcal{H}_{reg}$ for which any regular labeling $g$ yields a non-empty polytope $P^g_{\alpha,\beta,\gamma}$ for some $\gamma\in \mathcal{H}_{reg}$. However, for a fixed triple $(\alpha,\beta,\gamma)\in \mathcal{H}_{reg}$ there seems to be generically only a strict subset of $\{P^g_{\alpha,\beta,\gamma}\}_{g\text{ regular}}$ which are not empty and do contribute. Finally, remark that we only considered the case of regular conjugacy classes to ensure the existence of a density for the convolution product. Such a hypothesis is regularly assumed (see for example \cite{Witten_quantum_gauge, meinrenken1999moduli}). However, we expect similar results to hold in cases where less than $n/2$ coordinates of $\mathcal{H}$ are equal, in which case the convolution product is still expected to have a density.

\subsubsection*{Volume of moduli spaces of flat $SU_n$-connections on the sphere:} The main application of the previous result is the computation of the volume of moduli spaces of flat $SU_n$-connections on the three-punctured sphere. Computing such volume is an important task in the study of the Yang-Mills measure on Riemann surfaces in the small surface limit \cite{Forman_1993}, and it has been shown in \cite{witten1992two,meinrenken1999moduli} that this computation for arbitrary compact Riemann surfaces can be reduced to case of the three-punctured sphere by a sewing phenomenon. A similar inductive procedure is used in \cite{mirzakhani2007weil} to reduce the volume problem for the moduli space of curves to the genus zero case.

Denote by $\Sigma_0^3$ the sphere with three generic marked points $a,b,c$ removed. We then denote by $\mathcal{M}(\Sigma_0^3,\alpha,\beta,\gamma)$ the moduli space of flat $SU_n$-valued connections on $\Sigma_0^3$ for which the holonomies around $a,b,c$ respectively belong to $\mathcal{O}_{\alpha},\mathcal{O}_{\beta}$ and $\mathcal{O}_{\gamma}$. In the specific case of the punctured sphere, this moduli space can be alternatively described as 
$$\mathcal{M}(\Sigma_0^3,\alpha,\beta,\gamma)=\{(U_1,U_2,U_3)\in \mathcal{O}_{\alpha}\times\mathcal{O}_{\beta}\times \mathcal{O}_{\gamma},\, U_1U_2U_3=Id_{SU_n}\}/SU_n,$$
where $SU_n$ acts diagonally by conjugation, see \cite{agnihotri1998eigenvalues}, see also \cite[Section 3]{meinrenken1999moduli} for a general introduction to $SU_n$-valued flat connections and a proof of the latter equality. As a corollary of Theorem \ref{thm:main}, we thus get an expression of the volume of $\mathcal{M}(\Sigma_0^3,\alpha,\beta,\gamma)$ as a sum of volumes of explicit polytopes.

\begin{corollary}[Volume of flat $SU_n$-connections on the sphere]
\label{cor:volume_flat_connection}
Let $n\geq 3$ and consider the canonical volume form on $SU_n$. For $\alpha,\beta,\gamma\in \mathcal{H}_{reg}$  such that $\vert \alpha\vert_1,\,\vert \beta\vert_1,\,\vert \gamma\vert_1 \in \mathbb{N}$, then $Vol\left[(\mathcal{M}(\Sigma_0^3,\alpha,\beta,\gamma)\right]\not=0$ only if $\sum_{i=1}^n\alpha_i+\sum_{i=1}^n\beta_i+\sum_{i=1}^n\gamma_i=n+d$ for some $d\in\mathbb{N}$, in which case
$$Vol\left[(\mathcal{M}(\Sigma_0^3,\alpha,\beta,\gamma)\right]=\frac{2^{(n+1) [2]}(2\pi)^{(n-1)(n-2)}}{n!\Delta(\ed^{2i \pi \gamma})\Delta(\ed^{2i \pi \alpha} )\Delta(\ed^{2i \pi \beta})} \sum_{g:R_{d,n}\rightarrow \mathbb{Z}_3 \text{ regular}}Vol_g(P^g_{\alpha,\beta,\widetilde{\gamma}}),$$
where $\widetilde{\gamma}=(1-\gamma_n,\ldots,1-\gamma_1)$ and the polytopes $P^g_{\alpha,\beta,\widetilde{\gamma}}$ are defined in Theorem \ref{thm:main}.
\end{corollary}
Note that the choice of normalization for the volume of $SU_n$ slightly differs from the one used in \cite{Witten_quantum_gauge} for numerical applications. As a consequence of this corollary, the volume is a piecewise polynomial in $\alpha,\beta,\gamma$, up to the normalization factor coming from the volume of the conjugacy classes. Such a phenomenon, which is a reflect of the underlying symplectic structure, had already been observed in \cite{meinrenken1999moduli}. A same phenomenon occurs in the co-adjoint case, see \cite{Coquereaux_2018,etingof2018mittag} and in the study of moduli spaces of curves, \cite{mirzakhani2007weil}.

\subsubsection*{Sketch of the proof of Theorem \ref{thm:main} and Corollary \ref{cor:volume_flat_connection}}
Let us give the structure of the paper while sketching the proof of our main statement. The first step of the prooF is the semi-classical approximation of the density of the convolution product by a limit of quantum Littlewood-Richardson coefficients. Such approximation scheme is done in Section \ref{Sec:link_quantum_cohomolgy}. All the work of the remaining part of the manuscript consists in turning known expressions for the  quantum Littlewood-Richardson coefficients into integers points counting in convex bodies, for which the convergence towards volumes of polytopeS is straightforward, see \cite{Knutson_1999}. Section \ref{Sec:Puzzle_skeleton} introduces the puzzle expression for those coefficients obtained in \cite{puzzle_conj_two_step} and gives a first simplification of the puzzle formulation by only keeping the position of certain pieces of the puzzles. It is then deduced in Section \ref{Sec:discrete_two_colored_hive} an expression of the coefficients as the counting of integers points in a family of convex polytopes indexed by certain two-colored tilings which are reminiscent of Figure \ref{Fig:label_color}, see Theorem \ref{thm:puzzle_two_steps} and Corollary \ref{cor:qLR_dual_hive}. Up to this point, those polytopes are degenerated and non-rational polytopes in a higher dimensionAL space, which prevents any proper asymptotic counting in the semi-classical limit (a similar problematic situation already occurred with Berenstein-Zelevinsky polytopes in the co-adjoint case, leading to the hive formulation of Knutson and Tao, see \cite{Knutson_1999}). By a combinatorial work on the underlying two-colored tiling, we give in Section \ref{Sec:color_swap} a parametrization of the integer points of those polytopes in terms of integer points of genuine convex bodies. Remark that the results of Section \ref{Sec:discrete_two_colored_hive} hold more generally for any coefficient of the two-step flag variety, a fact which is not true anymore from Section \ref{sec:color_swap}. The asymptotic counting of integers points in convex bodies is then much more tractable, and the conclusion of the proof of Theorem \ref{thm:main} and Corollary \ref{cor:volume_flat_connection} is done in Section \ref{Sec:convergence}.

\section{Density formula via the quantum cohomology of the Grassmannians}
\label{Sec:link_quantum_cohomolgy}

\noindent
Let $n \geq 1$ and consider $\alpha, \beta \in \mathcal{H}_{reg}^2$ : $\alpha = (\alpha_1, \dots, \alpha_n)$, $\beta = (\beta_1, \dots, \beta_n)$ where
\begin{equation*}
    1 \geq \alpha_1 > \alpha_2 > \dots > \alpha_n \geq 0 \text{ and }  1 \geq \beta_1 > \beta_2 > \dots > \beta_n \geq 0.
\end{equation*}
Up to multiplication by the center of $U(n)$, suppose furthermore without loss of generality that 
\begin{equation}
\label{su_n_condition}
    \sum_{i=1}^n \alpha_i = k \text{ and } \sum_{i=1}^n \beta_i = k' 
\end{equation}
for some $k,k' \in \Z$. Let $A = U \ed^{2i \pi \alpha} U^*, B = V \ed^{2i \pi \beta} V^*$,  where $U, V$ are independent Haar distributed matrices on $U(n)$. Remark that $A$ and $B$ are respectively uniformly distributed on the conjugacy classes $\Ocal(\alpha)$ and $\Ocal(\beta)$, which lie in $SU(n)\subset U(n)$. The goal of this section, see Theorem \ref{th:density_limit_quantum_coef}, is to give a simple proof of the density formula \eqref{eq:density_limit_quantum_coefs} linking the probability $\diff \prob[\gamma | \alpha, \beta]$ that  $AB \in \Ocal(\gamma)$ for $\gamma \in \mathcal{H}$ to the structure constants of the quantum cohomology of Grassmannians defined in Section \ref{subsec:link_quantum_coho}. Such a semi-classical convergence had been already suggested and proven several times in different forms (see \cite{Witten_quantum_gauge} for a similar approach with fusion coefficient and \cite{Defosseux_2016} for a convergence in distribution).  In Section \ref{subsec:first_density}, we recall in Proposition \ref{prop:first_density} a classical expression of the density in terms of characters of irreducible representations of $SU(n)$. In Section \ref{subsec:link_quantum_coho}, we link the density of Proposition \ref{prop:first_density} to the structure constants and derive Theorem \ref{th:density_limit_quantum_coef}.

\subsection{A first density formula}
\label{subsec:first_density}

This part aims at recalling a proof of the formula \eqref{eq:density_horn_1} which gives the value of $\diff \prob[\gamma | \alpha, \beta]$ as an infinite sum of characters. A similar treatment of the convolution of orbit measures in the general context of Lie algebras can be found in \cite[Sec. 7]{Defosseux_2016}.\\
\\
Let us denote by $\diff g$ the \textit{normalized} Haar measure on $U(n)$ and for $\theta \in \mathcal{H}$, $\varphi_\theta$ the map 
\begin{align*}
    \varphi_\theta: U(n) &\rightarrow \Ocal(\theta) \subset SU(n) \\
    U &\mapsto U \ed^{2i \pi \theta} U^*.
\end{align*}

\noindent Let us write
\begin{equation}
    \label{eq:mu_theta}
m_\theta = \varphi_\theta \# \diff g
\end{equation}
for the push-forward of $\diff g$ by $ \varphi_\theta$. The measure $m_\theta$ is a measure on $\Ocal(\theta)$ called the orbital measure. For any function $f: \Ocal(\theta) \rightarrow \R$, 
\begin{equation*}
    \int_{\Ocal(\theta)} f \diff m_\theta  = \int_{U(n)} f(g \ed^{2i \pi \theta} g^{-1}) dg.
\end{equation*}

\noindent
Recall that the irreducible representations of the compact group $SU(n)$ are parameterized by $\lambda \in \Z_{\geq 0}^{n-1}$ and we denote by $(\rho_\lambda, V_\lambda)$ the corresponding representation where $\rho_\lambda: SU(n) \rightarrow V_\lambda$ and $\chi_\lambda: End_{V_\lambda} \rightarrow \C, x \mapsto Tr[x]$ is the associated character.

\begin{proposition}[Induced density of eigenvalues]
\label{prop:first_density}
Let $(\alpha, \beta) \in \mathcal{H}_{reg}^2$ and let $A, B \in \Ocal(\alpha) \times \Ocal(\beta)$ be two independent random variables sampled from $m_\alpha$ and $m_\beta$ respectively. Let $C = AB \in \Ocal(\gamma)$ for some random $\gamma \in [0, 1[^n / S_n$. The density of  $\gamma = \gamma_1 > \dots > \gamma_n \geq 0$ is given by the absolute convergent series

\begin{equation}
\label{eq:density_horn_1}
    \diff \prob[\gamma | \alpha, \beta] = \frac{|\Delta(\ed^{2i \pi \gamma})|^2}{(2 \pi)^{n-1}n!} \sum_{\lambda \in \Z_{\geq 0}^{n-1}} \frac{1}{\dim V_\lambda} \chi_\lambda(\ed^{2i \pi \alpha}) \chi_\lambda(\ed^{2i \pi \beta}) \chi_\lambda(\ed^{-2i \pi \gamma}).
\end{equation}
\end{proposition}

\noindent
Another expression of the density \eqref{eq:density_horn_1} is given in \eqref{eq:density_horn_2}. The rest of this section is devoted to the proof of Proposition \ref{prop:first_density}.

\begin{definition}[Fourier Transform on $SU(n)$]
    Let $m$ be a measure on $SU(n)$. The Fourier transform $ \widehat{m}$ of $m$ is defined as
    \begin{equation}
        \widehat{m}: \lambda \in \Z_{\geq 0}^{n-1} \mapsto \widehat{m}(\lambda) =  \int_{SU(n)} \rho_\lambda(g) \diff m(g) \in \mathrm{End}_{V_\lambda}.
        \label{eq:def_fourier_theta}
    \end{equation}
    
\end{definition}

\noindent
In the case where $m = m_\theta$, the expression $\widehat{m_\theta}(\lambda)$ is also known as the spherical transform introduced in \cite[eq. (56)]{Zhang_2021}.

\begin{lemma}[Fourier Transform of $m_\theta$]
    One has, for $\lambda \in \Z_{\geq 0}^{n-1}$,
    \begin{equation}
    \label{eq:fourier_theta}
        \widehat{m_\theta}(\lambda) = \frac{\chi_\lambda(\ed^{2i \pi \theta})}{\dim V_\lambda} \mathrm{id}_{V_\lambda},
    \end{equation}
    where $\mathrm{id}_{V_\lambda}$ is the identity element of $V_\lambda$.
\end{lemma}

\begin{proof}
    For any $\lambda \in \Z_{\geq 0}^{n-1}$ and $g\in SU(n)$, since the Haar measure is invariant by translation,
    \begin{align*}
        \widehat{m_\theta}(\lambda)\rho_\lambda(g)=  \left(\int_{U(n)} \rho_\lambda(h \ed^{2i \pi \theta} h^{-1}) dh\right) \rho_\lambda(g)=&\int_{U(n)} \rho_\lambda(h \ed^{2i \pi \theta} h^{-1}g) dh \\
        =&\rho_\lambda(g)\int_{U(n)} \rho_\lambda(g^{-1}h \ed^{2i \pi \theta} h^{-1}g) dh\\
        =&\rho_{\lambda}(g)\int_{U(n)} \rho_\lambda(h \ed^{2i \pi \theta} h^{-1}) dh=\rho_\lambda(g)\widehat{m_\theta}(\lambda).
    \end{align*}
    Hence, $\widehat{m_\theta}(\lambda)$ is a morphism of the irreducible representation $\rho_\lambda$ and thus $\widehat{m_\theta}(\lambda) = c \cdot \mathrm{id}_{V_\lambda}$ for some $c\in \mathbb{C}$. One computes the value of $c$ by taking the trace which gives 
    \begin{equation}
        c = \frac{\chi_\lambda(\ed^{2i \pi \theta})}{\dim V_\lambda}.
    \end{equation}
\end{proof}

\begin{definition}[Convolution of measures]
\label{def:convolution_of_measures}
    Let $m, m' $ be two measures on $SU(n)$. Let $m \boxtimes m' $ be the product measure on $SU(n) \times SU(n)$. Define $mult: SU(n) \times SU(n) \rightarrow SU(n)$ to be the multiplication on $SU(n)$ : $mult(g_1, g_2) = g_1g_2$.
    The convolution of $m$ and $m'$, denoted by $m * m'$, is defined as 
    \begin{align}
        m * m' = mult \# (m \boxtimes m'),
    \end{align}
    which means that for any function $f$ on $SU(n)$,
    \begin{equation*}
        \int_{SU(n)} f(g) \diff (m * m')(g) = \int_{SU(n)} \int_{SU(n)} f(g_1g_2) \diff m(g_1) \diff m(g_2).
    \end{equation*}
    For $(\alpha, \beta) \in \mathcal{H}^2$, we write $m_{\alpha, \beta} := m_\alpha * m_\beta $ the convolution of $m_\alpha$ and $m_\beta$.
\end{definition}

\noindent
By Definition \ref{def:convolution_of_measures}, the measure $m_{\alpha, \beta}$ is the law on $SU(n)$ of $C = A \cdot B$ where $A$ and $B$ are sampled from measures $\mu_\alpha$ and $\mu_\beta $ on $\Ocal(\alpha)$ and $\Ocal(\beta)$ respectively. Recall that for two measures $m, m' $ on $SU(n)$,
    \begin{equation}
        \widehat{m*m'}(\lambda) = \widehat{m}(\lambda) \widehat{m}'(\lambda).
    \end{equation}
    In particular 
    \begin{equation}
    \label{eq:prod_fourier}
        \widehat{m}_{\alpha, \beta}(\lambda) = \widehat{m}_\alpha(\lambda) \widehat{m}_\beta(\lambda).
    \end{equation}

\noindent
Recall that we are interested in the measure $\mu_{\alpha, \beta}$. By \eqref{eq:prod_fourier}, one knows how to compute its Fourier transform. Let us define the inverse Fourier transform. 

\begin{definition}[Inverse Fourier Transform]

Let $f: \lambda \in \Z_{\geq 0}^{n-1} \mapsto f(\lambda) \in End_{V_\lambda}$ be a function such that 

\begin{equation}
\label{eq:condition_L2}
    \norm{f}^2 = \sum_{\lambda \in \Z_{\geq 0}^{n-1}} \dim V_\lambda \cdot Tr[ f(\lambda)f^*(\lambda) ] < \infty.
\end{equation}

\noindent
The inverse Fourier transform of $f$ is 
\begin{align}
    f^\lor: SU(n) &\rightarrow \C \\
    g & \mapsto \sum_{\lambda \in \Z_{\geq 0}^{n-1}} \dim V_\lambda \cdot Tr[ \rho_\lambda(g^{-1})f(\lambda) ].
\label{eq:inv_fourier}
\end{align}
    
\end{definition}

\noindent
In order to apply inverse Fourier transform to $\widehat{m}_{\alpha, \beta}$, one needs to check condition \eqref{eq:condition_L2}. This is the purpose of the next lemma.

\begin{lemma}[Product Fourier transform is $L^2$]
\label{lem:convolution_is_l2}
    For $(\alpha, \beta) \in \mathcal{H}_{reg}^2$,
    \begin{equation}
        \sum_{\lambda \in \Z_{\geq 0}^{n-1}} \dim (V_\lambda) Tr[\widehat{m}_{\alpha, \beta}(\lambda)\widehat{m}_{\alpha, \beta}(\lambda)^*] < \infty.
    \end{equation}
\end{lemma}

\begin{proof}
Using \eqref{eq:prod_fourier} together with \eqref{eq:fourier_theta}, one has
    \begin{equation}
        \sum_{\lambda \in \Z_{\geq 0}^{n-1}} \dim (V_\lambda )Tr[\widehat{m}_{\alpha, \beta}(\lambda)\widehat{m}_{\alpha, \beta}(\lambda)^*] = \sum_{\lambda \in \Z_{\geq 0}^{n-1}} \frac{1}{\dim(V_\lambda)^2} \left| \chi_\lambda(\ed^{2i \pi \alpha}) \right|^2 \left| \chi_\lambda(\ed^{2i \pi \beta}) \right|^2.
    \end{equation}
Using Weyl's character formula \cite[eq. (21)]{Coquereaux_2018},

\begin{equation}
\label{eq:weyl_character_formula}
    \chi_\lambda(\ed^{i\theta}) = \frac{\det [\ed^{i \theta_r \lambda_s'}]_{1 \leq r,s \leq n}}{\Delta(\ed^{i \theta})}
\end{equation}
where $\lambda' = (\lambda_1,\ldots,\lambda_{n-1},0) + \rho$ with $\rho=(n-1,\ldots,0)$ and where $\Delta(x_1, \dots, x_n) = \prod_{1 \leq i < j \leq n} (x_i - x_j)$ is the Vandermonde determinant. Recall that by assumption, $\alpha_i \neq \alpha_j$ for $i \neq j$ and the same holds for $\beta$, so that the expressions $\Delta(\ed^{2i \pi \alpha})$ and $\Delta(\ed^{2i \pi \beta})$ are well-defined. The previous sum becomes
    \begin{align}
        & \frac{1}{|\Delta(\ed^{2i \pi \alpha})\Delta(\ed^{2i \pi \beta})|^2} \sum_{\lambda_1 \geq \dots \geq \lambda_{n-1} \geq \lambda_n = 0 } \frac{ \left|\det [\ed^{2i \pi \alpha_r \lambda_s'}] \right|^2 \left|\det [\ed^{2i \pi \beta_r \lambda_s'}] \right|^2}{\dim (V_\lambda)^2} \\
        &\leq \frac{ n^{2n}}{|\Delta(\ed^{2i \pi \alpha})\Delta(\ed^{2i \pi \beta})|^2} \sum_{\lambda_1 \geq \dots \geq \lambda_{n-1} \geq \lambda_n = 0 } \frac{1}{\dim (V_\lambda)^2} 
    \end{align}
where we used Hadamard's inequality for the upper bound on determinants. From the identity 
\begin{equation}
    \dim (V_\lambda) = \frac{\Delta(\lambda')}{sf(n-1)},
\end{equation}
 which can be found in \cite[Cor. 11.2.5]{Faraut_2008} and where $sf(n) = \prod_{1 \leq j \leq n} j!$, it suffices to show that

 \begin{equation}
 \label{eq:conv_inv_vander}
     V_n = \sum_{\lambda_1 > \dots > \lambda_{n}= 0} \frac{1}{\Delta(\lambda)^2}
 \end{equation}
converges for $n \geq 2$. One has that $V_2 = \sum_{k \geq 1} k^{-2} < \infty$. Let us write 
\begin{align*}
     \sum_{\lambda_1 > \dots > \lambda_{n} = 0} \frac{1}{\Delta(\lambda)^2} &=  \sum_{\lambda_1 > \dots > \lambda_{n} = 0} \prod_{1 \leq i<j \leq n}(\lambda_i - \lambda_j)^{-2} \\
     &= \sum_{\lambda_2 > \dots > \lambda_{n} = 0} \left[ \sum_{\lambda_1 > \lambda_2} \prod_{2 \leq j \leq n}(\lambda_1 - \lambda_j)^{-2}  \right] \prod_{2 \leq i<j \leq n}(\lambda_i - \lambda_j)^{-2} \\
     &\leq \sum_{\lambda_2 > \dots > \lambda_{n} = 0} \left[ \sum_{\lambda_1 > \lambda_2} (\lambda_1 - \lambda_2)^{-2(n-1)}  \right] \prod_{2 \leq i<j \leq n}(\lambda_i - \lambda_j)^{-2}
\end{align*}
the innermost sum is bounded by $\sum_{k \geq 1} k^{-2(n-1)} \leq \sum_{k \geq 1} k^{-2}  = V_2$ for $n \geq 2$. Thus,
\begin{equation*}
    V_n \leq V_2 V_{n-1}
\end{equation*}
so that for $n \geq 2$, $V_n \leq (V_2)^{n-1}$ which proves the convergence.

\end{proof}

\noindent
Lemma \ref{lem:convolution_is_l2} shows that the Fourier transform of $\mu_\alpha * \mu_\beta$ is in $L^2$, so that one can take its inverse Fourier Transform. This leads to the following result.

\begin{lemma}[Inverse Fourier of Convolution]

Let $(\alpha, \beta) \in \mathcal{H}_{reg}^2$ and $g \in SU(n)$. Then, 
\begin{equation}
\label{eq:inverse_of_convolution}
    (\widehat{m}_{\alpha, \beta})^\lor (g) = \sum_{\lambda \in \Z_{\geq 0}^{n-1}} \frac{1}{\dim V_\lambda} \chi_\lambda(\ed^{2i \pi \alpha}) \chi_\lambda(\ed^{2i \pi \beta}) \chi_\lambda(g^{-1}),
\end{equation}
where the sum converges in $L^2(SU_n)$.
\end{lemma}

\begin{proof}
    Using \eqref{eq:inv_fourier} together with \eqref{eq:prod_fourier} yields

    \begin{align}
        (\widehat{m}_{\alpha, \beta})^\lor (g) &= \sum_{\lambda \in \Z_{\geq 0}^{n-1}} \dim V_\lambda \cdot Tr[ \rho_\lambda(g^{-1}) \widehat{m}_\alpha(\lambda) \widehat{m}_\beta(\lambda) ] \\
        &= \sum_{\lambda \in \Z_{\geq 0}^{n-1}} \dim V_\lambda \cdot Tr \left[   \frac{\chi_\lambda(\ed^{2i \pi \alpha})\chi_\lambda(\ed^{2i \pi \beta})}{\dim V_\lambda^2} \rho_\lambda(g^{-1})  \right] \\
        &= \sum_{\lambda \in \Z_{\geq 0}^{n-1}} \frac{1}{\dim V_\lambda} \chi_\lambda(\ed^{2i \pi \alpha}) \chi_\lambda(\ed^{2i \pi \beta}) \chi_\lambda(g^{-1}).
    \end{align}
    
\end{proof}

\begin{proof}[Proof of Proposition \ref{prop:first_density}]
    The induced density on $\gamma$ is given by the density \eqref{eq:inverse_of_convolution} multiplied by the Jacobian of the diagonalization map $g \mapsto V \ed^{2i \pi \gamma} V^* $ with $\gamma = (\gamma_1, \dots, \gamma_n)$ such that $\sum \gamma_i \in \Z$. Since this Jacobian is $\frac{|\Delta(\ed^{2i \pi \gamma})|^2}{(2 \pi)^{n-1}n!}$, see \cite[Thm 11.2.1]{Faraut_2008}, we obtain the desired expression.
\end{proof}

\noindent
Writing the density \eqref{eq:density_horn_1} using \eqref{eq:weyl_character_formula} and the fact that $\dim V_{\lambda}=\frac{\Delta(\lambda')}{sf(n-1)}$ yields

\begin{align}
    \diff \prob[\gamma | \alpha, \beta] =& \frac{\Delta(\ed^{2i \pi \gamma}) sf(n-1)}{(2 \pi)^{n-1}n! \Delta(\ed^{2i \pi \alpha} )\Delta(\ed^{2i \pi \beta})} \sum_{\lambda \in \Z_{\geq 0}^{n-1}} \frac{1}{\Delta(\lambda')} 
    \det [\ed^{2i \pi \alpha_r \lambda_s'}]
    \det [\ed^{2i \pi \beta_r \lambda_s'}]
    \det [\ed^{-2i \pi \gamma_r \lambda_s'}]\notag\\
    =&\frac{sf(n-1)(2 \pi)^{(n-1)(n-2)/2}\Delta(\ed^{2i \pi \gamma}) }{\Delta(\ed^{2i \pi \alpha} )\Delta(\ed^{2i \pi \beta})n!}J[\gamma | \alpha, \beta],
    \label{eq:density_horn_2}
\end{align}
where 
\begin{equation}
\label{eq:definition_J}
J[\gamma | \alpha, \beta]=\frac{1}{(2 \pi)^{n(n-1)/2} } \sum_{\lambda \in \Z_{\geq 0}^{n-1}} \frac{1}{\Delta(\lambda')} 
    \det [\ed^{2i \pi \alpha_r \lambda_s'}]
    \det [\ed^{2i \pi \beta_r \lambda_s'}]
    \det [\ed^{-2i \pi \gamma_r \lambda_s'}].
\end{equation}
is called the volume function for the unitary Horn problem.

\subsection{Link with quantum cohomology of the Grassmannians}
\label{subsec:link_quantum_coho}

The goal of this section is to link the volume function \eqref{eq:definition_J} with structure constants of the quantum cohomology ring of Grassmannians $QH^\bullet(\Gr)$ in the same way as the volume function in the coadjoint case is related to the classical cohomology ring of Grassmannians, see \cite{Coquereaux_2018}. The structure constants in the unitary case are the Gromov-Witten invariants, which are related to characters via \cite[Cor. 6.2]{Rietsch_2001}. We refer the reader to \cite{book_j_holomorphic_curves} and \cite{Buch2003_solo} for an introduction to the subject. \\
\\
For $N \geq n$, denoted by $\Z_{N-n}^n$ the set of partition $\lambda\in \Z^n$ such that $N-n \geq \lambda_1 \geq \dots \geq \lambda_n \geq 0$. Then, the ring $QH^\bullet(\Gr(n, N))$ has an additive basis $(q^d \otimes \sigma_\lambda, d\geq 0, \lambda \in \Z_{N-n}^n)$. We will denote by $c_{\lambda, \mu}^{\nu, d}$ the structure constants of this ring so that 

$$ \sigma_\lambda \cdot \sigma_\mu = \sum_{\nu, d \geq 0} c_{\lambda, \mu}^{\nu, d} q^d\otimes\sigma_{\nu}.$$
where the sum is over pairs $(\nu, d) \in \Z_{N-n}^{n}\times \mathbb{N}$ such that $|\lambda| + |\mu| =  |\nu|  + Nd$. The structure coefficients $c_{\lambda, \mu}^{\nu, d}$ are the degree $d$ Gromov-Witten invariants associated to the Schubert cycles $\sigma_\lambda,\sigma_\mu,\sigma_{\nu^\lor}$, see \cite[Cor. 6.2]{Rietsch_2001}. The main result of this section is Theorem \ref{th:density_limit_quantum_coef} below.

\begin{theorem}[Density as limit of quantum coefficients]
\label{th:density_limit_quantum_coef}
    Let $(\alpha, \beta, \gamma) \in \mathcal{H}_{reg}^3$. For each $N \geq 1$, let $(\lambda_N, \mu_N, \nu_N)$ be three partitions in $\Z_{N-n}^{n}$ such that $ \vert \lambda_N \vert + \vert \mu_N \vert = \vert \nu_N \vert  + dN$ for some $d \in \Z_{\geq 0}$ and such that $\frac{1}{N} \lambda_N = \alpha + o(1)$, $\frac{1}{N} \mu_N = \beta + o(1)$ and $\frac{1}{N} \nu_N = \gamma + o(1)$ as $N \rightarrow + \infty$. Then,
\begin{equation}
\label{eq:density_limit_quantum_coefs}
    \lim_{N \rightarrow \infty} N^{-(n-1)(n-2)/2} c_{\lambda_N, \mu_N}^{\nu_N, d}=J[\gamma | \alpha, \beta]  =\frac{\Delta(\ed^{2i \pi \alpha} )\Delta(\ed^{2i \pi \beta} ) n! }{sf(n-1)(2\pi)^{(n-1)(n-2)/2} \Delta(\ed^{2i \pi \gamma} )} \diff \prob[\gamma | \alpha, \beta].
\end{equation}

\end{theorem}

\noindent
The rest of this section is devoted to the proof of Theorem \ref{th:density_limit_quantum_coef}. In subsection \ref{subsec:det_expression} we prove a determinantal formula for the coefficients $c_{\lambda, \mu}^{\nu, d}$ along with some results on the quantities involved in the expression. In subsection \ref{subsec:convergence_to_density} we prove Theorem \ref{th:density_limit_quantum_coef} using Lemmas \ref{lem:pt_conv} and \ref{lem:unif_control}.

\subsubsection{Determinantal expression for $c_{\lambda, \mu}^{\nu, d}$ }
\label{subsec:det_expression}

For $1 \leq n \leq N$, set
   $$ I_{n, N} = \left \{ (I_1, \dots, I_n) \in \left( \Z +\left(\frac{1}{2}\right)^{(n-1)[2]}\right)^n : -\frac{n-1}{2} \leq I_n < \dots < I_1 \leq N - \frac{n+1}{2} \right \}.$$

\begin{lemma}[Determinantal expression for $c_{\lambda, \mu}^{\nu, d}$]
    Let $\lambda, \mu, \nu$ such that $|\lambda| + |\mu|   = |\nu| + Nd$. Then,
    \begin{equation}
    \label{eq:det_expression_quant_const}
    c_{\lambda, \mu}^{\nu, d} = \frac{1}{N^n} \sum_{I \in I_{n, N}} \frac{\det [\ed^{ \frac{2i \pi I_r (\lambda_s+(s-1)) }{N}}]\det [\ed^{ \frac{2i \pi I_r(\mu_s+(s-1)) }{N}}]\det [\ed^{ -\frac{2i \pi I_r (\nu_s+(s-1)) }{N}}]}{\Delta(\xi^I)}
    \end{equation}
\end{lemma}

\begin{proof}
    
Let $\xi = \exp{(2i \pi /N)}$ and for $I \in I_{n, N}$, set $\xi^I = (\xi^{I_1}, \dots,\xi^{I_n})$. Let $S_\lambda(x_1, \dots, x_n) = \frac{\det[x_r^{(\lambda_s + (s-1))}, 1 \leq r,s \leq n]}{\Delta(x)}$ be the Schur function corresponding to the partition $\lambda$. Using \cite[Corollary 6.2]{Rietsch_2001}:

\begin{equation}
\label{eq:rietsch_1}
    c_{\lambda, \mu}^{\nu, d} = \frac{1}{N^n} \sum_{I \in I_{n,N}} S_\lambda(\xi^I) S_\mu(\xi^I) S_{\nu^\lor}(\xi^I) \frac{|\Delta(\xi^I)|^2}{S_{(N-n)}(\xi^I)}.
\end{equation}
Moreover, by \cite[eq. (4.3)]{Rietsch_2001}, one has 

   $$ \frac{S_{\nu^\lor}(\xi^I)}{S_{(N-n)}(\xi^I)} = \overline{S_{\nu}(\xi^I)}$$

so that 

\begin{align}
    c_{\lambda, \mu}^{\nu, d} &= \frac{1}{N^n} \sum_{I \in I_{n,m}} S_\lambda(\xi^I) S_\mu(\xi^I) S_{\nu}(\overline{\xi^I}) |\Delta(\xi^I)|^2 \nonumber\\
    &= \frac{1}{N^n} \sum_{I \in I_{n,m}} \frac{\det [\ed^{ \frac{2i \pi I_r (\lambda_s+(s-1)) }{N}}]\det [\ed^{ \frac{2i \pi I_r(\mu_s+(s-1)) }{N}}]\det [\ed^{ -\frac{2i \pi I_r (\nu_s+(s-1)) }{N}}]}{\Delta(\xi^I)} 
    \label{eq:gromov_det}
\end{align}

\end{proof}

\noindent
We are interested in the asymptotic behaviour of the previous expression as $N \rightarrow \infty$. Define

\begin{equation}
\label{eq:def_F}
    F(I, \lambda, \mu, \nu, N) =  \frac{\det [\ed^{ \frac{2i \pi I_r (\lambda_s+(s-1)) }{N}}]\det [\ed^{ \frac{2i \pi I_r(\mu_s+(s-1)) }{N}}]\det [\ed^{ -\frac{2i \pi I_r (\nu_s+(s-1)) }{N}}]}{\Delta(\xi^I)}.
\end{equation}

\begin{lemma}[Translation invariance]
\label{lem:trans_inv}
Let $I \in \left( \frac{1}{2} \Z \right)^n$ and $a \in \frac{1}{2} \Z$. We still assume that $|\lambda| + |\mu| = |\nu|  + Nd$ for some $d \in \Z_{\geq 0}$. Then,
\begin{equation}
    F(I+a, \lambda, \mu, \nu, N) = F(I, \lambda, \mu, \nu, N).
\end{equation}
    
\end{lemma}

\begin{proof}
    Since 
    \begin{equation*}
        \det \left[ \exp{ \left( \frac{2i \pi (I_r+a) (\lambda_s+s-1)}{N} \right)} \right] 
        = \det \left[\exp{ \left( \frac{2i \pi I_r (\lambda_s+s-1)}{N} \right) } \right] \exp{ \left( a \frac{2i \pi (|\lambda| + \sum_{l=0}^{n-1}l)}{N} \right)},
    \end{equation*}
    the numerator of $F(I+a, \lambda, \mu, \nu, N)$ is the one of $F(I, \lambda, \mu, \nu, N))$ times the factor
    
    \begin{equation*}
        \exp{\left( a\frac{2i \pi}{N} \left(|\lambda| + |\mu| - |\nu| + n(n-1)/2 \right) \right)} = \exp{ \left( a \frac{i \pi n(n-1)}{N}  \right)}
    \end{equation*}
    since $|\lambda| + |\mu| - |\nu| = 0 \ [N]$. The Vandermonde in the denominator of $F(I+a, \lambda, \mu, \nu, N)$ is
    \begin{align*}
        \Delta(\xi^{I + a}) &= \prod_{1 \leq r < s \leq n } \left( \exp{ \left( \frac{2i \pi (I_r+a) }{N} \right)} - \exp{ \left( \frac{2i \pi (I_s+a) }{N} \right)} \right) \\
        &= \prod_{1 \leq r < s \leq n } \left( \exp{ \left( \frac{2i \pi I_r) }{N} \right)} - \exp{ \left( \frac{2i \pi I_s }{N} \right)} \right) \exp{ \left( a \frac{i \pi n(n-1)}{N}  \right)} \\
        &= \Delta(\xi^I)\exp{ \left( a \frac{i \pi n(n-1)}{N}  \right)}
    \end{align*}
    so that the quotient cancels the common additional factor appearing in the numerator and denominator of $F(I+a, \lambda, \mu, \nu, N)$.
\end{proof}

\noindent
From Lemma \ref{lem:trans_inv}, we can shift $I$ by $\frac{n-1}{2}$. In the following, we will assume that $0 \leq I_n < \dots < I_1 \leq N - 1$ and that the $I\in\Z^n$.  Denote by $J_{n,N}$ the set 
$$J_{n,N}=\{I\in \{0,\ldots,N-1\}^n,\, i_1>i_2>\ldots>i_n\}.$$

\begin{definition}[Action $\Phi_N$ and orbits]
    
The translation action of $\Z$ on $ J_{n,N}$ is given by
\begin{align*}
    \Phi_N: & \ \Z \times J_{n,N} \hspace{2.1cm}\rightarrow J_{n,N}\\
    & (l, I= (i_1 > \dots > i_n)) \mapsto I + (l, \dots, l) \ [N]
\end{align*}
where the tuple $I + (l, \dots, l) [N]$ consists of the sequence of elements $i_1 + l [N], \dots, i_n + l [N]$ sorted in the decreasing order.
\end{definition}
\noindent
Lemma \ref{lem:trans_inv} shows that $F$ is invariant under the action of $\Phi_N$.\\
\\
In order to give some properties of orbits of $\Phi_N$, let us recall the lexicographic order on $\mathbb{Z}_{\geq 0}^n$. For $I, J \in \Z_{\geq 0}^{n-1}$, let $r^* = \inf_{1 \leq r \leq n} I_r \neq J_r$, with the convention that $r^* = 0$ if $I = J$. We say that $I > J$ if $I_{r^*} > J_{r^*}$ and $I < J$ if $I_{r^*} < J_{r^*}$. This defines a total order on $\Z_{\geq 0}^{n}$ and by restriction on $J_{n,N}$. \\
\\
Let $Orbits(N)$ denote the orbits of the action of $\Phi_N$ on $J_{n,N}$. For $\Omega_N \in Orbits(N)$, denote by $\min(\Omega_N)$ its minimal element with respect to the lexicographic order. Then, necessarily, $(\min(\Omega_N))_n = 0$, otherwise $\Phi(-1, \min(\Omega_N))$ would be an element of $\Omega_N$ strictly inferior to $\min(\Omega_N)$. For an ordered $n$-tuple $I$ of $\{0, \dots, N-1 \}^n$, let $\Omega(I, N)$ denote its orbit under the action of $\Phi_N$.

\begin{lemma}[Orbit structure for large $N$]
\label{lem:orbit_struct}

Let $I = (I_1 > \dots > I_{n-1} > I_n = 0)$. Then, for $N$ large enough, the orbit of $I$ under the action of $\Phi_N$ has cardinal $N$ and $I$ is its minimal element:
\begin{equation}
    \exists M=M(I), \forall N \geq M: I = \min(\Omega(I, N)) \text{ and }  |\Omega(I, N)| = N.
\end{equation}

\end{lemma}

\begin{proof}
    Let $G_{I,N}$ denote the stabilizer of $I$ under $\Phi_N$. Then, $N \Z \subset G_{I,N}$ so that $G_{I,N} = p_N \Z$ for some $p_N \geq 1$ such that $p_N | N$. Set $d(N) = N - I_1 \geq 1$. Then, for $p_N \Z$ to be the stabilizer of $I$, one must have $d(N) \leq p_N$ (recall that $I_n = 0$), for otherwise $I_1+p_N\not\in \{I_2,\ldots,I_n\}$. However, as $I_1$ is fixed, for $N > 2I_1 $, $\frac{N}{2} < d(N) \leq p_N$ which implies with $p_N | N$ that $p_N = N$. Hence, for such $N$, the corresponding orbit has cardinal $N$. \\
    \\
    Let us show that $I$ is minimal in its orbit $\Omega(I, N)$ when $N$ is large enough. Take $N$ such that $I_1 <\frac{N}{2}$. The only points $J$ in the orbit of $I$ such that $J_n=0$ and $J\not=I$ are
    \begin{equation*}
        \{ \Phi_N(-I_{n-1}, I), \Phi_N(-I_{n-2}, I), \dots, \Phi_N(-I_{1}, I) \}
    \end{equation*}
    These tuples are all strictly greater than $I$ since $ N - (I_k - I_{k+1}) - I_1 \geq N - 2I_1 > 0$. Since the minimal element of an orbit must have $I_n=0$, the only possibility is $I$.
\end{proof}

\begin{lemma}[Orbit decomposition]
\label{lem:orbit_decomp}
    Let $n, N$ be fixed. Then,
    \begin{equation}
        \sum_{I \in I_{n, N}} F(I, \lambda, \mu, \nu, N) = \sum_{I: I_n =0} \mathrm{1}_{I = \min(\Omega(I, N))} |\Omega(I, N) | F(I, \lambda, \mu, \nu, N).
    \end{equation}
\end{lemma}

\begin{proof}[Proof of Lemma \ref{lem:orbit_decomp}]
By the translation invariance of Lemma \ref{lem:trans_inv},

        $$\sum_{I \in I_{n, N}} F(I, \lambda, \mu, \nu, N) = \sum_{ 0 \leq I_n < \dots < I_1 \leq N-1 } F(I, \lambda, \mu, \nu, N) $$

\noindent
We decompose the elements $0 \leq I_n < \dots < I_1 \leq N-1$ along orbits of the action defined in Section \ref{subsec:det_expression}. 
%For $m$ large enough, thanks to Lemma \ref{lem:orbit_struct}, these orbits are indexed by elements $0 = I_n < \dots < I_1 \leq m-1$. Thus,
\begin{align}
    \sum_{ 0 \leq I_n < \dots < I_1 \leq N-1 } F(I, \lambda, \mu, \nu, N) &= \sum_{\Omega \in Orbits(N)} \sum_{I \in \Omega} F(I, \lambda, \mu, \nu, N) \nonumber\\
    &= \sum_{\Omega \in Orbits(N)} |\Omega| F(\min(\Omega), \lambda, \mu, \nu, N)\nonumber \\
    &=  \sum_{ 0 \leq I_n < \dots < I_1 \leq N-1 } \mathrm{1}_{I = \min(\Omega(I, N))} |\Omega(I, m)| F(I, \lambda, \mu, \nu, N) \nonumber
\end{align} 
    
\end{proof}

\noindent
We will need the following result which asserts that $I_1/N$ cannot be arbitrary close to one.

\begin{lemma}[Uniform spacing of $I_1$]
\label{lem:unif_spacing}
    Let $n$ be fixed. Then,
  
       $$ \forall N \geq n, \forall \ \Omega \in Orbits(N): \frac{(\min(\Omega))_1}{N} \leq 1 - \frac{1}{n}.$$

\end{lemma}

\begin{proof}[Proof of Lemma \ref{lem:unif_spacing}]
    Let $N \geq n$ and consider $\Omega \in Orbits(N)$. Denote $I = \min(\Omega)$ its minimal element. Assume for the sake of contradiction that $I_1 > N - \frac{N}{n}$. Divide the interval $]0,N]$ in $n$ disjoint sub-intervals $P_1, \dots P_n$ of length $\frac{N}{n}$ with
    \begin{equation*}
        P_j = \left] j-1 \frac{N}{n} , j\frac{N}{n} \right], 1 \leq j \leq n.
    \end{equation*}
    Since $I_1 > N - \frac{N}{n}$ and $I_n=0$, $I_1$ and $I_n$ both belong to the last interval $P_n$. There are $n-2$ remaining elements $I_2 > \dots > I_{n-1}$ to be placed inside the $n-1$ unused intervals $P_1, \dots, P_{n-1}$ and $P_n$. Thus, there exists $1 \leq j \leq n-1$ such that $I \cap P_j = \emptyset$. Take the maximal such $j$ and consider $r = \max \{ l \in [1, n]: I_l \geq  j \frac{N}{n} \}$ the index of the smallest element of $I$ greater than $P_j$. We claim that
    \begin{equation*}
        J = \Phi_N(- I_r, I) < I
    \end{equation*}
    Indeed, since $P_j$ is empty, $J_1 \leq N - \frac{N}{n} < I_1$. This contradicts the fact that $I$ is minimal in the orbit $\Omega_N$.

    \begin{figure}[h]
        \centering
        \includegraphics[scale=1.2]{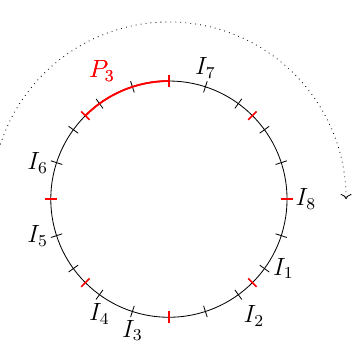}
        \caption{Illustration of the argument for $N=20$ and $n=8$. Red ticks are the $j\frac{N}{n}$ for $0 \leq j \leq n-1$ delimiting the $P_j$'s. Here $j=3$ is the maximal index for which $P_j$ is empty, see the red arc for $P_3$ and $r=6$ with $I_r=9$. The rotation $\Phi_N(-I_r, I) = \Phi_N(-9, I) = J$ is represented by the dotted arrow. $J$ has $J_1 = I_7 - 9 [20] = 15$ which is strictly inferior to $I_1 = 18$ leading to a contradiction as $I$ should be minimal in its orbit.}
        \label{fig:enter-label}
    \end{figure}
    
\end{proof}

\subsubsection{Convergence of scaled coefficients}
\label{subsec:convergence_to_density}

\begin{lemma}[Control of $F(I, \lambda, \mu, \nu, N)$]
\label{lem:unif_control}
    Let $n, N$ be fixed with $n \geq 3$. Then,
    \begin{equation}
       \frac{N^{-n+1}}{N^{(n-1)(n-2)/2}} | F(I, \lambda, \mu, \nu, N) | \leq C_I.
    \end{equation}
    for some $C_I$ such that $\sum_{I: I_n = 0} C_I < \infty$.
    
\end{lemma}

\begin{proof}[Proof of Lemma \ref{lem:unif_control}]
    One has

       $$ | F(I, \lambda, \mu, \nu, N) | = \left| \frac{\det [\ed^{ \frac{2i \pi I_r \lambda_s' }{N}}]\det [\ed^{ \frac{2i \pi I_r \mu_s' }{N}}]\det [\ed^{ -\frac{2i \pi I_r (\nu)_s' }{N}}]}{\Delta(\xi^I)} \right| \\
        \leq \frac{n^{3n}}{ | \Delta(\xi^I) |}. $$
First, 

\begin{align*}
    \frac{1}{ | \Delta(\xi^I) |} &=  \prod_{1 \leq r < s \leq n } \left| \exp{ \left( \frac{2i \pi I_r }{N} \right)} - \exp{ \left( \frac{2i \pi I_s }{N} \right)} \right|^{-1} \\
    &= \prod_{1 \leq r < s \leq n } \left| 2 \sin{ \left( \frac{ \pi (I_r-I_s) }{N} \right)} \right|^{-1} .
    %&\leq (m\pi)^{n(n-1)/2} \prod_{1 \leq r < s \leq n } \frac{1}{\min(I_r - I_s,m-(I_r-I_s))}.
\end{align*}
Recall that on $[0, c]$ for $0<c<\pi$, one has by concavity $\sin(x) \geq \frac{\sin(c)}{c}x$. Using Lemma \ref{lem:unif_spacing} for $I$ minimal in its orbit,

   $$ \forall 1 \leq  r < s \leq n: \pi \frac{ I_r-I_s }{N} \leq \pi \frac{ I_1 }{N} \leq \pi \left( 1 - \frac{1}{n} \right),$$
so that 
    $$\sin{ \left( \frac{ \pi (I_r-I_s) }{N} \right)} \geq  c_n  \frac{ I_r - I_s}{N}$$
with $c_n = \frac{\sin{\pi(1 - 1/n)}}{ (1- 1/n)} $. Thus,

    $$\frac{1}{ | \Delta(\xi^I) |} \leq \left( \frac{N}{2c_n} \right)^{n(n-1)/2} \prod_{1 \leq r < s \leq n } \frac{1}{I_r - I_s}.$$
\noindent
It remains to prove that $\sum_{I: I_n = 0} \frac{1}{\Delta(I)} < \infty$. We will proceed by induction on $n$. For $n =3$, the sum is

    $$\sum_{I_1 > I_2 > I_3 = 0} \frac{1}{(I_1 - I_2)I_1 I_2} =  \sum_{I_2 \geq 1} \frac{1}{I_2} \sum_{I_1 \geq I_2 +1} \frac{1}{(I_1 - I_2)I_1}.$$
Moreover, for $I_2 \geq 1$,

$$    \sum_{I_1 \geq I_2 +1} \frac{1}{(I_1 - I_2)I_1} \leq \frac{1}{I_2 + 1} + \int_{I_2+1}^{\infty} \frac{1}{t(t - I_2)} \diff t =  \frac{1}{I_2 + 1} + \frac{\ln(I_2+1)}{I_2}$$
which proves the convergence for $n=3$ since $ \sum_{I_2 \geq 1} \frac{1}{I_2} \left( \frac{1}{I_2 + 1} + \frac{\ln(I_2+1)}{I_2} \right) = C < \infty$ . For $n \geq 4$, 

   $$ \sum_{I_1 > \dots > I_{n-1} > I_n = 0} \prod_{1 \leq r < s \leq n } (I_r - I_s)^{-1} = \sum_{I_2 > \dots > I_{n-1} > I_n = 0} \prod_{2 \leq r < s \leq n } (I_r - I_s)^{-1} \sum_{I_1 > I_2} \prod_{2 \leq s \leq n } (I_1 - I_s)^{-1}$$
and, since 

$$    \sum_{I_1 > I_2} \prod_{2 \leq s \leq n } (I_1 - I_s)^{-1} \leq \sum_{I_1 > I_2} (I_1 - I_2)^{-(n-1)} \leq c_3 = \frac{\pi^2}{6},$$
we have  

$$    \sum_{I_1 > \dots > I_{n-1} > I_n = 0} \prod_{1 \leq r < s \leq n } (I_r - I_s)^{-1} \leq c_3  \sum_{I_2 > \dots > I_{n-1} > I_n = 0} \prod_{2 \leq r < s \leq n } (I_r - I_s)^{-1} \leq c_3^{n-3} C < \infty.$$

\noindent
Therefore, 

   $$\frac{N^{-n+1}}{N^{(n-1)(n-2)/2}} | F(I, \lambda, \mu, \nu, N) | \leq \frac{n^{3n}}{ | \Delta(\xi^I) |} =  C_I.$$
with $\sum_{I: I_n = 0} C_I < \infty$ as wanted.

\end{proof}

\begin{lemma}[Pointwise convergence]
\label{lem:pt_conv}
    Let $(\alpha, \beta, \gamma) \in \mathcal{H}^3$ such that $\sum_{i=1}^n\alpha_i+\sum_{i=1}^n\beta_i=\sum_{i=1}^n\gamma_i+d$ for $d\in\mathbb{N}$. For $N \geq 1$, let $(\lambda_N, \mu_N, \nu_N)$ be three partition in $\Z_{N-n}^n$ such that $ \vert \lambda_N \vert + \vert \mu_N \vert = \vert \nu_N \vert + dN$ for some $d \in \Z_{\geq 0}$ and such that $\frac{1}{N} \lambda_N = \alpha + o(1)$, $\frac{1}{N} \mu_N = \beta + o(1)$ and $\frac{1}{N} \nu_N = \gamma + o(1)$ as $N \rightarrow + \infty$. Let  $I = I_1 > \dots > I_{n-1} > I_n =0 $ be fixed. Then,
    \begin{align}
       \lim_{N \rightarrow \infty} \frac{N^{-n}\delta_{I = \min(\Omega(I, N))}|\Omega(I, N)|}{N^{(n-1)(n-2)/2}}  & F(I, \lambda_N, \mu_N, \nu_N, N) =
       \lim_{N \rightarrow \infty}  \frac{N^{-n+1}}{N^{(n-1)(n-2)/2}}  F(I, \lambda_N, \mu_N, \nu_N, N) \\
       &= (2 \pi)^{-n(n-1)/2} \frac{1}{\Delta(I)} \det [\ed^{2i \pi \alpha_r I_s}] \det [\ed^{2i \pi \beta_r I_s}] \det [\ed^{-2i \pi \gamma_r I_s}].
    \end{align}
\end{lemma}

\begin{proof}[Proof of Lemma \ref{lem:pt_conv}]

The first equality is derived from Lemma \ref{lem:orbit_struct} which implies that for any $I$ and $N$ large enough
\begin{equation*}
    \delta_{I = \min(\Omega(I, N))}|\Omega(I, N)| = N.
\end{equation*}
For a fixed $n \geq 3$, by continuity, 

\begin{align*}
     \lim_{N \rightarrow \infty}  &\det \left[\exp{ \left( \frac{2i \pi I_r (\lambda_{N, s} +s-1) }{N} \right)} \right] = \det [\exp{ (2i \pi I_r \alpha_s )}] \\
    \lim_{N \rightarrow \infty}  &\det \left[\exp{ \left( \frac{2i \pi I_r (\mu_{N, s} + s-1) }{N} \right)} \right] = \det [\exp{ (2i \pi I_r \beta_s )}] \\
    \lim_{N \rightarrow \infty}  &\det \left[\exp{ \left( -\frac{2i \pi I_r (\nu_{N, s}+s-1) }{N} \right)} \right] = \det [\exp{ (-2i \pi I_r \gamma_s )}] \\
    \lim_{N \rightarrow \infty}  & \frac{N^{-n+1}}{N^{(n-1)(n-2)/2}\Delta(\xi^I)} = \left( \frac{1}{2 \pi} \right)^{n(n-1)/2} \frac{1}{\Delta(I)}
\end{align*}
where for the last convergences, we used that $\sin{ \left( \frac{ \pi (I_r-I_s) }{N} \right)} \sim \frac{ \pi (I_r-I_s) }{N} $ for a fixed subset $I$. The four convergences above imply the result.

\end{proof}

\begin{proof}[Proof of Theorem \ref{th:density_limit_quantum_coef}]

From $\eqref{eq:gromov_det}$, together with Lemma \ref{lem:orbit_decomp}, one has

 $$   c_{\lambda_N, \mu_N}^{\nu_N, d}  = N^{-n}  \sum_{I: I_n = 0} \delta_{I = \min(\Omega(I, N))} |\Omega(I, N)| F(I, \lambda_N, \mu_N, \nu_N, N).$$

\noindent
By Lemma \ref{lem:pt_conv} and Lemma \ref{lem:unif_control} using the dominated convergence theorem we have that
$$    \lim_{N \rightarrow \infty} N^{-(n-1)(n-2)/2} c_{\lambda_N, \mu_N}^{\nu_N, d}  = \sum_{I: I_n = 0} \frac{(2 \pi)^{-n(n-1)/2}}{\Delta(I)} \det [\ed^{2i \pi \alpha_r I_s}] \det [\ed^{2i \pi \beta_r I_s}] \det [\ed^{-2i \pi \gamma_r I_s}] \\ = J[\gamma | \alpha, \beta],$$
where $J[\gamma | \alpha, \beta]$ was defined in \eqref{eq:definition_J} and is such that 

    $$\diff \prob[\gamma | \alpha, \beta] = \frac{sf(n-1)(2 \pi)^{(n-1)(n-2)/2}\Delta(\ed^{2i \pi \gamma}) }{\Delta(\ed^{2i \pi \alpha} )\Delta(\ed^{2i \pi \beta})n!}J[\gamma | \alpha, \beta].$$

\end{proof}

\section{Puzzles of the quantum cohomology of Grassmannians and their skeleton}\label{Sec:Puzzle_skeleton}
The main goal of this section is a rewriting of the puzzle formula of \cite{puzzle_conj_two_step} for the expression of quantum LR-coefficient in terms of a a more compact form approaching the hive model yielding the classical LR-coefficients, see \cite{Knutson_1999}. 

\subsection{Triangular grid}
\begin{definition}[Triangular grid]
\label{def:triangular_grid}
The triangular grid of size $N$, denoted by $T_N$, is the planar graph whose vertices are the set $V_N=\{r+se^{i\pi/3}, r,s\in \mathbb{N}, r+s\leq N\}$ and edges are the set $E_N=\{(x,y), x,y\in T_N, \vert y-x\vert=1\}$.

The set $F_N$ of faces of $T_N$ are triangles which are called direct (resp. reversed) if the corresponding vertices $(x^1,x^2,x^3)\in V_N^3$ can be labelled in such a way that $x^2-x^1=1$ and $x_3-x_1=e^{i\pi/3}$ (resp. $x_3-x_1=e^{-i\pi/3}$). 

Any union of two triangles sharing an edge $e$ is called a lozenge, and $e$ is then called the middle edge of the lozenge.
\end{definition}

\noindent
We denote by $F_N^+$ (resp. $F_N^-$) the set of directed (resp. reversed) triangles, so that $F_N=F_N^+\cup F_N^-$.  For $e\in E_N,f\in F_N$, we write $e\in f$ when $e$ is an edge on the boundary of $f$. \\
Remark that the set of edges can be partitioned into three subset depending on their orientation. If $x=r+se^{i\pi/3}\in T_N$ we define three coordinates 
$$x_0=N-(r+s),\,x_1=r,\, x_2=s,$$
and we usually denote an element of $T_N$ by those three coordinates to emphasize the threefold symmetry of the triangle.
We say that an edge $e=(x,x+v)$ is of type $\ell, \, \ell\in\{0,1,2\}$  when $v=e^{i\pi+2\ell i \pi/3}$.
\begin{figure}[h!]
  
   \begin{tikzpicture}
 \node[draw,circle,fill=gray!50] (P) at (0,0) {};
 \node[draw,circle,fill=gray!50] (P0) at (-2,0) {};
 \node[draw,circle,fill=gray!50] (P1) at (2-2*1/2,-2*0.866) {};
 \node[draw,circle,fill=gray!50] (P2) at (0+2*1/2,2*0.866) {};
 
 \draw (P) -- (P0);
 \draw (P) -- (P1);
 \draw (P) -- (P2);
 
 \node  at (-1,0.3) {Type $0$};
 \node  at (1.3,-1*0.866) {Type $1$};
 \node  at (1.3,1*0.866) {Type $2$};
 \node (Eti) at (-4,-2) {Origin of the edge};

 \draw[->,dashed]  (Eti) to [bend right] (P);
 
   \end{tikzpicture}
    \caption{Type of an edge in $T_N$}
    \label{fig:type_edge}
\end{figure}
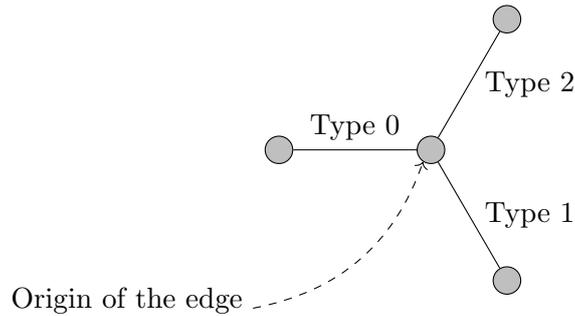

\begin{definition}[Edge coordinates]
\label{def:edge_coordinate}
For $x\in T_N$ and $\ell\in\{0,1,2\}$ such that $x+e^{i\pi+2\ell i\pi/3}\in T_N$, the coordinates of the edge $e=(x,x+e^{i\pi+2\ell i \pi/3})$ of type $\ell$ is the triple $(e_0,e_1,e_2)$ given by $e_i=x_i$. We define the height $h(e)$ of an edge of type $\ell$ by 
$$h(e)=e_{\ell}.$$
\end{definition}
\noindent
If $e=(x,y)$ is of type $\ell$, we have 
\begin{equation}\label{eq:coor_vertices_edge}
y_\ell=x_{\ell}+1,\quad y_{\ell-1}=x_{\ell-1},\quad y_{\ell+1}=x_{\ell+1}-1.
\end{equation}
 We denote by $E_k^{(\ell)}$ the set of edges of type $\ell$. Remark that the height of an edge does not characterize its position, since for example the translations of an edge of type $1$ by $e^{i\pi/3}$ will have the same height.

\begin{definition}[Discrete boundary]
The boundary $\partial T_N$ of the triangular grid $T_N$ is the set of edges $(x,y)$ lying on the boundary of the triangle $[0,N]\cup[N,Ne^{i\pi/3}]\cup[0,Ne^{i\pi/3}]$.
\end{definition}
The boundary $\partial T_N$ can be decomposed into three subsets $\partial T_N^{(i)}$, $0\leq i\leq 2$, where each set $\partial T_N^{(i)}$ consists of edges of type $i$. The coordinates of the corresponding edges are then the following.
    \begin{align*}
        \partial T_N^{(0)} &=  \left( (r,N-r,0), 0 \leq r \leq N-1 \right), \\
        \partial T_N^{(1)}&=  \left((0 , r,N-r), 0\leq r \leq N-1 \right), \\ 
        \partial T_N^{(2)} &=  \left( (N-r,0,r) , 0 \leq r \leq N-1 \right).
        \end{align*}

\subsection{Puzzles and the quantum-LR coefficients}
We will mainly work on puzzles describing the two-step flag cohomology from \cite{puzzle_conj_two_step}, in the special case where they describe the quantum Littlewood-Richardson coefficients previously introduced in Section \ref{Sec:link_quantum_cohomolgy}. 

Let us consider the set of puzzle pieces given in Figure \ref{fig:pieces}, which are considered as the assignment of a label in $\{0,\ldots,7\}$ to edges of $T_N$ around a triangular face. Each piece can be rotated by a multiple of $\frac{\pi}{3}$ but not reflected.

\begin{figure}[H]
    \centering
    \includegraphics{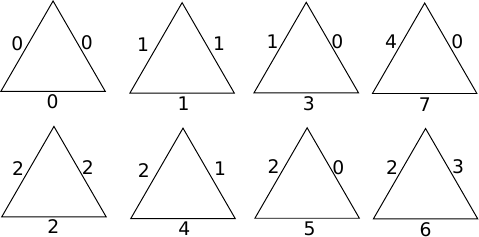}
    \caption{Possible pieces of the puzzle}
    \label{fig:pieces}
\end{figure}

%\begin{figure}[h!]
 %   \centering
  %  \includegraphics[scale=0.5]{images/pieces_list.png}
   % \caption{Possible pieces}
    %\label{fig:pieces}
%\end{figure}
\begin{definition}\label{def:puzzle} A triangular puzzle of size $N \geq 1$ is a map $P:E_N\rightarrow \{0,\ldots,7\}$ such that the value around each triangular face belongs to the set of possible puzzle piece displayed in Figure \ref{fig:pieces}.

The boundary coloring $\partial P$ of a puzzle $P$ is the sequence $(\omega_0,\omega_1,\omega_2)$ such that $\omega_\ell$ is the sequence $(P(e))_{e\in\partial^{(\ell)} T_N}$, where $\partial^{(\ell)} T_N$ is the sequence of boundary edges of $T_N$ of type $\ell$ ordered by their height.

\end{definition}
For any triple $(\omega_0,\omega_1,\omega_2)$ of words in $\{0,1,2\}^N$, we denote by $P(\omega_0,\omega_1,\omega_2)$ the set of puzzles whose boundary coloring is $(\omega_0,\omega_1,\omega_2)$. For $0\leq k_0\leq k_0+k_1\leq N$, denote by $\mathbb{F}l(k_0,k_1,N)$ the two-step flag manifold 
$$\mathbb{F}l(k_0,k_1,N)=\{V_0\subset V_1\subset \mathbb{C}^N,\,\dim V_0=k_0, \dim V_1=k_0+k_1\}.$$
The cohomology ring $H\mathbb{F}l(k_0,k_1,N)$ admits a basis $\{\sigma_{\omega}\}$ of Schubert cycles indexed by words in $\{0,1,2\}^N$ with $k_0$ occurrences of $0$ and $k_1$ occurrences of $1$. Proving a conjecture of Knutson, it has been shown in \cite{puzzle_conj_two_step} that the previously constructed puzzles describe the structure constants of $H\mathbb{F}l(k_0,k_1,N)$.

 \begin{theorem}[\cite{puzzle_conj_two_step}]\label{thm:puzzle_two_steps}
 For any triple $(\omega_0,\omega_1,\omega_2)$ of words in $\{0,1,2\}^N$ with same number of occurrences $k_0$ of $0$ and $k_1$ of $1$, 
 $$\langle \sigma_{\omega_0}\sigma_{\omega_1}\sigma_{\omega_2},\sigma_{0}\rangle_{H\mathbb{F}l(k_0,k_1,N)}=\#P(\omega_0,\omega_1,\omega_2),$$
 where $\sigma_{0}$ is the fundamental class of $H\mathbb{F}l(k_0,k_1,N)$.
 \end{theorem}
Thanks to a previous work \cite{Buch_2003} relating the quantum cohomology of Grassmannians to the classical cohomology of the two-step flag manifold, Theorem \ref{thm:puzzle_two_steps} yields a similar expression in terms of puzzles for the quantum Littlewood-Richardson coefficients. 
 \begin{corollary}[\cite{puzzle_conj_two_step}]\label{cor:puzzle_qLR}
 Let $1\leq n\leq N$ and $\lambda^0,\lambda^1,\lambda^2$ be partitions of length $n$ with first part smaller than $N-n$ such that $\vert \lambda^1\vert+\vert \lambda^2\vert=\vert \lambda^0\vert+Nd$. Then, $c_{\lambda^1,\lambda^2}^{\lambda^0,d}=\#P(\omega_0,\omega_1,\omega_2)$, where $\omega_\ell, \ell\in\{0,1,2\}$ are constructed as follows :
 \begin{enumerate}
     \item for $\ell\in\{1,2\}$, set $\omega_\ell(\lambda^{\ell}_i+(n-i))=0$ for $1\leq i\leq n$ and $\omega_\ell(i)=2$ otherwise,
     \item set $\omega_0(N-1-(\lambda^{0}_i+(n-i)))=0$ for $1\leq i\leq n$ and $\omega_0(i)=2$ otherwise,
     \item for $\ell\in\{0,1,2\}$, replace the $d$ last occurrences of $0$ and the $d$ first occurrences of $2$ in $\omega_\ell$ by $1$. 
 \end{enumerate}
 \end{corollary}
 The goal of this section and the next one is then to give a convex formulation of the latter results, yielding Theorem \ref{thm:bijection_puzzle_dual_hive} for the expression of the structure constants of $H\mathbb{F}l(k_0,k_1,N)$ and Corollary \ref{cor:qLR_dual_hive} for the corresponding result concerning the quantum Littlewood-Richardson coefficients.

\noindent

\subsection{Edge, vertex and face partitions}
 The set of pieces can be further simplified in two steps. First,  gluing two pieces with edges having label $2$ along an edge labeled $4,5$ or $6$ yields lozenges with edge labelled $2$ and either $0,1$ or $3$; then concatenating consecutively such lozenges having same label $0,1$ or $3$ and considering also the triangle with all edges labeled $2$ yield the pieces of Figure \ref{fig:type_II_pieces}, which are called pieces of type II. Let us then call pieces of type I any piece displayed in Figure \ref{fig:type_I_pieces} which consists of the first three triangles of Figure \ref{fig:pieces} and pieces obtained by concatenating two pieces with label $(2,4,1)$ and an arbitrary even number of pieces with label $(4,7,0)$. We will first show that a puzzle $P$ is completely characterized by the position of pieces of type I.

\begin{figure}[h!]
    \centering
    \includegraphics{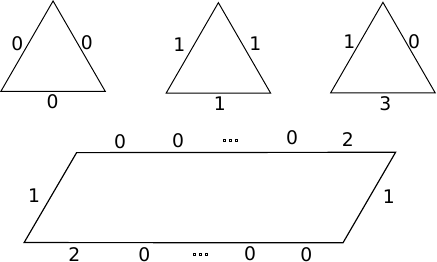}
    \caption{Puzzle pieces of type I.}
    \label{fig:type_I_pieces}
\end{figure}

\begin{figure}[h!]
    \centering
    \includegraphics{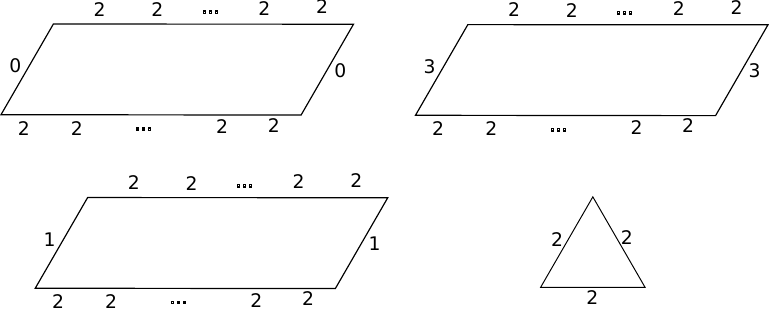}
    \caption{Puzzle pieces of type II.}
    \label{fig:type_II_pieces}
\end{figure}

Let us first mention a first general result on height of edges on the border of a same triangle.

\begin{lemma}[Triangle sum]\label{lem:triangle_sum}
Suppose that $f$ is a face of $T_N$ with edges $e^0,e^1,e^2$. Then, 
    \begin{align}
    \begin{split}
        h(e^0)+h(e^1)+h(e^2) &=N-1  \text{ if $f$ is direct} \\
        h(e^0)+h(e^1)+h(e^2) &=N-2 \text{ if $f$ is reversed}.  
    \end{split}
    \label{eq:triangle_sum}
    \end{align}
 Moreover, $e^i_{i-1}=e^{i-1}_{i-1}$ (resp. $e^{i}_{i-1}=e^{i+1}_{i-1}$), for $i\in\{0,1,2\}$ if $f$ is direct (resp. reversed).
\end{lemma}
\begin{proof}
    A direct triangle has edges $e^0=(N-(i+j+1),i+1,j)$, $e^1=(N-(i+j+1),i,j+1)$ and $e^2=(N-(i+j),i,j)$ for some $0\leq i,j\leq N-1$ with $i+j\leq N-1$, so that 
    $$h(e^0)+h(e^1)+h(e^2)=N-1+0+0=N-(i+j+1)+i+j=N-1.$$
    A reversed triangle has edges $e^0=(N-(i+j+1),i,j+1)$, $e^1=(N-(i+j),i-1,j+1)$ and $e^2=(N-(i+j),i,j)$ for some $1\leq i,j\leq N-1$ with $i+j\leq N-1$, so that 
    $$h(e^0)+h(e^1)+h(e^2)=N-1+0+0=N-(i+j+1)+i-1+j=N-2.$$
    It is also clear from the coordinates of the edges that $e^{i}_{i-1}=e^{i-1}_{i-1}$ (resp. $e^i_{i-1}=e^{i+1}_{i-1})$ for $i\in\{0,1,2\}$ if the triangle is direct (resp. reversed).
\end{proof}

\begin{definition}[Edge set, vertex, edge and face partitions]
The edge set of a puzzle $P$ is the set $\mathcal{E}$ of edges labeled $0,1,3$ of either a type I piece or on the boundary. 

The vertex partition of $\mathcal{E}$ is the  covering $\mathcal{P}_v$ of $\mathcal{E}$ whose sets of size greater than one consist of all the edges colored $\{0,1,3\}$ of a same type I piece and singletons consists of edges of $\mathcal{E}$ on the boundary of $T_N$ not belonging to a type $I$ piece.

The edge partition of $\mathcal{E}$ is the set partition $\mathcal{P}_e$ of $\mathcal{E}$ whose blocks of size greater than one consist of edges of a common type II piece.

The face partition $\mathcal{P}_f$ is the set partition of $V_N$ whose blocks are the connected components of the subgraph of $T_N$ obtained by only keeping the edges colored $2$.
\end{definition}
Remark that $\mathcal{P}_e$ can also have singletons. An element $e\in\mathcal{E}$ is a singleton of $\mathcal{P}_e$ if and only if it is a common edge of two type I pieces. However, no element of $\mathcal{E}$ can be a singleton of both $\mathcal{P}_v$ and $\mathcal{P}_e$, since a border edge colored $0$ or $1$ not belonging to a type $I$ piece has to belong to a type II piece.

\begin{lemma}[Blocks of $\mathcal{P}_v$]
\label{lem:block_vertex}
A block of order $3$ in $\mathcal{P}_v$ consists of three edges $e^0,e^1,e^2$ of type $0,1,2$ such that 
\begin{itemize}
    \item either $e^i_{i-1}=e^{i-1}_{i-1}$, $i\in\{0,1,2\}$ and $\sum_{i=0}^2h(e^i)=N-1$, or $e^i_{i-1}=e^{i+1}_{i-1}$, $i\in\{0,1,2\}$ and $\sum_{i=0}^2h(e^i)=N-2$,
    \item $(c(e^0),c(e^1),c(e^2))$ is either $(0,0,0)$, $(1,1,1)$ or any cyclic permutation of $(0,1,3)$.
\end{itemize}
A block of order $2(r+1)$, $r\geq 2$ in $\mathcal{P}_v$ consists of $2(r+1)$ edges $\{e^0,f^1,\ldots,f^r,e^{0'},f^{1'},\ldots,f^{r'}\}$ such that 
\begin{itemize}
     \item $e^{0},e^{0'}$ are of type $i$ and $f^1,\ldots,f^r,f^{1'},\ldots,f^{r'}$ are of type $i+1\mod 3$ for some $i\in\{0,1,2\}$,
     \item $h(e^0)=h(e^{0'})$ and $f^1_i=\dots=f^{r}_i=f^{1'}_i+1=\dots=f^{r'}_i+1=h(e^0)+1$,
    \item $h(f^{s})=h(f^{s'})=h(f^{1})+(s-1)=e^{0'}_{i+1}$ for $1 \leq s \leq r$,
    \item the edges $f^{0}$ and $f^{r+1'}$ of type $i+1$ with $f^{r+1}_{i}=f^{r+1'}_{i}+1=h(e^0)+1$ and $h(f^{0})=e^{0'}_{i+1}-1$ and $h(f^{r+1'})=e^{0}_{i+1}-1$ are not in $\mathcal{E}$.
\end{itemize}
Any edge $e\in\mathcal{E}$ belongs to at most two blocks of $\mathcal{P}_v$.
\end{lemma}
\begin{proof}
In the case of a block of order $3$, $\{e^0,e^1,e^2\}$ is a triangle of $T_N$ and the results on the height of the edges is given by Lemma \ref{lem:triangle_sum}. The results on the color of the edges is given by the possible coloring of edges of Type I pieces from Figure \ref{fig:type_I_pieces}.

In the case of a block $B$ of order $2(r+1)$, the edges correspond to the boundary edges not colored $2$ of a puzzle piece of the last shape of Figure \ref{fig:type_I_pieces}. We can thus first label cyclically the boundary edges colored $0$ and $1$ as $\{e^0,f^1,\ldots,f^r,e^{0'},f^{1'},\ldots,f^{r'}\}$ so that $e^0$ and $e^{0'}$ are colored $1$ and of type $i\in\{0,1,2\}$ and $f^i,f^{i'},\, 1\leq i\leq r$ are colored $0$ and are of type $i+1$. Then, remark that such a piece is the concatenation of $r+1$ direct triangles $T_1,\ldots,T_{r+1}$ and $r+1$ reversed triangles $T'_1,\ldots,T'_{r+1}$ such that $T_i$ and $T'_i$ (resp. $T'_i$ and $T_{i+1}$) share an edge of type $i+1$ (resp. $i - 1$), $e^0$ (resp. $e^{0'}$ is the edge of type $i$ of $T_1$ (resp. $T'_{r+1}$) end $f^i$ (resp. $f^{i'}$) is the edge of type $i+1$ of $T'_{i}$ (resp. $T_{i+1}$). The relations giving the height and the labels of the edges are then direct consequences of Lemma \ref{lem:triangle_sum}.
\end{proof}

\begin{definition}[Admissible pair and strip]
\label{def:strip}
A pair $B=\{e^1,e^2\}\subset E_N$ of edges is called admissible if $e^1$ and $e^2$ have same type $j\in \{0,1,2\}$ and same height. 

The strip $S_B$ of an admissible pair $B=\{e^1,e^2\}$ of type $j$ is the set of all edges $e=(x,y)\in E_N$ such that $x,y$ belong to the parallelogram delimited by $e^1$ and $e^2$. Namely, if $e^1_j=e^2_j$ and $e^2_{j-1}\leq e^1_{j-1}$, $S_B$ is the set of edges $(x,y)$ such that $e^1_j\leq x_j,y_j\leq e^1_j+1$, $e^2_{j-1}\leq x_{j-1},y_{j-1}\leq e^1_{j-1}$. 

The boundary $\partial_2 S_B$ of a strip $S_B$ consists of all edges of $S_B$ type $j+1$.
\end{definition}
Remark that such a definition is still valid if $e^1=e^2$, in which case $B=\{e^1\}$ is always admissible and $S_B=\{e^1\}$.
In particular, for any $B\in \mathcal{P}_e$ for a puzzle $P$, $S_B$ consists of all edges appearing on the type II piece bordered by elements of $B$: $B$ coincide exactly the boundary edges of $S_B$ which are not labeled $2$ and $\partial_2 S_B$ consists of the boundary edges of the type II piece which are labeled $2$ (that is, all boundary edges except for $B$).

Moreover, given a pair $B=\{e^1,e^2\}$ of edges of same type $j$ and same height, one can rephrase the condition of belonging to the strip $S_B$ depending on the type of the edge we consider:
\begin{itemize}
    \item if $e$ has type $j$, $e\in S_B$ if and only if $e_j=e^1_j$ and $e_{j-1}\in [e^2_{j-1}, e^1_{j-1}],$
    \item if $e$ has type $j+1$, $e\in S_B$ if and only if $e_j\in\{e^1_j,e^1_j+1\}$ and $e_{j-1}\in ]e^2_{j-1}, e^1_{j-1}],$
    \item if $e$ has type $j-1$, $e\in S_B$ if and only if $e_j=e^1_j+1$ and $e_{j-1}\in [e^2_{j-1}, e^1_{j-1}[.$
\end{itemize}

\begin{lemma}[Edge partition pairs are admissible]
\label{lem:condition_pairing}
Let $P$ be a puzzle and $\mathcal{P}_e$ the corresponding edge partition. Any pair $\{e^1,e^2\}\in \mathcal{P}_e$ (with possibly $e^1=e^2$) is admissible and is such that $c(e^1)=c(e^2)\in\{0,1,3\}$.
For any different blocks $B,B'\in \mathcal{P}_e$, $S_B\cap S_{B'}\subset \partial_2 S_B\cap \partial_2S_{B'}$.
\end{lemma}
\begin{proof}
Recall that blocks of $\mathcal{P}_e$ of size greater that one corresponds to edges colored $0,1$ or $3$ on the border of a same Type II piece from \ref{fig:type_II_pieces}. In particular the blocks of size greater than ones are only pairs, and the first part of the lemma is a direct consequence of the possible type $II$ pieces of Figure \ref{fig:type_II_pieces}.

For the second part of the lemma, remark first that for any block $\{e^1,e^2\}$ in $\mathcal{P}_e$, $S_B\cap \mathcal{E}=B$: first, boundary edges of $S_B$ are either colored $0,1$ or $3$ and in $B$ or colored $2$ and not in $\mathcal{E}$. Then, interior edges of $S_B$ which are colored $0,1$ or $3$ are boundary edges of two  pieces with the same labelling of the second row of Figure \ref{fig:pieces} and thus are not in $\mathcal{E}$ (remark that the second piece of the second row of Figure \ref{fig:pieces} is used in the last type I piece of Figure \ref{fig:pieces} but is surrounded by pieces with different boundary labels).

Consider two different blocks $B,B'\in \mathcal{P}_e$. Let $e\in S_B\cap S_{B'}$. Remark that any edge of $S_B$ or $S_{B'}$ is not on the border of the strip if and only if it is neither in $B\cup B'$ nor colored $2$. Hence, if $e$ is colored $2$, then $e\in \partial_2S_{B}\cap\partial_2S_{B'}$. Suppose by contradiction that $e$ is not colored $2$. Since $B\cap S_{B'}=S_{B}\cap B'=\emptyset$ and $e$ is not colored $2$, $e$ belongs to the interior of both $S_B$ and $S_{B'}$ : hence, the two triangular puzzle pieces whose boundary is $e$ belong to $S_B$ and $S_{B'}$, and we deduce that there is an edge $f^1$ colored $0,1$ or $3$ which belong to $S_{B}\cap S_{B'}$. If $f^1\not \in B\cup B'$, $f^1$ belong to the interior of $S_B$ and $S_{B'}$, and thus there exists an edge $f^2\in S_{B}\cap S_{B'}$ with $f^2_{i+1}=f^1_{i+1}+1$. Let us repeat the process until there is an edge $f^s\in S_B\cap S_{B'}$ which belongs to either $B$ or $B'$. This means that $B\cap S_{B'}$ or $B'\cap S_B$ is not empty, which contradicts the fact proven previously that $S_B\cap \mathcal{E}=B$ and $S_{B'}\cap \mathcal{E}=B'$. Hence, any edge not colored $2$ does not belong to $S_B\cap S_{B'}$. 
\end{proof}

\begin{definition}[Crossing pairs] \label{def:crossing_pairs}
    We say that two  admissible pairs $B=\{e^1,e^2\}$ and $B'=\{e^3,e^4\}$ of $E_N$ cross when 
\begin{itemize}
    \item either $B\cap S_{B'}\not=\emptyset$ or $B'\cap S_B\not=\emptyset$,
    \item or $B=\{e^1,e^2\}, B'=\{e^3,e^4\}$ are blocks of size two which are of respective type $j,j+1$ for some $j\in\{0,1,2\}$, and
$$e^3_j\leq e_j^1=e_j^2\leq e_j^4 \quad e^1_{j+1}\leq e^3_{j+1}=e_{j+1}^4< e_{j+1}^2.$$
\end{itemize}
\end{definition}

\begin{lemma}[Non-crossing condition]
For any distinct pairs $B,B' \subset E_N$, the two following properties are equivalent :
\begin{enumerate}
\item $B$ and $B'$ do not cross,
\item $S_B\cap S_{B'}\subset \partial_2 S_B\cap \partial_2 S_{B'}$.
\end{enumerate}
\end{lemma}
\begin{proof}
Let $B,B'$ be admissible pairs of $E_N$. 

If $B$  is a singleton, then $\partial_2 S_B=\emptyset$ and thus $\partial_2S_B\cap \partial_2S_{B'}=\emptyset$. Moreover, since $B$ is a singleton, $B=S_B$. Hence, the non-crossing condition is equivalent to $S_B\cap S_{B'}=\emptyset$, and $B$ and $B'$ do not cross if and only if $S_B\cap S_{B'}\subset\emptyset=\partial_2S_B\cap \partial_2S_{B'}$.

If $B=\{e^1,e^2\}$ and $B'=\{e^3,e^4\}$ are pairs of the same type $j\in\{0,1,2\}$, the non-crossing condition means that $B\cap S_{B'}=\emptyset$ and $B'\cap S_{B}=\emptyset$. Suppose that $B$ and $B'$ cross, and without loss of generality, assume that $B\cap S_{B'}\not= \emptyset$. Since $B\cap \partial_2S_B=\emptyset$, we deduce that $B\cap S_{B'}\not\subset \partial_2S_B\cap\partial_2S_{B'}$. Hence, $S_B\cap S_{B'}\not\subset \partial_2S_B\cap\partial_2S_{B'}$. 

Reciprocally, suppose that $S_B\cap S_{B'}\not\subset \partial_2S_B\cap\partial_2S_{B'}$, and assume without loss of generality that $S_B\cap S_{B'}\cap \left(S_B\setminus \partial_2S_B\right)\not=\emptyset$. Since $\partial_2S_B$ is the set of edges of $S_B$ of type $j+1$, there exists an edge $e=(x,y)$ of type $j$ or $j-1$ in $S_B\cap S_{B'}$. If $e$ is of type $j$, this means that $e_j=e^1_j=e^2_j$ and $e_{j+1}\in [e^{1}_{j+1},e^{2}_{j+1}]$. Similarly, $e_j=e^3_j=e^4_j$ and $e_{j+1}\in [e^{3}_{j+1},e^{4}_{j+1}]$. Hence, $[e^{1}_{j+1},e^{2}_{j+1}]\cap [e^{3}_{j+1},e^{4}_{j+1}]\not=\emptyset$, and the extremity of one of these intervals is contained in the other. Assume without loss of generality that $e^{1}_{j+1}\subset [e^{3}_{j+1},e^{4}_{j+1}]$. Then, $e_1$ is an edge of type $j$ such that $e^1_j=e^3_j=e^4_j$ and $e^1_{j+1}\in [e^{3}_{j+1},e^{4}_{j+1}]$, thus $e^1\in S_{B'}$ and thus $B\cap S_{B'}\not =\emptyset$.  If $e=(x,y)$ is of type $j-1$, then $x_{j-1}=y_{j-1}-1$ and $x_{j}=y_j+1$, see \eqref{eq:coor_vertices_edge}. Hence, the conditions $x_j,y_j\in \{e^1_j,e^1_j+1\}$ and $x_{j-1},y_{j-1}\in [e^2_{j-1},e^1_{j-1}]$ from Definition \ref{def:strip} yield that $e'=(x',x)$ with $x'_{j}=y_j$ and $x'_{j-1}=x_j$ is an edge of type $j$ which belongs to $S_B$. Similarly, $e'\in S_{B'}$, and the previous reasoning allows to conclude that $B\cap S_{B'}\not=\emptyset$ or $B'\cap S_B\not=\emptyset$.

Suppose finally that $B=\{e^1,e^2\}$ and $B'=\{e^3,e^4\}$ are pairs of respective type $j$ and $j+1$ for $j\in \{0,1,2\}$. Remark first that edges of $\partial_2S_B$ have type $j+1$ and edges of $\partial_2S_{B'}$ have type $j+2$, so that $\partial_2 S_B\cap \partial_2 S_{B'}=\emptyset$. 

Suppose that $B$ and $B'$ cross. First, if $B\cap S_{B'}\not=\emptyset$ or $B'\cap S_{B'}\not=\emptyset$, then $S_B\cap S_{B'}\not=\emptyset$ and thus $S_B\cap S_{B'}\not\subset \partial_2 S_B\cap \partial_2 S_{B'}$. Suppose that $B\cap S_{B'}=B'\cap S_{B'}=\emptyset$, and thus
$$e^3_j\leq e_j^1=e_j^2\leq e_j^4 \quad e^1_{j+1}\leq e^3_{j+1}=e_{j+1}^4<e_{j+1}^2.$$
Consider the edge $e=(x,y)$ of type $j+1$ with $x_j=e^1_j$ and $x_{j+1}=e^3_{j+1}$. Since $e$ is of type $j+1$, $y_j=x_j$ and $y_{j+1}=x_{j+1}+1$. First, since $x_{j+1},y_{j+1}\in\{e^3_{j+1},e^3_{j+1}+1\}$ and $x_j=y_j=e^1_j\in [e^{3}_j,e^4_{j}]$, $e\in S_{B'}$. Then, remark that 
$$x_{j-1}=N-(x_j+x_{j+1})=N-e^1_j-e^3_{j+1},\, e^1_{j-1}=N-e^1_j-e^1_{j+1},\,e^2_{j-1}=N-e^2_j-e^2_{j+1}.$$
From the equality $e^1_j=e^2_j$ and $e^1_{j+1}\leq e^3_{j+1}<e^2_{j+1}$, we deduce that $x_{j-1}\in ]e^2_{j-1},e^1_{j-1}]$. Since $y_{j-1}=N-y_j-y_{j+1}=N-x_j-x_{j+1}-1=x_{j-1}-1$, $y_{j-1}\in [e^2_{j-1},e^1_{j-1}]$. The two latter inclusions together with $y_j=x_j=e^1_j$ yield that $e\in S_B$. In particular, $S_B\cap S_{B'}\not=\emptyset$ and thus $S_B\cap S_{B'}\not\subset \partial_2 S_B\cap \partial_2 S_{B'}$.

Suppose that $S_B\cap S_{B'}\not=\emptyset$. If $B\cap S_{B'}\not=\emptyset$ or $B'\cap S_B\not=\emptyset$, then $B$ and $B'$ cross. When $B\cap S_{B'}=B'\cap S_{B}=\emptyset$, then $S_B\cap S_{B'}\not =\emptyset$ if and only if $\partial_2 S_B\cap  S_{B'}\not=\emptyset$. One implication is straightforward. For the other implication, remark that there is necessarily an edge $e\in S_B\cap S_{B'}$ which belongs to $B$ or $\partial_2S_B$, and since $B\cap S_{B'}=\emptyset$, $e\in \partial_2 S_B$. In particular, since $B$ has type $j$, edges of $\partial_2 S_B$ have type $j+1$ and thus $e$ has type $j+1$. 

The edge $e=(x,y)$ of type $j+1$ belongs to $\partial_2 S_B$ if and only
$$\left\lbrace \begin{aligned}
    &x_{j}=y_j\in\{e^1_{j},e^1_j+1\},\\
    &x_{j-1}\in[e^{2}_{j-1},e^{1}_{j-1}] \text{ and }y_{j-1}=x_{j-1}-1\in[e^{2}_{j-1},e^{1}_{j-1}].
\end{aligned}\right.$$
Similarly $e$ of type $j+1$ to belong to $S_{B'}$ if and only if  
$$\left\lbrace \begin{aligned}
    &x_{j+1}=y_{j+1}-1=e^3_{j+1}=e^4_{j+1},\\
    &x_{j}=y_j\in[e^{3}_{j},e^{4}_{j}].
\end{aligned}\right.$$
Hence, $e\in \partial_2 S_B\cap S_{B'} $ if and only if
$$\left\lbrace \begin{aligned}
    &x_j=y_j\in\{e^1_{j},e^1_j+1\}\cap[e^{3}_{j},e^{4}_{j}], \,x_{j+1}=y_{j+1}-1=e^3_{j+1}\\
    &N-x_j-e^3_{j+1}\in ]N-e^{2}_{j}-e^2_{j+1},N-e^{1}_{j}-e^1_{j+1}].
\end{aligned}\right.$$
If $x_j=e^1_j$, the latter conditions imply
$$\left\lbrace \begin{aligned}
    &e^{3}_{j}\leq e^1_{j}=e^2_j\leq e^{4}_{j}, \\
    & e^1_{j+1}\leq e^3_{j+1}=e^4_{j+1}<e^2_{j+1}.
\end{aligned}\right.$$
If $x_j=e^1_j+1$, the latter conditions yield
$$\left\lbrace \begin{aligned}
    &e^{3}_{j}-1\leq e^1_{j}=e^2_j\leq e^{4}_{j}-1,\,x_{j+1}=y_{j+1}-1=e^3_{j+1}\\
    & e^1_{j+1}-1\leq e^3_{j+1}=e^4_{j+1}<e^2_{j+1}-1.
\end{aligned}\right.$$
If $e^3_j-1=e^2_j$ and $e^1_{j+1}-1\leq e^3_{j+1}<e^2_{j+1}-1$, then $e^3_j=e^2_j+1$ and 
$$e^3_{j-1}=N-e^3_j-e^3_{j+1}\in ]N-e^2_j-e^2_{j+1},N-e^2_j-e^1_{j+1}]= ]e^2_{j-1},e^1_{j-1}],$$
so that $e^3\in S_B$ by the condition following Definition \ref{def:strip}. Likewise, if $e^3_{j+1}=e^1_{j+1}-1$ and $e^1_j\in[e^3_j,e^4_j]$, then $e^1\in S_{B'}$. Hence, the fact that $B\cap S_{B'}=B'\cap S_B=\emptyset$ strengthens the above condition to imply 
$$\left\lbrace \begin{aligned}
     &e^{3}_{j}\leq e^1_{j}=e^2_j\leq e^{4}_{j}, \\
    & e^1_{j+1}\leq e^3_{j+1}=e^4_{j+1}<e^2_{j+1}.
\end{aligned}\right.$$
\end{proof}

\noindent
From the latter lemma, we deduce a description of puzzles in terms of their type I pieces. 

\begin{proposition}[Partitions determine puzzles]
\label{prop:puzzle_from_partition}
The map $\Phi:P\mapsto (\mathcal{E},c)$ is a bijection from the set of puzzles to the set of subsets of $E_N$ with a coloring $c:\mathcal{E}\rightarrow \{0,1,3\}$ such that there exist a covering $\mathcal{P}_v$ and a partition $\mathcal{P}_e$ of $\mathcal{E}$ with
\begin{itemize}
    \item blocks of $\mathcal{P}_v$ are of size $3$ or $2(r+1),\,r\geq 1$ and satisfy the properties of Lemma \ref{lem:block_vertex},
    \item blocks of $\mathcal{P}_e$ are either singletons or pairs satisfying the properties of Lemma \ref{lem:condition_pairing},
    \item If $e$ belong to only one block of $\mathcal{P}_v$, then $e$ is not a singleton of $\mathcal{P}_e$ and if $e\not=e'$ belong to a same block of $\mathcal{P}_v$, then $\{e,e'\}\not\in \mathcal{P}_e$.
\end{itemize}
\end{proposition}
\begin{proof}
Let us build the candidate inverse bijection, and consider a subset $\mathcal{E}\subset E_N$ with a coloring $c:\mathcal{E}\rightarrow \{0,1,3\}$ and a covering $\mathcal{P}_v$ and a partition $\mathcal{P}_e$ satisfying the conditions of Proposition \ref{prop:puzzle_from_partition}. For each pair $B=(e,e')\in\mathcal{P}_e$ of type $i$ and color $c$, color all edges of type $i$ (resp. $i+1$, resp. $i+2$) of the strip $S_B\setminus \{e,e'\}$ with the color $c$ (resp. $2$, resp. $c+5$). Remark that such a coloring is possible, since by properties of Lemma \ref{lem:condition_pairing}, any edge belonging to $S_B\cap S_{B'}$ for two strips of respective types $i,i'$ must be included in the boundary $\partial_2S_B\cap \partial_2S_{B'}$, which consists then of edges of same type $i+1=i'+1$ colored $2$ by the above rule. 

Then, consider any block $B\in\mathcal{P}_v$ of order $2(r+1)$ whose edges $e,e'$ colored $1$ are of type $i$. Remark then that by the properties of Lemma \ref{lem:block_vertex}, $B=\{e,e'\}$ is an admissible pair of edges of $\mathcal{E}$ of same color $1$. Moreover, all edges but two of the boundary $\partial_2S_{B}$ of the strip $S_{B}$ consists of edges of $\mathcal{E}$ colored $0$. 

Let $i$ be the type of $B$, and suppose by contradiction that there is an edge $e^0$ of type $i$ or $i-1$ inside the strip $S_{B}$ which is contained in a strip $S_{\tilde{B}}$ for some $\tilde{B}\in\mathcal{P}_e$. Suppose first that $\tilde{B}\not=\{e^0\}$. Then, since $e^0$ shares a vertex with at least three edges of type $i+1$ colored $0$, $S_{\tilde{B}}$ must contain one of those three edges, called $f^1$; since $\mathcal{P}_e$ is non-crossing and $f^1\in \mathcal{E}\cap S_{\tilde{B}}$, $f^1\in \tilde{B}$ and thus $\tilde{B}$ is of type $i+1$. Since $e^0\in S_{\tilde{B}}$ and $\tilde{B}$ is of type $i+1$, the other edge $f^2$ of $S_{B}$ with same height as $f^1$ must also belong to $S_{\tilde{B}}$, and by the non-crossing condition we have $\tilde{B}=\{f^1,f^2\}$. This contradicts the fact that two elements of the same block of $\mathcal{P}_v$ do not belong to the same block of $\mathcal{P}_e$. 

If $e^0$ is a singleton of $\mathcal{P}_e$, then $e\in\mathcal{E}$ and belongs to at least two blocks of $\mathcal{P}_v$. Hence, there must be another edge $e^1$ between $e^0$ and $e$ in the strip $S_{B}$ which belongs to $\mathcal{E}$. Iterating the process yields an edge $\tilde{e}$ such that $e$ and $\tilde{e}$ belong to a same block $\tilde{B}$ in $\mathcal{P}_v$. Then, $\tilde{e}$ cannot be of type $i-1$ otherwise this block would be a block of order $3$, contradicting the fact that the edge $f$ of type $i+1$ with $f_i=e_i, f_{i+1}=e_{i+1}-1$, which would then belong to this block $\tilde{B}$, does not belong to $\mathcal{E}$. Similarly, if $\tilde{e}$ is of type $i$, then $e$ and $\tilde{e}$ would be boundary edges of type I piece with $2(r'+1)$ edges, with $1\leq r'<r$. This is impossible since the boundary $\partial_2S_{\{e,\tilde{e}\}}$ contains at most one edge which is not in $\mathcal{E}$. Hence, no edge of type $i$ or $i-1$ inside $S_{B}$ belongs to some strip $S_B$ for $B\in \mathcal{P}_e$ and thus none of those edges has been colored $2$ in the previous labelling. Therefore, one can color all type $i$ edge in $S_{\tilde{B}}$ different from $e,e'$ with the label $7$ and all type $i-1$ edge in $S_{\tilde{B}}$ with label $6$.

Finally, the edge $f$ of type $i+1$ with $f_i=e_i, f_{i+1}=e_{i+1}-1$ can not be part of a strip $S_{\tilde{B}}$ for some $\tilde{B}=(f^1,f^2)$ of type $i+1$ or $i-1$, for otherwise $e$ would also belong to $S_{\tilde{B}}$ and $\mathcal{P}_e$ would not be non-crossing. Hence, either $f\in \partial_2 S_B$ for some strip $S_B$ and $f$ has been labelled $2$ in the first coloring step, or $f$ has not been colored before and thus $f$ can be colored $2$. 

Color all remaining edges with the label $2$. One then checks that the labels on the boundary of any triangle of the puzzle satisfy the conditions of Figure \ref{fig:pieces}, so that the labelling of edges of $E_N$ yields a genuine puzzle $P$. It is then straightforward to check that $P_v=\mathcal{P}_v$ and $P_e=\mathcal{P}_e$. The map $\Phi$ is thus surjective.

For the injectivity, remark that the data of $\mathcal{P}_v$ alone gives the list and position of type I pieces of the puzzle, which uniquely characterizes it.
\end{proof}

\subsection{Graph of a puzzle}\label{sec:graph_puzzle}
\begin{definition}[Graph of a puzzle]
The graph of a puzzle $P$ is the graph $\mathcal{G}_P$ whose set of vertices is $P_v$, set of edges is $P_e$ and set faces is $P_f$. 

The endpoints of an edge $B_e\in P_e$ are the vertices $B_v,B_{v'}\in P_v$ such that $B_e\cap B_v\not=\emptyset$ and $B_{e}\cap B_{v'}\not=\emptyset$.

The boundary of a face $B_f\in P_v$ are the edges $B\in P_e$ such that there is $e\in B, v\in B_f$ such that $v$ is an endpoint of $e$. A face $B_f\in P_f$ is called an outer face (resp. inner face) if there is an element (resp. no element) $v\in B_f$ on the border of $T_N$. 
\end{definition}

Remark that elements of $P_v$ and $P_e$ are sets of edges of $T_N$ while elements of $P_f$ are set of vertices of $T_N$. Moreover, any edge $B\in P_e$ has a type $\ell\in\{0,1,2\}$ and a color $c\in\{0,1,3\}$, which is the type and the color of the edges of $T_N$ in $B$.

Let $B_f\in \mathcal{P}_f$ and denote by $\partial B_f$ the set of edges on the boundary of $B_f$. Then, there is a natural cyclic order on $\partial B_f$ such that $\partial B_f=(B_1<\ldots<B_p)$ where $B_i$ and $B_{i+1}$ share a vertex of $\mathcal{P}_v$ and the edges of $\partial B_f$ are read in the clockwise order around the region $B_f$.

\begin{lemma}[Type of face boundaries]
\label{lem:type_face_edge}
Let $B_f\in \mathcal{P}_f$. Then, the sequence of type of edges on the boundary of $B_f$ is a subsequence of $(0,1,2,0,1,2)$ (up to cyclic rotation), and two consecutive edges $B<B'$ on the boundary of $B_f$ sharing a vertex $B_v\in \mathcal{P}_v$ are 
\begin{itemize}
\item of type $(\ell,\ell+1)$ if $B_v$ is a block of size three and the color of $B$ and $B'$ are either both $0$, both $1$ or $(3,0)$, $(0,1)$ or $(1,3)$, or a block of size $2r,\, r\geq 2$ and $B,B'$ have respective color $(0,1)$. 
\item of type $(\ell,\ell+2)$ if $B_v$ is a block of size $2r, \,r\geq 2$ and $B$ and $B'$ have color $(0,1)$,
\item of type $\ell$ if $B_v$ is a block of size $2r, \,r\geq 3$ and $B$ and $B'$ have color $0$.
\end{itemize}
\end{lemma}
\begin{proof}
Let $(B_1,\ldots,B_p)$ be the previously defined cyclic ordering of the edges around $B_f$ such that $B_i,B_{i+1}$ share a vertex in $\mathcal{P}_v$. Let $1\leq i\leq p$ and denote by $\ell$ the type of $B_{i}$. Since $B_i,B_{i+1}$ share the vertex $B_v$, there exist $e^i\in B_i$ and $e^{i+1}\in B_{i+1}$ such that $e^i,e^{i+1}\in B_v$. Since $e^i,e^{i+1}$ are not colored $2$, the type and colors of $e^i$ (resp. $e^{i+1}$) are the ones of $B_{i}$ (resp. $B_{i+1}$).

If $B_v$ is a block of size $3$, then it is a triangle vertex whose boundary colors in the clockwise order are either $(0,0,0)$, always $(1,1,1)$, or $(1,3,0)$ up to a rotation. Since the angle between $e^{i}$ and $e^{i+1}$ is $-\pi/3$, the type of $e_{i+1}$ is $\ell+1$, and the colors $B_i, B_{i+1}$ are either $(0,0)$, $(1,1)$, $(3,0),(0,1)$ or $(1,3)$.

If $B_v$ is a block of size $2(r+1)$, $r\geq 1$, then $3$ configurations can occur depending on the colors of the consecutive edges :
\begin{itemize}
\item if $e^i$ is colored $1$ and $e^{i+1}$ is colored $0$, then the type of $e^{i+1}$ is $\ell+1$,
\item if $e^i$ is colored $0$ and $e^{i+1}$ is colored $1$, then the type of $e^{i+1}$ is $\ell-1$,
\item if $e^i$ and $e^{i+1}$ are both colored $0$ then the edges are adjacent and have same type $\ell$. Remark that in this case, the vertex $B_v$ must have at least $6$ edges.
\end{itemize}
Finally, remark that the angle between two consecutive edges $B_i,B_{i+1}$ is equal to $(1-r_i/3)\pi$ if the difference of the type from $B_i$ to $B_{i+1}$ is $r_i$ (with $r_i=3$ if $B_i$ and $B_{i+1}$ have both type $\ell$). Since the sum of the angles must be equal to the $(p-2)\pi$ if $B_f$ is an inner face and smaller otherwise, we must have $\sum_{i=1}^p(1-r_i/3)\leq p-2$, so that 
$$\sum_{i=1}^pr_i\leq 6.$$
We deduce that the sequence of types of edges of the boundary must be a subsequence of $(0,1,2,0,1,2)$, up to cyclic permutation.
\end{proof}

\section{Discrete two-colored dual hive model}\label{Sec:discrete_two_colored_hive}
In this section, we associate to each puzzle of size $N$ a two-colored hive in the same spirit as in \cite{Knutson_1999}. Beware that because of the rigid crossings from Figure \ref{fig:type_I_pieces}, the discrete hives won't be actual hives as in \cite{Knutson_1999} but rather a dual hive. Let us fix in this section the number $k$ of edges colored $0$ or $1$ on one edge of the puzzles.  This number is the same on each edge of the puzzle and is part of the boundary data of the puzzle. The two colored dual hive associated to a puzzle will then be a decoration of the triangular grid $T_k$ instead of $T_N$.  All the notation introduced for $T_N$ are thus still valid for $T_k$. 

\begin{definition}[Two-colored discrete dual hive]
\label{def:two_col_hives}
A two-colored discrete dual hive of size $(k,N)$ is given by the following combinatorial data on $T_k$:
\begin{itemize}
    \item a color map $C:E_k\rightarrow \{0,1,3,m\}$, such that the boundary colors around each triangular face in the clockwise order is either $(0,0,0),(1,1,1),(1,0,3)$ or $(0,1,m)$ up to a cyclic rotation.
    \item a label map $L:E_k\rightarrow \mathbb{N}$, with the two following conditions:
    \begin{enumerate}
        \item  for all $f\in F_k$ with boundary edges $e^0,e^1,e^2$, $L(e^0)+L(e^1)+L(e^2)=N-1$ except when $f\in F_k^-$ with boundary colors different from $\{0,1,m\}$, in which case $L(e^0)+L(e^1)+L(e^2)=N-2$,
        \item if $e,e'$ are edges of same type $\ell\in\{0,1,2\}$ on the boundary of a same lozenge, then 
     \begin{enumerate}
     \item $L(e)=L(e')$ if the middle edge is colored $m$,
     \item $L(e)\geq L(e')$ if $e'_{\ell+1}=e_{\ell+1}+1$ and no edge different from $e'$ is colored $m$,
     \item $L(e)>L(e')$ if either $e_{\ell}>e'_{\ell}$ or if both $e'_{\ell+1}=e_{\ell+1}+1$ and one of the boundary edges different from $e'$ is colored $m$.
     \end{enumerate}
    \end{enumerate}
\end{itemize}
The boundary value $[(c^{(0)},c^{(1)},c^{(2)}),(l^{(0)},l^{(1)},l^{(2)})]$ of a two-colored discrete dual hive is the restriction of $(C,L)$ to $\partial T_k$, where $c^{(i)}\in \{0,1,3,m\}^k$ (resp. $l^{(i)}\in \mathbb{N}^k$) is the restriction of $C$ (resp. $L$) to $\partial T_k^{(i)}$, for $0\leq i\leq 2$. 
\end{definition}

For $(c,l)=(c^{(0)},c^{(1)},c^{(2)},l^{(0)},l^{(1)},l^{(2)})\in \{0,1\}^{3k}\times \mathbb{N}^{3k}$, we denote by $\mathcal{H}(c,l,N)$ the set of two-colored discrete dual hives with boundary value $(c,l)$. 
\begin{remark}\label{rem:strict_increasing_label}
As a corollary of the Condition $(2)$ on the label map, we have $L(e)>L(e')$ for any pair of edges $e,e'$ of same type $\ell$ such that $e_{\ell}>e'_{\ell}$ and $e_{\ell+1}\leq e'_{\ell+1}$. Indeed, it suffices to show this for $e,e'$ such that $e_{\ell}=e'_{\ell}+1$ and $e'_{\ell+1}\in\{e_{\ell+1},e_{\ell+1}+1\}$. The case $e'_{\ell+1}=e_{\ell+1}$ is given by Condition $(2.c)$, and we now suppose that $e'_{\ell+1}=e_{\ell+1}+1$. Let $e''$ be such that $e''_{\ell+1}=e'_{\ell+1}$ and $e''_{\ell}=e_{\ell}=e_{\ell}'+1$. If the middle edge of the lozenge with boundary $e',e''$ is not colored $m$, then by Condition $(2.c)$ we have $L(e'')>L(e')$, and then by $(2.a)$ or $(2.b)$, we get $L(e)\geq L(e'')>L(e')$. If the middle edge of the lozenge with boundary $e',e''$ is colored $m$, then $L(e')=L(e'')$. Then, the middle edge colored $m$ of this lozenge is then a boundary edge of the lozenge with boundary $e'',e$ different from $e''$ and $e$, so that $(2.c)$ implies that $L(e)>L(e'')=L(e')$.
\end{remark}

For any triple $\omega=(\omega_0,\omega_1,\omega_2)$ of words $\{0,1,2\}^N$ with $k_0$ occurrences of $0$ and $k_1$ occurrences of $1$, denote by $(c(\omega),l(\omega))$ the sequence $(c^{(0)},c^{(1)},c^{(2)},l^{(0)},l^{(1)},l^{(2)})\in \{0,1\}^{3k}\times \mathbb{N}^{3k}$ where $k=k_0+k_1$ and $c^{(i)}$ is the word obtained from $\omega_i$ by deleting the letters $2$ and $l^{(i)}$ is the sequence of positions of the letters $0$ or $1$ in $\omega_i$. The following result gives a formulation of Theorem \ref{thm:puzzle_two_steps} in terms of integer points counting of polytopes.

\begin{theorem}[Dual hive in the two-step case]
\label{thm:bijection_puzzle_dual_hive}
For any triple $\omega=(\omega_0,\omega_1,\omega_2)$ of words $\{0,1,2\}^N$ with $k_0$ occurrences of $0$ and $k_1$ occurrences of $1$,
\begin{equation}
\langle\sigma_{\omega_0}\sigma_{\omega_1}\sigma_{\omega_2},\sigma_0\rangle_{H\mathbb{F}(k_0,k_1,N)}=\# \mathcal{H}(c(\omega),l(\omega),N).
\end{equation}

\end{theorem}

\noindent
By Corollary \ref{cor:puzzle_qLR}, the latter theorem directly yields a similar expression of the quantum Littlewood-Richardson coefficients as the number of integer points in discrete dual hives. For three partitions $\lambda,\mu,\nu$ of length $n$ with first part smaller than $N-n$ and such that $\vert \lambda\vert+\vert \mu\vert=\vert \nu\vert+dN$, set 
$$H(\lambda,\mu,\nu,N)=\mathcal{H}(c(\omega),l(\omega),N),$$
where $\omega$ is the triple of words in $\{0,1,2\}^N$ built as in Corollary \ref{cor:puzzle_qLR} for $\lambda^1=\lambda,\,\lambda^2=\mu,\,\lambda^0=\nu$ and for the corresponding $d$.
\begin{corollary}[Dual hive for the $q$-LR coefficients]\label{cor:qLR_dual_hive} 
For $\lambda,\mu,\nu$ of length $n$ with first part smaller than $N-n$ and such that $\vert \lambda\vert+\vert \mu\vert=\vert \nu\vert+dN$,
$$c_{\lambda,\mu}^{\nu,d}=\#H(\lambda,\mu,\nu,N).$$
\end{corollary}
%We prove this proposition by exhibiting a bijection $\Phi:P(\omega^{c,l})\rightarrow \mathcal{H}(c,l)$ which is of interest for the sequel.
The rest of this section is then devoted to a proof of Theorem \ref{thm:bijection_puzzle_dual_hive}, which is obtained by exhibiting a bijection $\zeta: P(\omega)\rightarrow \mathcal{H}(c(\omega),l(\omega),N)$.

Given a puzzle $P\in P(\omega)$, let $\mathcal{G}_{P}$ be the corresponding graph introduced in Section \ref{sec:graph_puzzle}. Let us first transform the graph $\mathcal{G}_{P}$ into a new graph $\widehat{G}_{P}$ by blowing up each vertex $v\in \mathcal{P}_v$ of size $2(r+1)$ as follows.

\begin{definition}[Blowup of vertex]\label{def:blowup_vertex}
        Let vertex $v\in \mathcal{P}_v$ be a vertex of size $2(r+1)$ with adjacent edges $(B^1,\ldots,B^{2r+2})$ (indexed in the cyclic order) such that $B^1,B^{r+2}$ have type $\ell \in\{0,1,2\}$ and are colored $1$ and $B^i, i\not\in \{1,r+2\}$ have type $\ell+1$ and are colored $0$. Introduce $2r-1$ new edges $\tilde{B}^1,\ldots,\tilde{B}^{2r-1}$ of  type $\ell-1,\ell,\ldots, \ell-1$ and colored $m,1,\ldots,m$ and transform $v$ into $2r$ vertices $v^{1},\ldots,v^{2r}$ such that the edges adjacent to $v^{2j+1}$ are $(\tilde{B}^{2j},B^{2r+2-j},\tilde{B}^{2j+1})$ and edges adjacent to $v^{2j+2}$ are $(\tilde{B}^{2j+1},\tilde{B}^{2j+2},B^{j+2} )$ with the convention $\tilde{B}^0=B^1$ and $\tilde{B}^{2r}=B^{r+2}$. We define the height of $\tilde{B}^i$ as $h(\tilde{B}^{2i})=h(B^1)$ and $h(\tilde{B}^{2i-1})=N-1-h(B^1)-h(B^{2i})$ for $1\leq i\leq r-1$. \\
        \\
        The resulting graph is called the blowup of $v$. 
\end{definition}
The picture of the blowup of a vertex of size $6$ is given in Figure \ref{fig:blowup_vertex}.
\begin{figure}[H]
    \centering
    \scalebox{0.8}{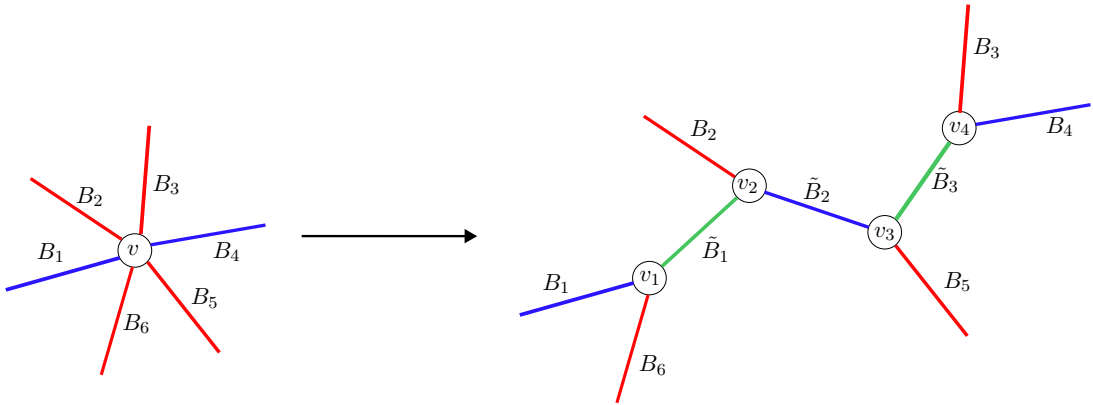}
    \caption{Blowup of a vertex of size $6$ : edges in blue (resp. red, resp. green) are of type $\ell$ (resp. $\ell+1$, resp. $\ell+2$) and colored $1$ (resp. $0$, resp. $m$).}
    \label{fig:blowup_vertex}
\end{figure}
\begin{definition}[Blowup of the graph of a puzzle]
The blowup $\widehat{G}_{P}$ of the graph $\mathcal{G}_P$ is the graph obtained by blowing up every vertex of size $2(r+1)$, $r\geq 1$.
\end{definition}
Remark that the blowup of  graph is well-defined because any vertex of size $2(r+1)$ of $\mathcal{G}_{P}$ has the form of Definition \ref{def:blowup_vertex} thanks to Lemma \ref{lem:block_vertex}. The blowup graph has then only vertices of degree $3$ or singletons (which correspond to edges of $T_N$ on the boundary of the triangle). 
\begin{lemma}[Faces in $\widehat{G}_P$]
Let $B_f$ be a face in $\widehat{G}_P$. Then, the boundary of $B_f$ has 
\begin{itemize}
\item $6$ edges if no edge of the boundary of $B_f$ is a boundary edge of $\mathcal{G}_P$,
\item $4$ edges if two edges of the boundary of $B_f$ are boundary edges of $\mathcal{G}_P$ of the same type,
\item $2$ edges if two edges of the boundary of $B_f$ are boundary edges of $\mathcal{G}_P$ of different type.
\end{itemize}
\end{lemma}
\begin{proof}
By Lemma \ref{lem:type_face_edge} and the blowing-up of vertices of degree larger than $4$, two edges $B,B'$ on the boundary of $B_f$ sharing a vertex  are of type $(\ell,\ell+1)$, and the edge type of boundary edges of $B_f$ is a subsequence of $(0,1,2,0,1,2)$ (up to a cyclic rotation). The only possibility for $B_f$ to have less than $6$ edges on the boundary is then having edges which have a singleton as boundary vertex. The edge of $\mathcal{E}$ corresponding to this singleton is necessary a boundary edge of the graph and thus $B_f$ is a connected component of $\mathcal{P}_2$ touching the boundary of $T_N$. There are then two possibilities : either $B_{f}$ contain one of the three extreme vertex of $T_N$, in which case the boundary edges of $B_f$ have different type and by convexity, $B_f$ has only two edges in $\widehat{G}_P$, or $B_f$ contains only boundary vertices which are not extreme points of $T_N$, in which case the boundary edges have same type $\ell$ and the boundary of $B_f$ consists of four edges of type $(\ell,\ell+1,\ell+2,\ell)$. 
\end{proof}

\subsubsection*{Construction of a discrete two-colored dual hive from a puzzle}

The resulting planar graph $\widehat{\mathcal{G}}_P$ is thus a graph with only trivalent vertices and hexagonal inner faces. From each side of the triangle $T_N$, there are $k-2$ faces $B\in\mathcal{P}_f$ which have degree $4$ (one for each pair of consecutive boundary edge labeled $0$ or $1$ on a same side of $T_N$) and from each extreme vertex of $T_N$ there is a face of degree $2$.

Let us denote by $\widetilde{\mathcal{G}}_P$ the dual graph, namely the graph whose vertices are faces of $\widehat{\mathcal{G}}_P$, faces are vertices of $\widehat{\mathcal{G}}_P$ and such that there is one edge between each neighboring faces of $\widetilde{\mathcal{G}}_P$  (which correspond then to vertices of $\widehat{\mathcal{G}}_P$). 

\begin{lemma}[Dual graph to $T_k$]
There is an isomorphism from $\widetilde{\mathcal{G}}_P$ to $T_k$ mapping edges of type $\ell$ of $\widetilde{\mathcal{G}}_P$ to edges of type $\ell$ of $T_k$.
\end{lemma}

\begin{proof}
Since vertices of $\widehat{\mathcal{G}}_P$ are trivalent, faces of $\widetilde{\mathcal{G}}_P$ are triangular. Similarly, inner faces of $\widehat{\mathcal{G}}_P$ have degree $6$, and thus inner vertices of $\widetilde{\mathcal{G}}_P$ have degree $6$. Hence, $\widetilde{\mathcal{G}}_P$ is isomorphic to a polygon $H$ of the planar triangular grid.
Since the sequence of degrees of the $3k$ outer faces of $\widehat{\mathcal{G}}_{P}$ is 
$$(2,\underset{k\text{ times}}{\underbrace{4,\ldots,4}},2,\underset{k\text{ times}}{\underbrace{4,\ldots,4}},2,\underset{k\text{ times}}{\underbrace{4,\ldots,4}}),$$
the same holds for the sequence of degrees of outer vertices of $\widetilde{\mathcal{G}}_{P}$. Remark that for each vertex of degree $4$ (resp. $2$) in $\widetilde{\mathcal{G}}_P$, the angle of the boundary at the corresponding vertex in $H$ is $\pi$ (resp. $5\pi/3$). We deduce that the boundary of the $\widetilde{\mathcal{G}}_{P}$ is isomorphic to the one of $T_k$, and thus $\widetilde{\mathcal{G}}_{P}$ is isomorphic to $T_k$. Let us denote by $\zeta: E(\widetilde{\mathcal{G}}_P) \rightarrow E_k$ the corresponding bijection between set of edges.

Remark that around every triangle of $\widetilde{\mathcal{G}}_P$ the type of the edges is $(\ell,\ell+1,\ell+2)$ (this is true for faces coming from  trivalent vertex of $\mathcal{G}_P$ and true by construction for faces coming from the blowing up of higher order vertices of $\widehat{\mathcal{G}}_P$). We deduce that all edges with the same type in $\widetilde{\mathcal{G}}_P$ are sent through $\zeta$ to edges with the same orientation in $T_k$. Up to composing $\zeta$ with an internal rotational symmetry of $T_k$, we can thus assume that $\zeta$ preserves the type of the edges.
\end{proof}

\noindent
Each edge $B$ of $\widetilde{\mathcal{G}}_P$ has then a color $c(B)$ and a height $h(B)$ coming from the dual edge of $\widehat{\mathcal{G}}_P$. Composing with $\zeta^{-1}$ yields maps $C:E_k\rightarrow \mathbb{N}$ and $L:E_k\rightarrow \mathbb{N}$ with $C= c \circ \zeta^{-1} $ and $L = h \circ \zeta^{-1}$.  

\begin{lemma}[Image is a dual hive]
The resulting pair of maps $(C,L)$ is a discrete two-colored dual hive $\zeta(P)$ in $\mathcal{H}(c(\omega),l(\omega),N)$.
\end{lemma}
\begin{proof}
Lemma \ref{lem:block_vertex} and the color rules introduced before in case of blowing-up of even degree vertices yield that edges around each trivalent vertex of $\widetilde{\mathcal{G}}_P$ are colored $(0,0,0)$, $(1,1,1)$, $(1,0,3)$ or $(0,1,m)$ in the clockwise order, which translates into the same color rule around each triangular face of $T_k$.

It remains to prove that the map $L$ on $E_k$ satisfies the two conditions of Definition \ref{def:two_col_hives}. The sum condition (1) around a triangle is a consequence of Lemma \ref{lem:triangle_sum} in case no edge is colored $m$, and the direct deduction of the blowing up of vertices of even degree in case one of the edges is colored $m$.

The condition (2) is checked case by case. By Lemma \ref{lem:block_vertex} and the definition of $L$ on edges colored $1$ coming from the blowing-up of even vertices, $L(e)=L(e')$ for any opposite edges $e,e'$ of a lozenge with middle edge colored $m$, yielding the condition $(2.a)$.

Without loss of generality, suppose that $e\in E_k$ has type $\ell$ and let $s$ be a lozenge such that $e$ is a border edge of type $\ell$ of $s$ and the opposite edge $e'$ is a translation of $e$ such that $h(e)\geq h(e')$ and $e_{\ell+1}\leq e'_{\ell+1}$. Hence, either $h(e)=h(e')$, $e_{\ell+1}=e'_{\ell+1}-1$ and the middle edge $f$ of $s$ is of type $k=\ell-1$ or $h(e')=h(e)-1$, $e'_{\ell+1}=e_{\ell+1}$ and the middle edge $f$ is of type $k=\ell+1$.  Suppose now that at least one of the edges of the lozenge is colored $m$ and the middle edge $f$ is not dual to an edge of $\hat{\mathcal{G}}_P$ coming from blowing-up an even vertex. Hence, $f$ corresponds to a strip $S_{f}$ of type $k$. 
Moreover, there exist $\tilde{e},\tilde{e'}\in E_N^2$ of type $\ell$ with $h(\tilde{e})=L(e)$ and $h(\tilde{e}')=L(e')$, $\tilde{f},\tilde{f'}\in S_{f}$  of type $k$ such that $\tilde{e}$ and $\tilde{f}$ (resp. $\tilde{e}'$ and $\tilde{f}'$) comes from a same vertex $v$ (resp. $v'$) of $P_v$ (either directly or after a blowing-up).

Remark that $\tilde{f}_\ell\geq\tilde{f}'_{\ell}$ for otherwise, in the strip $S_f$, there would be an edge of type different from $f$ and labeled $0$ or $1$ coming from $v$, which is not possible from Figure \ref{fig:type_II_pieces}. Then, if $\tilde{e}$ is of type $\ell$ and $\tilde{f}$ is of type $\ell-1$ coming from a same triangle of $T_N$, we resume in Figure \ref{Fig:coordinates_label} the relation between $\tilde{f}_{\ell}$ and $h(\tilde{e})$ depending on the orientation of the triangle and the colors of the boundary edges (the color and position of $\tilde{f}$ is bold). Those relations are consequences of Lemma \ref{lem:block_vertex} and blowups of even degree vertices.

\vspace{0.5cm}
\begin{figure}[H]
\begin{tabular}{|c|c|c|c|c|}
\hline 
Coloring of a direct triangle & \begin{tikzpicture}

\draw (0+0*1/2,0*0.866)--(1+0*1/2,0*0.866)  node[midway,below] {$y$};
\draw (1+0*1/2,0*0.866)--(0+1*1/2,1*0.866) node[midway,right] {$z$};
\draw (0+1*1/2,1*0.866)--(0+0*1/2,0*0.866)node[midway,left] {$\mathbf{x}$};
\node at (2.5,0.2) {$x,y,z\not=m$}; 
\end{tikzpicture}& \begin{tikzpicture}

\draw (0+0*1/2,0*0.866)--(1+0*1/2,0*0.866)  node[midway,below] {$0$};
\draw (1+0*1/2,0*0.866)--(0+1*1/2,1*0.866) node[midway,right] {$m$};
\draw (0+1*1/2,1*0.866)--(0+0*1/2,0*0.866)node[midway,left] {$\mathbf{1}$};

\end{tikzpicture} & \begin{tikzpicture}

\draw (0+0*1/2,0*0.866)--(1+0*1/2,0*0.866)  node[midway,below] {$m$};
\draw (1+0*1/2,0*0.866)--(0+1*1/2,1*0.866) node[midway,right] {$1$};
\draw (0+1*1/2,1*0.866)--(0+0*1/2,0*0.866)node[midway,left] {$\mathbf{0}$};

\end{tikzpicture} & \begin{tikzpicture}

\draw (0+0*1/2,0*0.866)--(1+0*1/2,0*0.866)  node[midway,below] {$1$};
\draw (1+0*1/2,0*0.866)--(0+1*1/2,1*0.866) node[midway,right] {$0$};
\draw (0+1*1/2,1*0.866)--(0+0*1/2,0*0.866)node[midway,left] {$\mathbf{m}$};

\end{tikzpicture} \\ 
\hline 
$\tilde{f}_{\ell}$ & $h(\tilde{e})+1$ & $h(\tilde{e})+2$ & $h(\tilde{e})+1$ & $h(\tilde{e})+1$ \\ 
\hline 
\end{tabular} 

\vspace{0.5cm}

\begin{tabular}{|c|c|c|c|c|}
\hline 
Coloring of a reversed triangle & \begin{tikzpicture}

\draw (0+0*1/2,-0*0.866)--(1+0*1/2,-0*0.866)  node[midway,above] {$y$};
\draw (1+0*1/2,-0*0.866)--(0+1*1/2,-1*0.866) node[midway,right] {$\mathbf{x}$};
\draw (0+1*1/2,-1*0.866)--(0+0*1/2,-0*0.866)node[midway,left] {$z$};
\node at (2.5,-0.2) {$x,y,z\not=m$}; 
\end{tikzpicture}& \begin{tikzpicture}

\draw (0+0*1/2,-0*0.866)--(1+0*1/2,-0*0.866)  node[midway,above] {$0$};
\draw (1+0*1/2,-0*0.866)--(0+1*1/2,-1*0.866) node[midway,right] {$\mathbf{1}$};
\draw (0+1*1/2,-1*0.866)--(0+0*1/2,-0*0.866)node[midway,left] {$m$};

\end{tikzpicture} & \begin{tikzpicture}

\draw (0+0*1/2,-0*0.866)--(1+0*1/2,-0*0.866)  node[midway,above] {$1$};
\draw (1+0*1/2,-0*0.866)--(0+1*1/2,-1*0.866) node[midway,right] {$\mathbf{m}$};
\draw (0+1*1/2,-1*0.866)--(0+0*1/2,-0*0.866)node[midway,left] {$0$};

\end{tikzpicture} & \begin{tikzpicture}

\draw (0+0*1/2,-0*0.866)--(1+0*1/2,-0*0.866)  node[midway,above] {$m$};
\draw (1+0*1/2,-0*0.866)--(0+1*1/2,-1*0.866) node[midway,right] {$\mathbf{0}$};
\draw (0+1*1/2,-1*0.866)--(0+0*1/2,-0*0.866)node[midway,left] {$1$};

\end{tikzpicture} \\ 
\hline 
$\tilde{f}_{\ell}$ & $h(\tilde{e})+1$ & $h(\tilde{e})$ & $h(\tilde{e})+1$ & $h(\tilde{e})$ \\ 
\hline 
\end{tabular} 
\caption{Coordinates of an edge in function of the coloring and height of the next edge in a triangle ($\tilde{f}$ correspond to the bold edge and $\tilde{e}$ corresponds to the horizontal edge).\label{Fig:coordinates_label}}
\end{figure}

\noindent
From those relation and the fact that $\tilde{f}_\ell\geq \tilde{f}'_{\ell}$, we deduce that $h(\tilde{e})\geq h(\tilde{e}')$, i.e $L(e) \geq L(e')$, in the case $e'_{\ell+1}=e_{\ell+1}+1$ and that $h(\tilde{e})> h(\tilde{e}')$, i.e $L(e) > L(e')$, if one of the boundary edge of $s$ different from $e$ is colored $m$. The case $e_{\ell}=e'_{\ell}+1$ is done similarly, yielding always $h(\tilde{e})> h(\tilde{e}')$, i.e $L(e) > L(e')$. \\
\noindent
Finally, if the middle edge is coming from the blowing up of an even vertex and is not colored $m$, then this edge is necessarily colored $1$, and thus $e$ and $e'$ are colored $0$ and the opposite edges of their lozenge are colored $1$. The strict inequality is directly deduced from construction of $L = h \circ \zeta^{-1}$ and Lemma \ref{lem:block_vertex} giving the height of edges colored $0$ in an even vertex.

Hence, $T_k$ with the labelling $(C,L)$ is a genuine discrete two-colored dual hive, which we denote by $\zeta(P)$. The boundary values are directly deduced from the boundary of $P$, so that the resulting hive is in $\mathcal{H}(c(\omega),l(\omega),N)$.
\end{proof}
\noindent

\begin{proof}[Proof of Theorem \ref{thm:bijection_puzzle_dual_hive}]
Let us construct the reverse bijection. Let $H = (C, L)$ be a two-colored dual hive in $\mathcal{H}(c(\omega),l(\omega),N)$. We first define the candidate vertex partition (without the coloring for now) $\mathcal{P}_v$ as follows :
\begin{itemize}
\item for each triangle face $t=(t^0,t^1,t^2)\in T_k$, with $t^\ell$ of type $\ell$, with boundary colors in $(0,0,0)$, $(1,1,1)$ or $(1,0,3)$ (up to a cyclic order), we define a block $B_t\subset E_N$ with edges $e^0,e^1,e^0$, with $e^\ell$ of type $\ell$, and such that :
$$h(e^\ell)=L(t^\ell),\, e_{\ell+1}^\ell=L(t^{\ell+1})+1.$$
\item for each long rhombus with boundary $u=(s^0,t^1,\ldots,t^r,s^1,t^{r+1},\ldots,t^{2r})$ with $s^i$ of type $\ell$ and $t^i$ of type $\ell+1$ and $t^1_i=s^0_i+1$ such that $C(s^0)=C(s^1)=1$, $C(t^i)=0$, $L(t^i)=L(t^{2r+1-i})=L(t^1)-(i-1)$ for $1\leq i\leq 2r$, and which is not included in an other rhombus satisfying such property, we define a block $B_u\subset E_N$ with edges $(e^0,e^1,f^1,\ldots,f^{2r})$ with $h(v)=L(v)$ for $v\in B_u$ and 
$$e^0_{\ell+1}=L(t^1)+2,\,e^1_{\ell+1}=L(t^{r}), \,f^i_{\ell+2}=f^{2r+1-i}_{\ell+2}-1=e^0_{\ell+2}+i,\, 1\leq i\leq r.$$
Remark that for 
$1\leq i\leq r$, $f^{i}_{\ell+2}=f^{2r+1-i}_{\ell+2}-1=t(g^i)$, where $g^i$ is the edge of type $\ell+2$ colored $m$ adjacent to $f^i$ or $f^{2r+1-i}$.
\item for each edge $s$ of type $\ell$ on the boundary of $T_k$, we define a singleton in $B_e\subset E_N$ consisting of the unique edge $e$ of type $\ell$ with $h(e)=L(s)$.
\end{itemize}

 Moreover, by Lemma \ref{lem:two_color_non_crossing_same_type} below, edges $e,e'$ coming from different edges $t,t'\in T_k$ by the previous constructions are distinct. Let $\mathcal{E}=\bigcup_{B\in \mathcal{P}_v} B$. We can thus define a coloring $c$ on $\mathcal{E}$ by setting $c(e)=C(t)$ when $e$ is constructed from $t$ above. By construction and the property $(1)$ of Definition \ref{def:two_col_hives}, the covering $\mathcal{P}_v$ satisfies all the properties of Lemma \ref{lem:block_vertex}. With the above constructions and the conditions (1) and (2.a) of a two-color dual hive, we can then check that all the relations from Figure \ref{Fig:coordinates_label} is still satisfied when the bold edge is an element of $\mathcal{E}$.

Define then a relation $\sim$ on $\mathcal{E}$ by saying that $e\sim e'$ if $e,e'$ are coming from a same edge of $T_k$ through the previous construction, and denote by $\mathcal{P}_e$ the set partition coming from this relation. By the properties of a two-colored dual hive, any long rhombus considered before of border edges of type $\ell,\ell+1$ has its inner middle edges of type $\ell-1$ colored $m$, so that none of the triangles inside this long rhombus yields block of $\mathcal{P}_v$ through the first step. Hence, each edge of $T_k$ yields at most $2$ edges of $\mathcal{E}$. Remark that an edge of $T_k$ is either adjacent to two faces or to one face and the boundary of $T_k$, so that $\mathcal{P}_e$ consists of pairs or singletons, and in the latter case the singleton belongs to two blocks of $\mathcal{P}_v$. If $e\sim e'$, by the above construction $c(e)=c(e'), L(e)=L(e')$ and $e,e'$ have the same type, so that $\mathcal{P}_e$ only has admissible pairs (which can be reduced to a singleton).
If $e\not= e'$ belong to a same block of $\mathcal{P}_v$ they come from different edges of $T_k$ and thus $\{e,e'\}\not\in \mathcal{P}_e$.

In view of applying Proposition \ref{prop:puzzle_from_partition}, it suffices to prove that two pairs of $\mathcal{P}_e$ do not cross.  Suppose that $B=\{e^1,e^2\}$ and $B'=\{e^3,e^3\}$ are two blocks of $\mathcal{P}_{e}$. If they are of same type $\ell$, then Lemma \ref{lem:two_color_non_crossing_same_type} yields that $B\cap S_{B'}=B'\cap S_{B}=\emptyset$. If $B$ are of different type $\ell$ and $\ell+1$, Lemma \ref{lem:two_color_non_crossing_different_type} yields that the second condition of crossing strips is never satisfied, and the first condition may only be satisfied in the case (4) of Lemma \ref{lem:two_color_non_crossing_different_type} where $t'_{\ell}\geq t_{\ell}$ and $t'_{\ell+1}\geq t_{\ell+1}$, when $e^2_{\ell+1}=e^3_{\ell+1}$. But in the latter case, by Definition \ref{def:strip}, edges of type $\ell+1$ of the strip $S_{B}$ have $\ell+1$-coordinate strictly smaller than $e^2_{\ell+1}$, so that $B'\cap S_B=\emptyset$.  Likewise, edges of the strip $S_{B'}$ of type $\ell$ have $\ell+1$-coordinate strictly larger than $\min(e^3_{\ell+1},e^4_{\ell+1})$ so that $B\cap S_{B'}=\emptyset$. Hence, $S_B$ and $S_{B'}$ do not cross.

Pairs of $\mathcal{P}_e$ are admissible and any two different pairs $B,B'\in \mathcal{P}_e$ do not cross, thus partition $\mathcal{P}_e$ satisfies the properties of Lemma \ref{lem:condition_pairing}. Finally, by Proposition \ref{prop:puzzle_from_partition} applied to $(\mathcal{P}_v,\mathcal{P}_e)$, there exists a unique puzzle $P$ such that the corresponding vertex and edge partitions are respectively $\mathcal{P}_v$ and $\mathcal{P}_e$. Denote by $\chi(H)$ this puzzle. It is clear from the above constructions that $\chi(H)\in P(\omega)$ and that $\chi\circ \zeta$ and $\zeta\circ \chi$ are respectively identity maps of $P(\omega)$ and $\mathcal{H}(c(\omega),l(\omega),N)$.
\end{proof}

\begin{lemma}[Same type blocks do not cross]
\label{lem:two_color_non_crossing_same_type}
Let $t\not=t'\in T_k$ of same type $\ell$ in blocks of $\mathcal{P}_v$, and suppose without loss of generality that $t_{\ell+1}> t'_{\ell+1}$ or $t_{\ell+1}=t'_{\ell+1}$ and $h(t')>h(t)$. Then, if $h(t)< h(t')$ and $t_{\ell+1}>t'_{\ell+1}$, the edges $e^1,e^2$ (resp. $e^3,e^4$) of $T_N$ associated to $t$ (resp. $t'$) satisfy 
$$h(e^1)=h(e^2)<h(e^3)=h(e^4),$$
if $h(t)< h(t')$ and $t_{\ell+1}=t'_{\ell+1}$,
$$\min(e^1_{\ell-1}, e^2_{\ell-1})>\max(e^{3}_{\ell-1}, e^4_{\ell-1}).$$
and if $h(t)\geq h(t')$ and $t_{\ell+1}> t'_{\ell+1}$,
$$\min(e^1_{\ell+1},e^2_{\ell+1})>\max( e^3_{\ell+1},e^4_{\ell+1}).$$
\end{lemma}

\begin{proof}
Since $t$ and $t'$ are in blocks of $\mathcal{P}_v$, neither $t$ nor $t'$ are colored $m$. If $h(t)<h(t')$ and $t_{\ell+1}> t'_{\ell+1}$, $L(t)<L(t')$ by Remark \ref{rem:strict_increasing_label}. Since $h(e^1)=h(e^2)=L(t)$ and $h(e^3)=h(e^4)=L(t')$, this implies
$$h(e^1)=h(e^2)<h(e^3)=h(e^4).$$
If $t_{\ell+1}=t'_{\ell+1}$ and $h(t')>h(t)$, then by Condition $(2.c)$ of Definition \ref{def:two_col_hives}, $L(t')>L(t)$ except if the middle edge of all lozenges between $t$ and $t'$ are colored $m$. In the latter case, let $s$ (resp. $s'$)  be the edge of type $\ell-1$ such that $t,s$ form a reverse triangle (resp. $t',s'$ form a direct triangle). Since $s_{\ell}=t_{\ell}+1$ and $s'_{\ell}=t'_{\ell}+1$, the inequality $t'_{\ell}>t_{\ell}$ implies that $s'_{\ell}>s_{\ell}$. Similarly, since $s_{\ell-1}=t_{\ell-1}-1$ and $s'_{\ell-1}=t'_{\ell-1}$, we have 

$$s_{\ell-1}=t_{\ell-1}-1=N-t_{\ell}-t_{\ell+1}-1> N-t'_{\ell}-t'_{\ell+1} -1\geq  t'_{\ell-1}\geq s'_{\ell-1}.$$

Hence, $L(s')\leq L(s)$. Let us introduce the third edge $r$ (resp. $r'$) of the triangle with edges $s,t$ (resp. $s',t')$. By Figure \ref{Fig:coordinates_label} and Condition (1) from Definition \ref{def:two_col_hives}, we get that $e^3_{\ell-1}=N-e^3_{\ell}-e^3_{\ell+1}=N-L(t')-L(r')-1=L(s')$ and $e^2_{\ell-1}=N-e^2_{\ell}-e^2_{\ell+1}=N-L(t)-L(r)=L(s)+1$ and thus 
$e^2_{\ell-1}=L(s)+1>L(s')=e^3_{\ell-1}$, so that 
$$e^1_{\ell-1}\geq e^2_{\ell-1}>e^{3}_{\ell-1}\geq e^4_{\ell-1}.$$ \\
\noindent
If $h(t)\geq h(t')$ and $t_{\ell+1}> t'_{\ell+1}$, then $t'_{\ell+2}>t_{\ell+2}$. Suppose without loss of generality that $e^1_{\ell+1}\geq e^{2}_{\ell+1}$ and $e^3_{\ell+1}\geq e^4_{\ell+1}$. Let us consider the edges $s,s'$ of type $\ell+1$ such that $(t,s,u)$ and $(t',s',u')$ are respectively direct and reverse triangles of $T_k$ so that the corresponding edge of $t$ and the piece containing the direct triangle is $e^2$ and the corresponding edge for $t'$ and the reverse triangle is $e^3$. Since $h(s)=t_{\ell+1}-1$ and $s_{\ell+2}=t_{\ell+2}+1$ and $h(s')=t'_{\ell+1}-1$ and $s'_{\ell+2}=t'_{\ell+2}$, $h(s)>h(s')$ and $s'_{\ell+2}\geq s_{\ell+2}$, so that $L(s)>L(s')$ by Remark \ref{rem:strict_increasing_label}. Then, since $(t,u,s)$ is a direct triangle, Figure \ref{Fig:coordinates_label},
$e^2_{\ell+1}=L(s)+1$, except if $c(t)=1, c(s)=0$ and $c(u)=m$ where $e^2_{\ell+1}=L(s)+2$. Likewise, since $(t',s',u')$ is a reverse triangle, $e^3_{\ell+1}=L(s')+1$ except if $c(t)=1, c(s)=0,c(u)=m$ or $c(t)=0,c(s)=m,c(u)=1$ where $e^3_{\ell+1}=L(s')$. Hence, in any case,
$$e^3_{\ell+1}\leq L(s')+1<L(s)+1\leq e^2_{\ell+1}$$
and
$$e^1_{\ell+1},e^2_{\ell+1}> e^3_{\ell+1},e^4_{\ell+1}.$$

\end{proof}

\begin{lemma}[Different type blocks do not cross]
\label{lem:two_color_non_crossing_different_type}
Let $t,t'\in T_k$ be of respective type $\ell,\ell+1$ yielding edges in $\mathcal{E}$, and denote by $e^1,e^2$ (resp. $e^3,e^4$) the edges of $T_N$ corresponding to $t$ (resp. $t'$). Then, 
\begin{enumerate}
\item if $t'_{\ell}>t_{\ell}$ and $t'_{\ell+1}< t_{\ell+1}$, then 
$$e^1_{\ell}=e^2_{\ell}< e^3_{\ell}\leq e^4_{\ell}.$$
\item if $t'_{\ell}< t_{\ell}$ and $t'_{\ell+1}\geq t_{\ell+1}$, then 
$$e^3_{\ell}\leq e^4_{\ell}< e^1_{\ell}=e^2_{\ell}.$$
\item if $t'_{\ell}\leq t_{\ell}$ and $t'_{\ell+1}<t_{\ell+1}$,
$$e^3_{\ell+1}\leq e^4_{\ell+1}< e^1_{\ell+1}\leq e^2_{\ell+1}.$$
\item if $t'_{\ell}\geq t_{\ell}$ and $t'_{\ell+1}\geq t_{\ell+1}$,
$$e^1_{\ell+1}\leq e^2_{\ell+1}\leq e^3_{\ell+1}=e^4_{\ell+1}.$$
\end{enumerate}
\end{lemma}
\begin{proof}
The proof of the four assertions are similar.

\begin{enumerate}
\item  Suppose that $t'_{\ell}>t_{\ell}$ and $t'_{\ell+1}<t_{\ell+1}$. Let $(t',s',r')$ and $(t',s'',r'')$ be the reverse and direct triangles belonging to pieces yielding respectively $e^3$ and $e^4$, with $s',s''$ of type $\ell$. Since then $e^4_{\ell}\geq e^3_{\ell}$, it suffices to show that $e^3_{\ell}>e^2_{\ell}$. Since $s'_{\ell}=t'_{\ell}-1$ and $s'_{\ell+1}=t'_{\ell+1}+1$, we have $s'_{\ell}\geq t_{\ell}$ and $t_{\ell+1}\geq s'_{\ell+1}$, so that $L(s')\geq L(t)$. Since $t'$ is of type $\ell+1$ not and colored $m$ and $(t',s',r')$ is a reverse triangle, Figure \ref{Fig:coordinates_label} and Condition (1) from Definition \ref{def:two_col_hives} yield that either $e^3_{\ell+2}= L(r')+1$ and $L(r')+L(s')+L(t')=N-2$ or $e^3_{\ell+2}= L(r')$ and $L(r')+L(s')+L(t')=N-1$. In any case, $e^3_{\ell+2}=N-1-L(t')-L(s')$, so that
$$e^3_{\ell}=N-e^3_{\ell+1}-e^3_{\ell+2}= N-L(t')-(N-1-L(t')-L(s'))=L(s')+1> L(t)=e^2_{\ell}=e^1_{\ell}.$$
\item Suppose that $t'_{\ell}<t_{\ell}$ and $t'_{\ell+1}\geq t_{\ell+1}$, and let $s'$ be the edge of type $\ell$ such that $(t',s',r')$ is a direct triangle. Since $s'_{\ell}=t'_{\ell}$ and $s'_{\ell+1}=t'_{\ell+1}+1$, $s'_{\ell}<t_{\ell}$ and $s'_{\ell+1}>t_{\ell+1}$, so that $L(t)>L(s')$ by Remark \ref{rem:strict_increasing_label}. Since $(t',s',r')$ is a direct triangle, Figure \ref{Fig:coordinates_label} and Condition (1) from Definition \ref{def:two_col_hives} yield by a same reasoning as above that $e^4_{\ell+2}\geq L(r')+1=N-L(s')-L(t')$, so that, using that $e^4_{\ell+1}=L(t')$, 
$$e^4_{\ell}=N-e^4_{\ell+1}-e^4_{\ell+2}\leq L(s')<L(t)=e^1_{\ell}=e^2_{\ell}.$$
\item Suppose that $t'_{\ell}\leq t_{\ell}$ and $t'_{\ell+1}<t_{\ell+1}$, and let $s$ be the edge of type $\ell+1$ such that $(t,s)$ is part of a direct triangle. Then, $s_{\ell}=t_{\ell}$ and $s_{\ell+1}=t_{\ell+1}-1$, so that $s_{\ell}\geq t'_{\ell}$ and $s_{\ell+1}\geq t'_{\ell+1}$. We deduce that $s_{\ell+2}\leq t'_{\ell+2}$, and thus $L(s)\geq L(t')$. Since $e^1_{\ell+1}\geq L(s)+1$ by Figure \ref{Fig:coordinates_label}, we thus have 
$$e^3_{\ell+1}=e^4_{\ell+1}=L(t')\leq L(s)<e^1_{\ell+1}\leq e_{\ell+1}.$$
\item Suppose that $t'_{\ell}\geq t_{\ell}$ and $t'_{\ell+1}\geq t_{\ell+1}$, and let $s$ be the edge of type $\ell+1$ such that $(t,s,r)$ is a reverse triangle. Then, $s_{\ell}=t_{\ell}+1$ and $s_{\ell+1}=t_{\ell+1}-1$. Hence, $t'_{\ell+1}>s_{\ell+1}$ and $t'_{\ell+2}=N-t'_{\ell+1}-t'_{\ell}\leq N-s_{\ell+1}-1-s_{\ell}+1\leq s_{\ell+2}$ and the inequality is strict except when $t'_{\ell}=t_{\ell}$ and $t'_{\ell+1}=s_{\ell+1}$. Hence, by Remark \ref{rem:strict_increasing_label} in the case of strict inequality and Condition (2.b) and (2.c) from Definition \ref{def:two_col_hives}, $L(t')>L(s)$, except when $t'_{\ell}=t_{\ell}, t'_{\ell+1}=s_{\ell+1}$ and $C(r)=m$, in which case $L(t')=L(s)$. In the first case, by Figure \ref{Fig:coordinates_label} we have $e^2_{\ell+1}\leq L(s)+1\leq L(t')$. In the second case, since $C(r)=m$ we have $e^2_{\ell+1}=L(s)\leq L(t')$, so that in any case 
$$e^1_{\ell+1}\leq e^2_{\ell+1}\leq e^3_{\ell+1}= e^4_{\ell+1}.$$
\end{enumerate}
\end{proof}

\section{Color swap}\label{Sec:color_swap}
\label{sec:color_swap}

The goal of this combinatorial section is to exhibit a convex body of dimension $D = \frac{(n-1)(n-2)}{2}$ having integer points counted by quantum Littlewood-Richardson coefficients, so that the limit expression \eqref{eq:density_limit_quantum_coefs} converges to the volume of a polytope. 
From this section to the end of the paper, we set $k = n+d$ and $\xi = \ed^{\frac{i\pi}{3}}$. We also assume that the color maps of dual hives are regular in the sense of Definition \ref{def:reg_boundary}. 

\begin{definition}[Regular boundaries]
\label{def:reg_boundary}
    A color map $C: E_k \rightarrow \{ 0,1,3,m \}$ is regular (or has regular boundaries) if for every $i \in \{ 0,1,2 \}$,
    \begin{equation}
        c^{(i)} = (\underset{d\text{ times}}{\underbrace{1,\ldots,1}},\underset{n-d\text{ times}}{\underbrace{0,\ldots,0}},\underset{d\text{ times}}{\underbrace{1,\ldots,1}}).
    \end{equation}
    Moreover, we say that a dual hive $H = (C, L)$ is regular if its color map $C$ is.
\end{definition}

\subsection{Arrows and hexagons}

Let us start with some definitions on local configurations of edges in $E_k$.

\begin{definition}[Openings and arrows]

    Let $x \in T_k$. An opening of type $l \in \{ 0,1,2 \}$ at $x$ is the pair of edges $(e, e') \in E_k^2$ such that 
    \begin{align*}
        \{ e, e' \} &= \{ (x, x + \xi^{1+l}), (x, x + \xi^{5+l}) \} \text{ or } \{ e, e' \} = \{ (x + \xi^{2+l}, x), (x + \xi^{4+l}, x) \} \\
        C(e) &= C(e') \in \{ 0,1 \}. 
    \end{align*}
    The color of the opening is defined as the color of edges $e$ and $e'$.
\end{definition}

\noindent
Consider an opening $a = (e, e')$ at $x$ of type $l$ and color $c \in \{ 0,1 \}$. Let $e'' = e''(a)$ be the edge such that $e, e'$ are edges of the lozenge with middle edge $e''$. The only possible colors of the edge $e''$ are $C(e'') \in \{ 0,1 \}$. If $C(e'') = c$, faces of the lozenge with middle edge $e''$ have all of their edges colored $c$. If $C(e'') \neq c$, then there is an opening $a'$ of type $l$ and color $c$ at the other endpoint of $e''$. Note that there can only be finitely many such openings $r \geq 0$ before $C(e'') = c$.

\begin{definition}[Arrow]
    Let $a = (e, e')$ be an opening of type $l$ and color $c$. Let $ r \geq 0$ be the number of successive openings having middle edge $e''$ such that $C(e'') \neq c$ with $C(e'') \in \{ 0,1 \}$ as in the previous paragraph. We say that the configuration of edges consisting of the $r \geq 0$ successive pairs of $3$ and $m$ lozenges together with the pair of direct and reverse faces with boundary edges of color $c$ is an arrow of length $r \geq 0$ at the opening $a$. 
\end{definition}

\noindent
See Figure \ref{fig:arrows} for examples of openings and arrows.
\begin{figure}[H]
    \centering
    \includegraphics[scale=0.7]{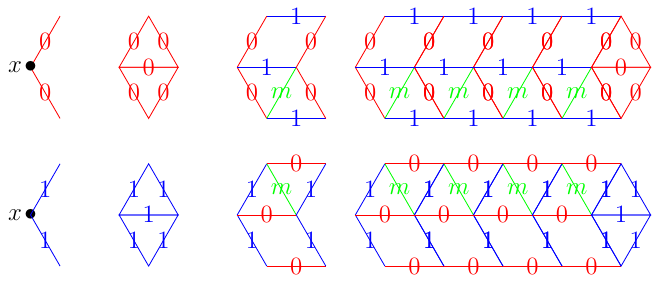}
    \caption{First row from left to right : an opening $a$ with color $0$ and type $0$ at $x$, the case $C(e'') = c$, the case $C(e'') \neq c$ and an arrow of length $r=4$. The second row is the analog for color $1$. Uncolored edges have color $3$.}
    \label{fig:arrows}
\end{figure}

\begin{definition}[ABC hexagons]
\label{def:hexagons}
    Let $C$ be a color map and let $h$ be a hexagon, that is, the union of six triangular faces sharing one vertex in $T_k$. We say that $h$ is an ABC hexagon (for the color map $C$) if the color map $C$ restricted to $h$ is any of the three configurations in Figure \ref{fig:hexagons} up to a rotation.
    \\
    A rotation of an ABC hexagon $h$ is the replacement of the values of $C$ by the ones obtained from a rotation of $h$ which preserves the value of $C$ on the boundary $\partial h \cap E_k$.

\end{definition}

\begin{figure}[H]
    \centering
    \includegraphics[scale=1]{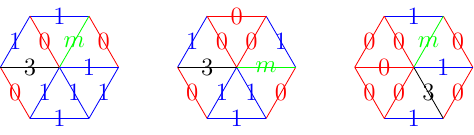}
    \caption{The three types of hexagons : A (left), B (center) and C (right).}
    \label{fig:hexagons}
\end{figure}
\noindent
Note that type B has three possible rotations whereas A and C only have two. For a hexagon $h$, denote $E_h$ the edges of $E_k$ which are in $h$. \\
\\
Let $A$ be an arrow of length $r \geq 1$ and type $l$ at an opening with center $x$. The center of the last opening is $y \in \{ x + r\xi^l, x - r \xi^l \}$. Notice that if the color $c$ of $A$ is $0$ (resp. $1$), then the hexagon $h(y)$ with center $y$ is of type $C$ (resp.) $A$. Applying a rotation to $h(y)$ yields an arrow $A'$ of length $r-1$ of type $l$ at the same opening with center $x$. By applying hexagon rotations to $ x + r\xi^l,  x + (r-1)\xi^l, \dots,  x + \xi^l$ (or $ x - r\xi^l,  x - (r-1)\xi^l, \dots,  x - \xi^l$) in this order, one gets an arrow $R(A)$ of length $r$ of type $l$ the same opening with center $x + r\xi^l$. We call this sequence of $r$ hexagon rotations the reversal of the arrow $A$. An example of arrow reversal is given in Figure \ref{fig:arrow_reversal}

\begin{figure}[H]
    \centering
    \includegraphics{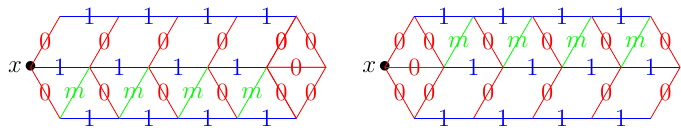}
    \caption{Reversal of an arrow of length $4$ at $x$. Uncolored edges have color $3$.}
    \label{fig:arrow_reversal}
\end{figure}

\subsection{Gash propagation}

\noindent
We define a local configuration called a gash in Definition \ref{def:gash} which one can propagate in a dual hive. Local propagation rules are given in Definition \ref{def:gash_propag} and the general propagation algorithm on dual hives is defined in Definition \ref{def:propagation_algorithm}. The goal of the propagation algorithm is to find rigid lozenges of a dual hive in view of the next section.

\begin{definition}[Gash]
\label{def:gash}
    Let $x \in T_k$. A gash $g$ with center $x = x(g)$ is the union of the two edges $(x, x - \xi^{2l}), (x + \xi^{2l}, x)$ for $l \in \{0,1,2 \}$ such that
    \begin{align*}
        C((x, x - \xi^{2l})) = 1, \ C((x + \xi^{2l}, x)) = 0 &\text{ if } l \in \{ 0,1 \} \\
        C((x, x - \xi^{2l})) = 0, \ C((x + \xi^{2l}, x)) = 1 &\text{ if } l = 2.
    \end{align*}
    The type of a gash denoted $t(g)$ is defined as the type $l \in \{0,1,2 \}$ of its edges.
\end{definition}

\noindent
Note that this definition only depends on the color map $C$ of $H$. Let $g$ be a gash. There are only six possible configurations given in Figure \ref{fig:gash_completions} adjacent to $g$. In this section, we show that such a gash $g$ can be moved across the color map $C$ using local moves until reaching configuration $(v)$ or $(vi)$ of Figure \ref{fig:gash_completions}.
\begin{figure}[H]
    \centering
    \includegraphics[scale=0.8]{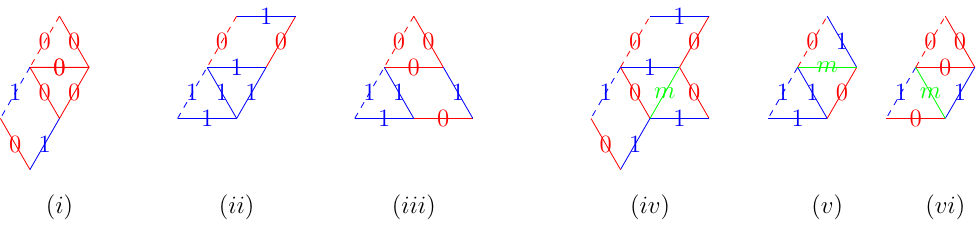}
    \caption{The six possible adjacent configurations to a gash of type $2$ shown in dashed edges. The same holds for a gash of other types up to rotations.}
    \label{fig:gash_completions}
\end{figure}

\begin{definition}[Gash propagation]
\label{def:gash_propag}
    Let $g$ be gash of type $l$ with center $x$. 
    \begin{enumerate}
        \item Suppose that $g$ is adjacent to a configuration $(i)$. Let $y = x+ \xi^4$, (resp. $y=x + 1$, $y=x + \xi^5$) if $l= 0$, (resp. $l=1, l=2$). We call the propagation of $g$ the gash $g'$ of type $l$ at center $y$. 
        \item Suppose that $g$ is adjacent to a configuration $(ii)$. Let $y = x+ \xi^5$, (resp. $y = x + \xi, y = x + 1$) if $l= 0$, (resp. $l=1, l=2$). We call the propagation of $g$ the gash $g'$ of type $l$ at center $y$. 
        \item Suppose that $g$ is adjacent to a configuration $(iii)$. Let $y = x+1$ if $l= 1$ or $l=2$. We call the propagation of $g$ the gash $g'$ of type $3-l$ at center $y$. 
        \item Suppose that $g$ is adjacent to a configuration $(iv)$. Notice that there is a $0$ opening at $x$ and thus an arrow of color $0$ at $x$ with type $l+1$. Reverting this arrow yields a configuration $(i)$ adjacent to $g$ and we define the propagation of $g$ to be the gash $g'$ of type $l$ as in step $(1)$. See Figure \ref{fig:prop_iv} for an illustration.
    \end{enumerate}
\end{definition}

\begin{figure}[H]
    \centering
    \includegraphics[scale=0.8]{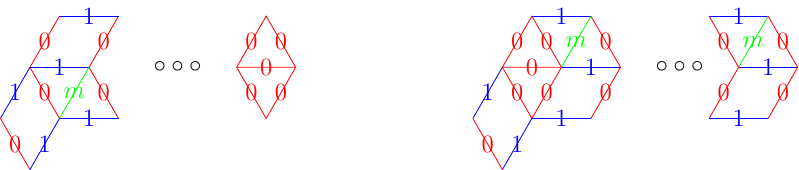}
    \caption{Propagation of a gash $g$ adjacent to a configuration $(iv)$. The arrow of color $0$ has been reversed yielding a configuration $(i)$ adjacent to $g$.}
    \label{fig:prop_iv}
\end{figure}

\noindent
Remark that the only propagation in which the type of the gash changes is $(iii)$.
We now give a general procedure using local propagations from Definition \ref{def:gash_propag}. This procedure starts from a gash and propagates it until reaching a configuration that is either $(v)$ or $(vi)$. 

\begin{definition}[Propagation algorithm]
\label{def:propagation_algorithm}
    Define the following algorithm.
    \begin{enumerate}
    \item[\textbf{Input:}] A color map $C$ and a gash $g$ of type $l \in \{1,2\}$.
    \item Set $g^{(0)} = g$, $x^{(0)} = x(g)$, $t^{(0)} = t(g)$.
        \item \textbf{WHILE} $g^{(s)}$ is adjacent to $(i), (ii), (iii)$ or $(iv)$: set $g^{(s+1)}$ to be the propagation of $g^{(s)}$ with center $x^{(s+1)}$ and type $t^{(s+1)}$.
    \end{enumerate}
\end{definition}

\begin{proposition}[Propagation algorithm is correct]
\label{prop:gash_propagation}
    Let $g$ be a gash of type $2$ in $T_k$. The propagation algorithm terminates at a gash $\Tilde{g}$ adjacent to configuration of type $(v)$ or $(vi)$.
\end{proposition}

\noindent
For the proof of Proposition \ref{prop:gash_propagation}, we need Lemma \ref{lem:regular_triangle} which shows that any triangular region having two of its sides with edges colored $c \in \{0,1 \}$ has all its edges colored $c$. 

\begin{lemma}[Regular equilateral triangles]
\label{lem:regular_triangle}
    Let $C: E_k \rightarrow \{ 0, 1, 3, m \}$ be a color map on edges of $T_k$. Let $R$ be any subset of edges of $E_k$ such that $\partial R$ is an equilateral triangle of size $s \geq 1$. Let $c \in \{ 0,1 \}$ and assume that two boundaries of $R$ have edges $e$ which are all colored $c$. Then, every edge in $R$ is colored $c$. 
\end{lemma}

\begin{proof}
    Let us first show a general fact about a shape described below that we call a cup. For $r \geq 1$, we call a cup of length $r$ and type $i$ the union of $r$ consecutive type $i$ edges together with one edge of type $i+1$, respectively of type $i-1$, forming an angle of $\frac{2 \pi}{3}$ with type $i$ edge with maximal and minimal heights. See Figure \ref{fig:cup} for an example. Suppose that edges of a cup are all colored $c \in \{ 0,1 \}$. Let us show by induction on $r$ that the only possible color of edges in the convex hull of the cup is $c$. 
    \begin{figure}[H]
        \centering
        \includegraphics{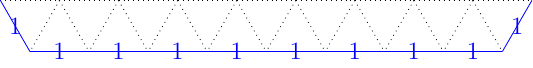}
        \caption{A cup of length $r=8$ and type $0$ with color $c=1$. Edges in the convex hull of the cup are dotted}
        \label{fig:cup}
    \end{figure}
    \noindent
    Two adjacent edges of the cup of different types form an opening of color $c$. One checks that the arrow at this opening has length zero for otherwise all the $r$ edges of type $i$ would belong to some non-rigid lozenges with opposite edges of type $i$ (resp. $i-1$) colored $c$ (resp. $1-c$) which is incompatible with the color $c$ of the other opening of the cup. Thus, the arrow has zero length. Then, the rest of the convex hull is a cup of length $r-1$ which completes the proof of cup completion by induction. \\
    \\
    Let us consider the region $R$ of the statement, and suppose that the boundary of type $i-1$ and $i+1$ of $R$ are colored $c$. The corner of the triangular face between boundaries of type $i-1$ and $i+1$ has all of its edges colored $c$, as two of them lie on the boundary of $R$. This induces a cup of length $1$ and type $i$ having edges of color $c$. By the previous reasoning, its only color completion consists of edges of color $c$. Each completion of a cup of size $r$ with $r<s$ yields a cup of size $r+1$ with edges colored $c$. Filling cups of sizes $1, 2, \dots, s-1$ with edges of color $c$ fills then $R$ with edges colored $c$ and thus proves the claim. 
\end{proof}

\begin{corollary}[Corners of regular boundaries]
\label{cor:regular_corner}
    Let $C: E_k \rightarrow \{ 0,1,3,m \}$ be a regular color map. Then, the three equilateral regions of side length $d$ each containing an extremal vertex of $T_k$ have all of their edges colored $1$.
\end{corollary}

\begin{proof}
    Notice that the regularity of the color map $C$ implies that each such triangular region of side length $d$ has two boundaries which lie on $\partial T_k^{(0)} \cup \partial T_k^{(1)} \cup \partial T_k^{(2)}$ having edges colored $1$. Applying Lemma \ref{lem:regular_triangle} yields the result.
\end{proof}

\begin{proof}[Proof of Proposition \ref{prop:gash_propagation}]

    At each step of the gash propagation, one has that $ x^{(s+1)}_0 < x^{(s)}_0$ or $x^{(s+1)}_1 > x^{(s)}_1$ and $t^{(s)} \in \{ 1,2 \}$ which implies that the while loop terminates on a last gash $g^{(\infty)}$. Assume for the sake of contradiction that $g^{(\infty)}$ is not adjacent to a configuration $(v)$ or $(vi)$. Since $g^{(\infty)}$ is the last gash, it is of type $1$ with center $x^{(\infty)} \in \partial T_k^{(1)}$. As the boundary $\partial T_k^{(1)}$ is regular, the edge of color $1$ in $g^{(\infty)}$ is the edge with height $n$ so that $x^{(\infty)} _1 = n$, see Figure \ref{fig:last_i}. Moreover, by Corollary \ref{cor:regular_corner}, the equilateral triangle $R_d \subset E_{k}$ of length $d$ having one of its boundaries between $x^{(\infty)}$ and $(n, 0)$ has all of its edges colored $1$. \\
    \\
    Consider the last $(iii)$ configuration encountered before reaching $\partial T_k^{(1)}$ having a gash $g'$ of type $1$. Note that such a configuration exists as all other configurations preserve the type of the gash during propagation of Definition \ref{def:gash_propag}. Consider the last gash $g$ with center $x$ resulting from a propagation of type $(i)$ after $g'$ with the convention that $g = g'$ if no configuration $(i)$ or $(iv)$ happen after $g'$ (recall that configuration $(iv)$ reduces to a propagation of type $(i)$, see step $(4)$ of Definition \ref{def:gash_propag}). Propagations $(g^{(s)}, s \geq 0)$ after $g = g^{(0)}$ are then of type $(ii)$ so that for $s \geq 0$, $x^{(s+1)}_1 = x^{(s)}_1$ and $x^{(s+1)}_0 = x^{(s)}_0 -1$. Since $x^{(\infty)} \in R_d$, we would have $x \in R_d$ and the edge $e' \in E_k$ with origin $x + \xi^5$ of type $0$ and color $0$ in the last $(i)$ configuration before $g$ (or in the last configuration $(iii)$ if no such configuration $(i)$ exists) would have both its endpoints in $R_d$ which contradicts the fact that edges in $R_d$ are all colored $1$. See Figure \ref{fig:last_i} for an illustration of the above argument.
\end{proof}

\begin{figure}[H]
    \centering
    \includegraphics{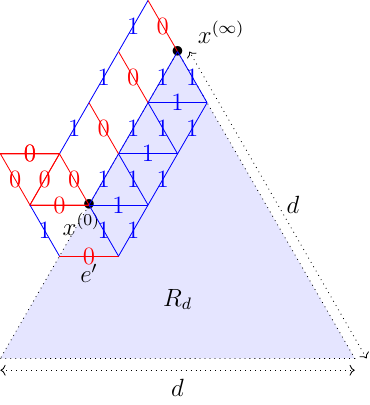}
    \caption{The $0$ colored edge $e'$ in the last configuration $(i)$ would be in $R_d$.}
    \label{fig:last_i}
\end{figure}

\begin{remark}[Type $1$ propagation algorithm]
\label{rmk:type_1_propagation}
    Note that the argument given in the proof of Proposition \ref{prop:gash_propagation} remains valid in the case of a type $1$ gash $g$ with center $x$ such that the type $0$ edge with origin $x + \xi^5$ is colored $0$. 
\end{remark}

\subsection{Color swap path}

In Proposition \ref{prop:gash_propagation}, we showed that any gash $g$ of type $2$ can be propagated to find a configuration $(v)$ or $(vi)$ having a rigid lozenge in it. In this section, we show that via hexagon rotations of Definition \ref{def:hexagons}, one can bring this rigid lozenge at the location of $x = x(g)$. 

\begin{lemma}[Back propagation]
\label{lem:back_propagation}
    Let $g$ be gash with center $x$ adjacent to a configuration $(i, ii)$ or $(iii)$ and suppose that its propagation $g'$ with center $x'$ is adjacent to a configuration $(v)$ or $(vi)$. Then, the hexagon $h'$ with center $x'$ is an ABC hexagon. Moreover, using a hexagon rotation of $h'$, the hexagon $h$ with center $x$ is an ABC hexagon.
\end{lemma}

\begin{proof}
    One can check that the statement of Lemma \ref{lem:back_propagation} holds in all possible cases of configurations, see Figure \ref{fig:hexagons}.
\end{proof}

\begin{proposition}[Gash reduction]
\label{prop:gash_reduction}
    Let $C$ be a color map and let $g$ be a gash of type $2$ in $C$. Let $\Tilde{g} = g^{(s)}$ for some $s \geq 0$ be the last gash as in Proposition \ref{prop:gash_propagation}. Using hexagon rotations of ABC hexagons with centers given by $x^{(s)}, x^{(s-1)}, \dots, x^{(0)} $ in this order, $C$ can be mapped to a color map $C'$ such that $g$ is a gash for $C'$ adjacent to a configuration $(v)$ or $(vi)$.
\end{proposition}

\begin{proof}
    Applying Lemma \ref{lem:back_propagation} to every center $x^{(s)}, x^{(s-1)}, \dots, x^{(1)} $ in this order yields the desired configuration.
\end{proof}

\begin{remark}[Reduction of a type $1$ gash]
\label{rem:type_1_reduction}
    As in Remark \ref{rmk:type_1_propagation}, the result of Proposition \ref{prop:gash_reduction} still holds if one considers $g$ of type $1$ with center $x$ such that the type $0$ edge with origin $x + \xi^5$ is colored $0$. 
\end{remark}

\subsection{Color map reduction}

Recall that we only consider regular boundary conditions for $T_k$ that is, the color map $C: E_k \rightarrow \{0,1,3,m \}$ is given by $1 \dots 1 0 \dots 0 1 \dots 1$ on every boundary of $T_k$, where there are $d$ ones on each side of the $n-d$ zeros. The goal of this section is to show that any regular color map can be mapped via hexagon rotations to the simple color map of Definition \ref{def:c0}.

\begin{definition}[Simple color map]
\label{def:c0}
    Let $n\geq d \geq 0$ and $k = n+d$. The color map $C_0: E_k \rightarrow \{ 0,1,3,m \}$ called the simple color map is defined by 
    \begin{enumerate}
        \item $C_0$ is regular and thus $C_0(e) = 1$ for every edge $e \in E_k$ in any corner equilateral triangle of side length $d$ in $T_k$ as in Corollary \ref{cor:regular_corner},
        \item $C_0(e) = 1$ for every edge $e \in E_k$ in the lozenge of side length $d$ in $T_k$ having outer vertices $n\xi, n\xi + d\xi, n\xi + d, n\xi + d \xi^5$, 
        \item $C_0(e) = 0$ for every edge $e \in E_k$ in the equilateral triangle having outer vertices $d, n$ and $d + n \xi$, 
        \item $C_0(e) = m$, respectively $C_0(e) = 3$ for every edge $e \in E_k$ of type $0$ with origin $x$ such that $d \leq x_0 \leq n-1$ and $1 \leq x_1 \leq d$, respectively for every edge $e \in E_k$ of type $0$ with origin $y$ such that $0 \leq y_0 \leq d-1$ and $d+1 \leq y_1 \leq n$. 
    \end{enumerate}
\end{definition}

\noindent
Figure \ref{fig:c0_example} shows an example of the simple color map $C_0$ for $k = 5$ and $d = 2$. 

\begin{figure}[H]
    \centering
    \includegraphics[scale=0.5]{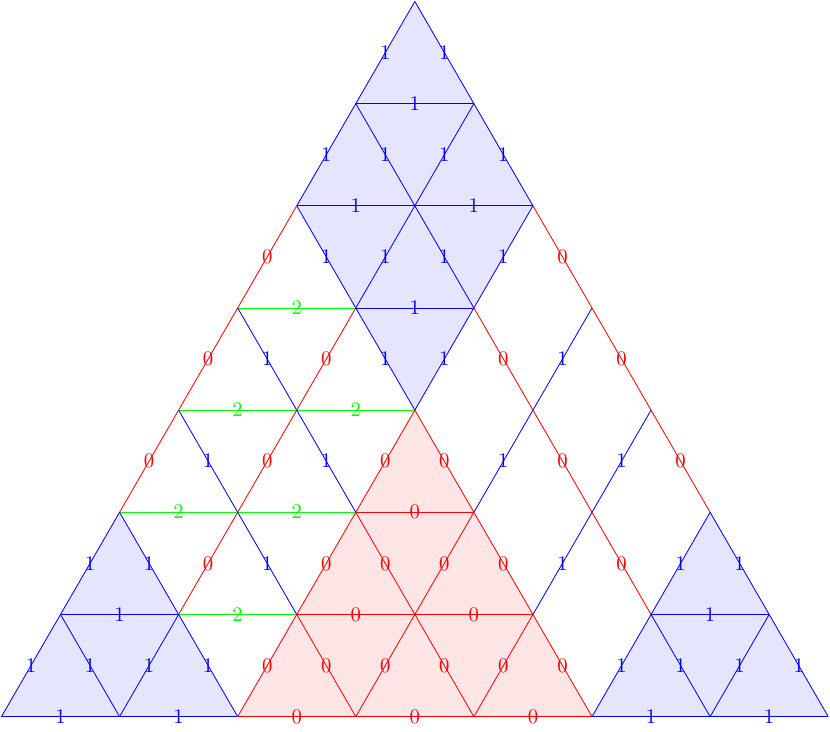}
    \caption{The color map $C_0$ for $k = 5$ and $d = 2$. Uncolored edges have color $3$ (picture done with the module {\it Knutson-Tao puzzles} of Sage \cite{sagemath}).}
    \label{fig:c0_example}
\end{figure}

\begin{proposition}[Color map reduction]
\label{prop:reduction_to_simple_color_map}
    Let $C: E_k \rightarrow \{0,1,3,m \}$ be a regular color map. Using hexagon rotations, one can map $C$ to $C_0$, where $C_0$ is the simple color map of Definition \ref{def:c0}.
\end{proposition}

\begin{proof}
    Let $\{ x^{(s)}, 1 \leq s \leq d(k-d) \}$ be the vertices in $T_k$ ordered such that for $1 \leq s \leq d(k-d)$, $s-1 = s_1d + s_2$ with $s_1 \geq 0$ and $0 \leq s_2 \leq d-1$,
    \begin{equation*}
        x^{(s)} = (d+s_1) \xi + s_2 \xi^5.
    \end{equation*}
    Let $C$ be a regular color map on $E_k$. Let us show by induction on $s$ that using hexagon rotations, $C$ can be mapped to a regular color map $C^{(s)}$ such that the type $0$ edges with origins $x^{(1)} + 1, \dots, x^{(s)}+1$ are colored $m$. 
    We first prove it for $s=1$. Notice that since $C$ is regular, there is a gash $g^{(1)}$ of type $2$ with center $x^{(1)}$. Applying Proposition \ref{prop:gash_reduction} yields that using hexagon rotations, $C$ can be mapped to $C^{(1)}$ such that $g^{(1)}$ is adjacent to a configuration $(v)$ or $(vi)$. Since $C^{(1)}$ is regular, this configuration is necessarily $(v)$ which implies that the edge of type $0$ with origin $x^{(1)} + 1$ is colored $m$. \\
    Assume that $C^{(s)}$ is a color map such that the type $0$ edges with origins $x^{(1)} + 1, \dots, x^{(s)}+1$ are colored $m$. Notice that there is a $01$ opening at $x^{(s+1)}$ that is, the edges $(x^{(s+1)}, x^{(s+1)} + \xi)$ and $(x^{(s+1)}, x^{(s+1)}+\xi^5)$ are colored respectively $0$ and $1$. A $01$ opening has only three possible completions showed in Figure \ref{fig:01opneing}.
    \begin{figure}[H]
        \centering
        \includegraphics{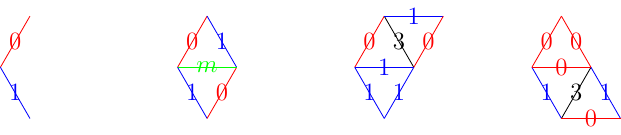}
        \caption{A $01$ opening (left) and its possible completions}
        \label{fig:01opneing}
    \end{figure}
    \noindent
    In the case of the first completion, the color map $C^{(s+1)} = C^{(s)}$ satisfies the desired conditions. In the case of the third and fourth completion, there is a gash $g^{(s+1)}$ of type $2$ and $1$ respectively with center $x^{(s+1)} + 1$. Applying Proposition \ref{prop:gash_reduction} and Remark \ref{rem:type_1_reduction} respectively shows that using hexagon rotations along the propagation path started from $g^{(s+1)}$, the latter is adjacent to a configuration $(v)$ or $(vi)$. Note that rotations in this propagation path do not affect edges colored $m$ with origins $x^{(1)} + 1, \dots, x^{(s)}+1$ thanks to the above ordering. Using a hexagon rotation for the hexagon with center $x^{(s+1)} + 1$ yields a color map $C^{(s+1)}$ such that the type $0$ edges with origins $x^{(0)} + 1, \dots, x^{(s+1)}+1$ are colored $m$ and ends the induction.
    Therefore, $C$ is equivalent up to hexagon rotations to the color map where type $0$ edges with origins $x^{(0)} + 1, \dots, x^{(d(k-d))}+1$ are colored $m$. From this configuration, there is only one possible color map to complete the rest of the hive $E_k$ which is the simple color map $C_0$.
\end{proof}

\noindent
Since the number of pieces of each type is preserved under hexagon rotations, one derives the following enumerations (only the enumeration of edges colored $m$ will be used afterwards).

\begin{corollary}[Tiles enumeration]
    Let $H$ be a dual hive with regular boundaries. Let $hc(H)$, respectively, $sc(H), $ be the number of $m$, respectively $3$ colored edges inside $H$. Then,
    \begin{equation}
    \label{eq:nb_hard_soft_crossing}
        hc(H) = d(n-d) = sc(H).
    \end{equation}
    Moreover, for $i \in \{ 0, 1 \}$, denote $nt^{(i)}(H)$, respectively $st^{(i)}(H)$, the number of direct, respectively reversed, triangular pieces of size $1$ with color $i$ on each side. Then, 
    \begin{equation}
    \label{eq:nb_triangular_pieces}
        nt^{(i)}(H) = \frac{n(i)(n(i)+1)}{2} \ \text{ and } \ st^{(i)}(H) = \frac{n(i)(n(i)-1)}{2}. 
    \end{equation}
    where $n(0) = n-d$ and $n(1) = 2d$. 
\end{corollary}
The detailed number of each type of triangles is thus
$$nt^{(0)}(H) = \frac{(n-d)(n-d+1)}{2},\,st^{(0)}(H) = \frac{(n-d)(n-d-1)}{2},\,nt^{(1)}(H) = d(2d+1),\,st^{(1)}(H) = d(2d-1).$$
\subsection{Quasi dual hives }

The goal of this section is to extend hexagon rotations to hives. As of now, hexagon rotations map one color map to another. To also change label maps, we need to relax the inequality constraints of Definition \ref{def:two_col_hives} leading to quasi hives of Definition \ref{def:quasi_hive}. From this section to the end, we view regular dual hives of $T_{n+d}$ as in the discrete hexagon $R_{d, n}$, see Definitions \ref{def:hexagonal_dual_hives}, \ref{def:boundary_hexagon} and Figure \ref{fig:from_tri_to_hexa} below.

\begin{definition}[Hexagonal dual hives]
\label{def:hexagonal_dual_hives}
    Let $n, d \geq 1$. Denote $E_{n,d} = \{ \{u, v\} \in R_{d,n}^2: d(u,v) = 1 \}$ the set of edges of the discrete hexagon $R_{d,n}$. A hexagon dual hive is a pair of maps $(C, L)$, $C: E_{n,d} \rightarrow \{ 0, 1, 3, m \}$ and $L: E_{n,d} \rightarrow \frac{1}{N} \Z$ such that $NL(.)$ satisfies the conditions of Definition \ref{def:two_col_hives} restricted to edges $e \in E_{n,d}$. 
\end{definition}

\begin{definition}[Boundary value of a hexagonal dual hives]
\label{def:boundary_hexagon}
    Define the following subsets of $E_{n, d}$ for $l \in \{ 0,1,2 \}$.
    \begin{align*}
    \partial^{(l, l)} E_{n, d} &= \{ e \in E_{n, d}: e \text{ is of type } l \text{ and } e_{l-1} = 0 \} \\
    \partial^{(l, l+1)} E_{n, d} &= \{ e \in E_{n, d}: e \text{ is of type } l \text{ and } e_{l-1} = n \}.
\end{align*}
    The boundary value $[(c^{(0,0)}, c^{(1,1)}, c^{(2,2)}, c^{(0,1)}, c^{(1, 2)}, c^{(2, 0)}), (\ell^{(0,0)}, \ell^{(1,1)}, \ell^{(2,2)}, \ell^{(0,1)}, \ell^{(1, 2)}, \ell^{(2, 0)})]$ of a hexagonal dual hive is the restriction of $(C, L)$ to $\partial E_{n, d}$ where
    \begin{itemize}
        \item $c^{(l,l)} \in \{ 0,1,3,m \}^{n-d}$ (resp. $c^{(l,l+1)} \in \{ 0,1,3,m \}^{d}$) is the restriction of $C$ to $ \partial^{(l, l)} E_{n, d} $ (resp. $ \partial^{(l, l+1)} E_{n, d} $ ), 
        \item $\ell^{(l,l)} \in (\frac{1}{N} \N)^{d}$ (resp. $\ell^{(l,l+1)} \in (\frac{1}{N} \N)^{d}$) is the restriction of $L(.)$ to $ \partial^{(l, l)} E_{n, d} $ (resp. $ \partial^{(l, l+1)} E_{n, d} $).
    \end{itemize}
    For $(c, l) = (c^{(0,0)}, c^{(1,1)}, c^{(2,2)}, c^{(0,1)}, c^{(1, 2)}, c^{(2, 0)}, \ell^{(0,0)}, \ell^{(1,1)}, \ell^{(2,2)}, \ell^{(0,1)}, \ell^{(1, 2)}, \ell^{(2, 0)}) \in \{ 0,1 \}^{3(n-d)} \times  \{ 0,1 \}^{3d} \times (\frac{1}{N} \N)^{3(n-d)} \times (\frac{1}{N} \N)^{3d}$, we denote by $\mathcal{H}_{hex}(c,l)$ the set of hexagonal dual hives with boundary value $(c, l)$.
\end{definition}

\noindent
Let $(\lambda, \mu, \nu)$ be partitions of length $n$ with first part smaller than $N-n$ such that $\vert \lambda \vert+\vert \mu\vert=\vert \nu\vert+Nd$ and such that $\min(\lambda_n, \mu_n, N-n-\nu_1) \geq d-1$. Then, the associated boundary labels $(l^{(0)}, l^{(1)}, l^{(2)})$ defined in Definition \ref{def:two_col_hives} on $T_{n+d}$ are given by
\begin{align*}
    l^{(0)} &= (0, \dots, d-1, N-n- \nu_1, N-n- \nu_2+1, \dots, N - \nu_{n-d}-d-1, N - \nu_{n-d+1} -d, \dots, N - \nu_n -1) \\
    l^{(1)} &= (0, \dots, d-1, \mu_n, \mu_{n-1}+1, \dots, \mu_{d+1} +n-d-1, \mu_{d} +n-d, \dots, \mu_1 + n -1) \\
    l^{(2)} &= (0, \dots, d-1, \lambda_n, \lambda_{n-1}+1, \dots, \lambda_{d+1} +n-d-1, \lambda_{d} +n-d, \dots, \lambda_1 + n -1)
\end{align*}
so that the associated boundary colors $(c^{(0)}, c^{(1)}, c^{(2)})$ are regular in the sense of Definition \ref{def:reg_boundary} which means that on each boundary of $T_{n+d}$, the $d$ first and last edges are colored $1$ and the remaining $n-d$ edges are colored $0$. By Lemma \ref{cor:regular_corner}, any dual hive $H = (C, L) \in H(\lambda, \mu, \nu, N)$ with regular boundary conditions has every of its equilateral triangles of size $d$ anchored in a corner of $T_{n+d}$ colored $1$. Each of these triangular regions have one boundary with labels equal to $(0, \dots, d-1)$. 
The corresponding region in the puzzle yields a unique position of the corresponding triangular pieces with edge colors $(1,1,1)$ and this unique configuration gives labels to the third boundary edge of the region. 
Let $\tilde{l}^{(0)}, \tilde{l}^{(1)}, \tilde{l}^{(2)}$ be the respective labels of edges in $\partial^{(0, 1)} E_{n, d}, \partial^{(1, 2)} E_{n, d}, \partial^{(2, 0)} E_{n, d}$. Then, reading decreasingly with respect to edges heights $h(e)$, these labels are given by the following (see Figure \ref{fig:from_tri_to_hexa} below),
\begin{align}
    \tilde{l}^{(0)} &= (N-n-\lambda_d + d-1, N-n-\lambda_{d+1}+d, \dots, N-n-\lambda_1), \\
    \tilde{l}^{(1)} &= (N-n-\mu_d + d-1, N-n-\mu_{d+1}+d, \dots, N-n-\mu_1), \\
    \tilde{l}^{(2)} &= (\nu_{n-d+1} + d-1, \dots, \nu_{n-1} +1, \nu_n).
\end{align}

\begin{figure}
    \centering
    \includegraphics[scale=0.8]{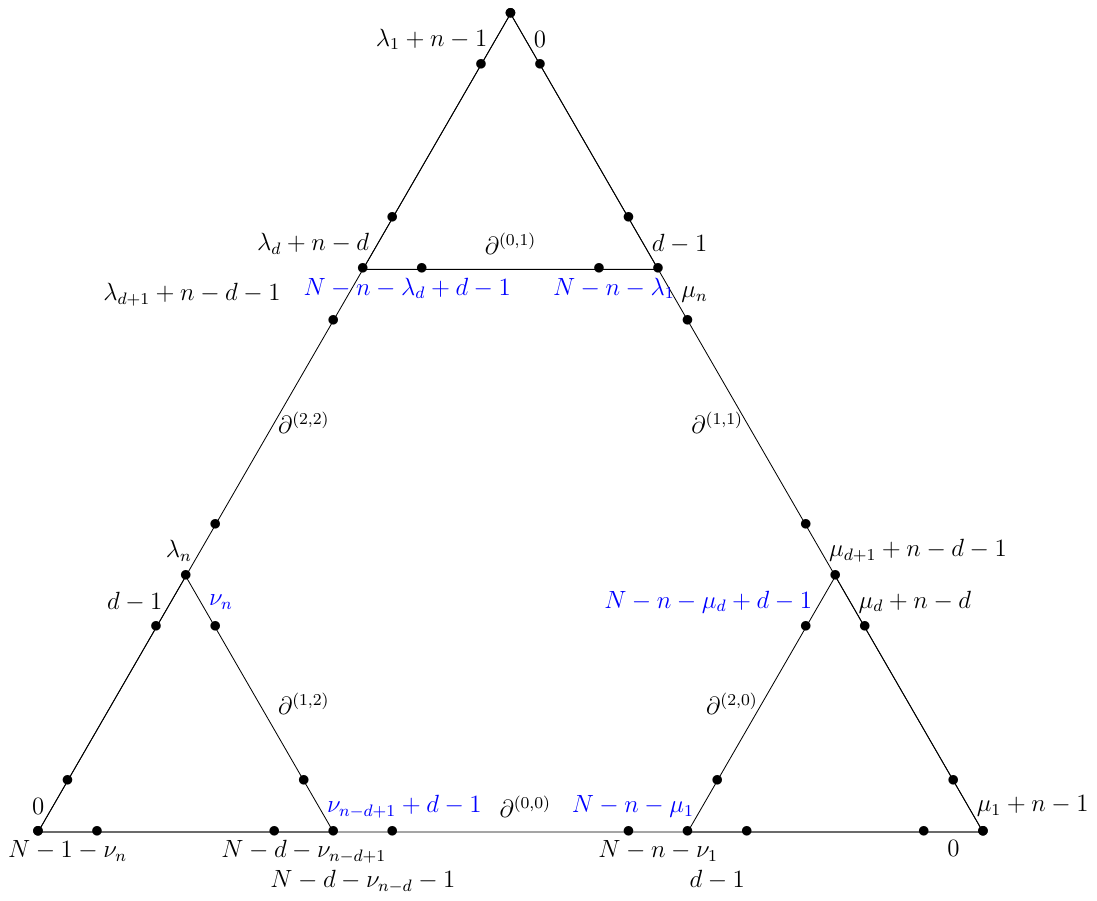}
    \caption{The restriction of a regular dual hive on $E_{n+d}$ to the hexagon $E_{n, d}$ and induced boundary value for the label map. Blue labels are determined by the unique label map on each triangular corner. Multiplying edge labels by $\frac{1}{N}$ yields boundary conditions in $H_{hex}(\lambda, \mu, \nu, N)$.}
    \label{fig:from_tri_to_hexa}
\end{figure}
\noindent
In this section, we now view regular dual hives $H = (C, L)$ on $T_{n+d}$ as hexagonal dual hives by restricting $C$ and $L$ to $E_{n,d}$. By the above, the boundary conditions on $E_{n, d}$ are given by
\begin{equation}
\label{eq:boundary_value_from_triangle_to_hexagon}
    c^{(l,l)} = (0, \dots, 0), \ c^{(l,l+1)} = (1, \dots, 1), \ \ell^{(l,l)} = \frac{1}{N} L_{\vert \partial^{(l, l)} E_{n, d}}, \ \ell^{(l,l+1)} = \frac{1}{N} \tilde{l}^{(l)}.
\end{equation}
for $l \in \{0,1,2 \}$.
We write $H_{hex}(\lambda, \mu, \nu, N)$ for the set of hexagonal dual hives having boundary colors and labels given by \eqref{eq:boundary_value_from_triangle_to_hexagon} coming from restriction of dual hives on $T_{n+d}$. This restriction
\begin{equation}
   H(\lambda, \mu, \nu, N) \rightarrow H_{hex}(\lambda, \mu, \nu, N)
\end{equation}
is a bijection where the inverse map is given by extending $C$ from the hexagon $R_{d, n}$ to the triangle $T_{n+d}$ setting $C(e) = 1$ for $e \in E_{n+d} \setminus E_{n, d}$  and completing the labels $L$ in the unique possible way in corner triangles. In this section, we now view dual hives in the hexagon $R_{d, n}$ and we will omit the subscript writing $H(\lambda, \mu, \nu, N)$ for notation convenience. 

\begin{definition}[Quasi label map, quasi dual hive]
\label{def:quasi_hive}
    Let $C: E_{n, d} \rightarrow \{ 0,1,3,m \}$ be a color map and $N \geq 1$. A quasi label map is a map $L: E_{n, d} \rightarrow \frac{1}{N} \Z $ such that $NL(.)$ satisfying the equality conditions of Definition \ref{def:two_col_hives}, that is, equality condition on every triangular face and rigid lozenges with respect to the color map $C$, and boundary values on $\partial E_{n,d}$ given by the corresponding two-colored dual hives. Denote $\widetilde{L}^C(\lambda, \mu, \nu, N)$ the set of such label maps. \\
    \noindent
    A quasi dual hive is the data of a color map $C$ and a quasi label map $L$. We denote by $\widetilde{H}(\lambda, \mu, \nu, N)$ the set of quasi dual hives with boundary conditions $(\lambda, \mu, \nu)$. 
\end{definition}
The difference with label maps of dual hives is that one does not impose the inequality constraints of Definition \ref{def:two_col_hives} inside the hexagonal region $E_{n,d}$. 
\noindent
Note that dual two-colored hives are in particular quasi dual hives that is, $H(\lambda, \mu, \nu,N)\subset \widetilde{H}(\lambda, \mu, \nu, N)$.

\begin{lemma}[Boundary value determine interior]
\label{lem:boundary_values_det_interior}
    Let $H = (C, L)$ be a quasi dual hive and $h$ an ABC hexagon for its color map $C$. Then, the values of $L$ on $E_h^\mathrm{o}$, the set of interior edges of $h$, are uniquely determined by the values of $L$ on boundary edges of $h$ and by the position of the rigid lozenge in $h$. Moreover, the values of $L$ on $E_h^\mathrm{o}$ are affine combinations of the values of $L$ on boundary edges of $h$.
\end{lemma}

\begin{proof}
    Suppose that the values of $L$ on $\partial h$ are given by $l_1, \dots, l_6$ and values on $E_h^\mathrm{o}$ by $l_7, \dots, l_{12}$. We will do the proof for a type A hexagon and the other types B and C can be treated using similar arguments. Without loss of generality up to some permutation of the indexes suppose that $C(e_{7}) = m$, see Figure \ref{fig:label_h}.
    \begin{figure}[H]
        \centering
        \includegraphics{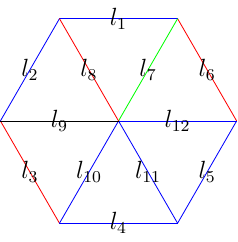}
        \caption{Labels of edges in $h$ a type A hexagon.}
        \label{fig:label_h}
    \end{figure}    
    \noindent
    By the equality conditions on opposite sides of the rigid lozenge, one has $l_{12} = l_1$ and $l_{8} = l_6$. Let us show that the values $l_7, l_9, l_{10}, l_{11}$ are uniquely determined in the region. Writing equality conditions on the five triangular faces gives
    \begin{align*}
        l_9 + l_8 + l_2 &= 1 - \frac{1}{N} \\
        l_{10} + l_9 + l_3 &= 1 - \frac{2}{N} \\
        l_{11} + l_{10} + l_4 &= 1 - \frac{1}{N} \\
        l_{12} + l_{11} + l_5 &= 1 - \frac{2}{N} \\
        l_{12} + l_{7} + l_6 &= 1 - \frac{1}{N}
    \end{align*}
    which implies
    \begin{equation}
    \label{eq:boundary_hexagon_condition}
        l_1 + l_3 + l_5 = l_2 + l_4 + l_6 - \frac{2}{N}. 
    \end{equation}
    Suppose that the former holds. The fourth equations of the system is redundant with the others and \eqref{eq:boundary_hexagon_condition} so that $l_7, l_8, l_9, l_{10}, l_{11}, l_{12}$ is solution to the invertible system
    
    \begin{equation}
    \label{eq:system_hexagon}
        \begin{pmatrix}
            1 & 0 & 0 & 0 & 0 & 0 \\
            0 & 1 & 0 & 0 & 0 & 0 \\
            0 & 0 & 1 & 0 & 0 & 0 \\
            0 & 0 & 1 & 1 & 0 & 0 \\
            0 & 0 & 0 & 1 & 1 & 0 \\
            0 & 0 & 0 & 0 & 0 & 1 \\
        \end{pmatrix} \begin{pmatrix}
            l_7 \\
            l_8 \\
            l_9 \\
            l_{10} \\
            l_{11} \\
            l_{12} \\
        \end{pmatrix} = \begin{pmatrix}
            1 - \frac{1}{N} - l_1 - l_6 \\
            l_6 \\
            1 - \frac{1}{N} - l_2 - l_6 \\
            1 - \frac{2}{N} - l_3 \\
            1 - \frac{1}{N} - l_{4} \\
            l_{1} \\
        \end{pmatrix}.
    \end{equation}
The labels on the inner edges are thus uniquely determined as linear combinations of the labels on the outer edges.

\end{proof}

\noindent
We now give a definition of hexagon rotations that incorporates the label map of a quasi hive. 

\begin{definition}[Rotation map]
\label{def:rotation_map}
    Let $C, C'$ be two color maps that differ by a hexagon rotation $h \rightarrow h'$ that is, $C'(e) = C(rot(e))$ where $rot: E_{n,d} \rightarrow E_{n,d}$ is the permutation of edges induces by the rotation mapping $h$ to $h'$. Define 
    \begin{align}
        Rot[C \rightarrow C']:  \widetilde{L}^C(\lambda, \mu, \nu, N) & \rightarrow \widetilde{L}^{C'}(\lambda, \mu, \nu, N) \\
         L & \mapsto Rot[C \rightarrow C'](L) = L'
    \end{align}
    by setting $L'(e) = L(e)$ for $e \in E_{n,d} \setminus E_h^\mathrm{o}$ and extending $L'$ to $E_h^\mathrm{o}=E_{h'}^\mathrm{o}$ by Lemma \ref{lem:boundary_values_det_interior}.
\end{definition}

\noindent
For notation convenience, we will call a hexagon rotation the image of a map of Definition \ref{def:rotation_map} for some ABC hexagon $h$ inside a quasi hive. 

\begin{lemma}[Rotation is affine bijection]
\label{lem:rotation_linear}
    Let $C, C'$ be two color maps that differ by a hexagon rotation. The map $R[C \rightarrow C']: \ \widetilde{L}^C(\lambda, \mu, \nu, N) \rightarrow \widetilde{L}^{C'}(\lambda, \mu, \nu, N)$ is an affine bijection with coefficients in $\mathbb{Z}[\lambda_i, \mu_i, \nu_i, 1/N]$ and whose inverse is $R[C' \rightarrow C]$.
\end{lemma}

\begin{proof}
    For every edge $e$ not interior to $h$, one has $L'(e) = L(e)$. In particular, $L'(e) = L(e)$ for $e \in \partial h$. Since the matrix of \eqref{eq:system_hexagon} has integer coefficients and is lower triangular, the values $\{ L'(e), e \in E_h^\mathrm{o} \}$ are integer combinations of the values $\{ L(e) = L'(e), e \in \partial h \}$ and $\frac{1}{N}\mathbb{Z}$. 
\end{proof}

\noindent
Figure \ref{fig:ex_rotation} shows an example of a hexagon rotation and the corresponding affine map $L \mapsto L'$.
\begin{figure}[H]
    \centering
    \includegraphics{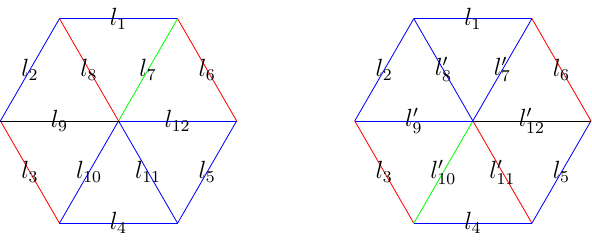}
    \caption{Action of a rotation on labels of inner edges.}
    \label{fig:ex_rotation}
\end{figure}

\noindent Using face summation constraints together with equality constraints in the rigid lozenge in $h'$ one has 
\begin{align*}
    l'_9 &= l_4, & l'_7 &= 1 - \frac{1}{N} - l_6 - l'_{12} = \frac{1}{N} + l_3 + l_5 - l_6, \\
    l'_{11} &= l_3, & l'_8 &= 1 - \frac{2}{N} - l_1 -l '_7 = \frac{1}{N} + l_3 + l_5 -l_6, \\
    l'_{12} &= 1 - \frac{2}{N} - l_3 - l_5, & l'_{10} &= 1 - \frac{1}{N} - l_4 - l'_3.
\end{align*}
which is a affine combination of values of $L$ with integer coefficients. \\
\\
\noindent For two color maps $C$, $C'$ that differ by more than one hexagon rotation, we denote $Rot[C \rightarrow C']$ the composition of maps in Definition \ref{def:rotation_map} for each hexagon rotation needed to go from $C$ to $C_0$ and then from $C_0$ to $C'$ where the existence of such paths was given by Proposition \ref{prop:reduction_to_simple_color_map}. Note that there might be multiple rotation paths from $C$ to $C'$ so that such a map is not unique. \\
\\
\noindent
Let $I \subset E_{n, d}$ be the set of edges of $E_{n, d}$ that are not in a rigid lozenge of $C_0$ except the edges of type $2$ between a rigid lozenge and a triangular face with colors $(0,0,0)$ and edges of type $1$ between a rigid lozenge and a triangular face with colors $(1,1,1)$, see Figure \ref{fig:region_I} below for an example. Denote $I_0$ the edges of $I$ of type $0$ where we remove the east-most such edge on each row. By a counting argument, there are $D=\frac{(n-1)(n-2)}{2}$ such edges so that $I_0 = (e_1, \dots, e_D)$.

\begin{lemma}[Simple quasi hive]
\label{lem:unique_label_map_c0}
    For any $z = (z_1, \dots, z_D) \in \left( \frac{1}{N}  \Z \right)^D$, there exists a unique label map $\Phi^{C_0}(z) \in L^{C_0}(\lambda, \mu, \nu, N)$ such that for all $ 1 \leq i \leq D$: $\Phi^{C_0}(z)(e_i) = z_i$. Moreover, for all $1 \leq i \leq D$, $\Phi^{C_0}(z)(e_i)$ is given by an affine combination with integer coefficients of $(z_1, \dots, z_D, \lambda, \mu, \nu, \frac{1}{N})$.
\end{lemma}

\begin{figure}[H]
    \centering
    \includegraphics{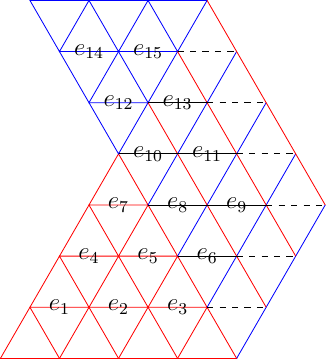}
    \caption{The region $I$ for $n = 7$ and $d=3$. $I_0 = \{e_i, 1 \leq i \leq 15=D \}$. Dashed edges are edges of type $0$ in $I \setminus I_0$.}
    \label{fig:region_I}
\end{figure}

\begin{proof}
    Let $L: \partial E_{n,d}\cup I_0 \rightarrow \frac{1}{N} \Z$ be a function satisfying the boundary condition $(\lambda, \mu, \nu, N)$ as in \eqref{eq:boundary_value_from_triangle_to_hexagon}. We will show that $L$ can be extended to a quasi label $E_{n,d}$ in a unique way. For any edge of type $2$ part of an rigid lozenge, there exists an edge $e_{\partial}(e) \in \partial E_{n,d}$ of the same type obtained by translation of $e$ by a multiple of $\ed^{i\frac{2\pi}{3}}$. Likewise, for any edge of type $1$ part of an rigid lozenge, there exists an edge $e_{\partial}(e) \in \partial E_{n,d}$ of the same type obtained by translation of $e$ by a multiple of $\ed^{i\frac{\pi}{3}}$. Assign $L(e) = L(e_{\partial}(e))$ for each such edge $e$. By the equality condition on opposite edges of rigid lozenges, any quasi label map has the same values on these edges. \\
    \\
    It remains to extend $L$ to edges $e \in I$. The values of $\partial I$ are already determined uniquely by the boundary conditions. Set $L(e_i) = z_i$ for $1 \leq i \leq D$. We call a band the following configuration of adjacent faces where the west-most triangular face has both its edges of type $0$ and $2$ labeled, the east-most triangular face has its type $1$ edge labeled and all other faces in between have their type $0$ edge labeled. 

    \begin{figure}[H]
        \centering
        \includegraphics{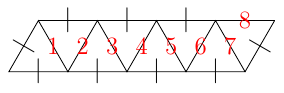}
        \caption{A band of size four. Marked edges are the already labeled edges. The labels of other edges are determined in the order of the red numbers by face summation constraints.}
        \label{fig:band}
    \end{figure}
    \noindent
    We claim that there in a unique labeling of the edges in the band such that the face summation constraints hold. The west-most face of the band has two of its three edges labeled so that the third label is determined uniquely. The south pointing triangular face east to it has two out of three edges labeled so that the third one is also fixed. By inductively propagating east, one labels the edges of the band. Note that the label of the last east-most edge of type $0$ is determined. This is why we do not require to fix the values of type $0$ edges in $I \setminus I_0$. \\
    \\
    The previous paragraph shows that one can extend $L$ to all edges in the south most band of $I$. Notice then that the region above is also a band. By induction, one extends $L$ to the $n-d$ south most bands of $I$. The remaining region consists of bands turned upside down which are also uniquely determined by the same reasoning. The resulting map $L$ on $E_{n,d}$ satisfies the face summation constraints so that $L \in \widetilde{L}^{C_0}(\lambda, \mu, \nu, N)$. Since the band completions are unique at each step, the labels of any other quasi label map $L' \in \widetilde{L}^{C_0}(\lambda, \mu, \nu, N)$ with same values on $I_0$ would agree with $L$.
\end{proof}

\begin{definition}[Label map associated to edge coordinates]
\label{lem:def_phi_C}
    Define $\Phi^{C_0}(z) \in L^{C_0}(\lambda, \mu, \nu)$ to be the unique label map constructed in Lemma \ref{lem:unique_label_map_c0} from specifying edge coordinates $z \in \left( \frac{1}{N}  \Z \right)^D$ on $I_0$. Let $C: E_{n,d} \rightarrow \{0,1,3,m \}$ be any color map and let $D = (n-1)(n-2)/2$. Define the following map
    \begin{align*}
        \Phi^C: \hspace{1,5cm}\left( \frac{1}{N}  \Z \right)^D & \rightarrow L^C(\lambda, \mu, \nu, N) \\
        z = (z_1, \dots, z_D) & \mapsto \Phi^C(z) = Rot[C_0 \rightarrow C](\Phi^{C_0}(z)).
    \end{align*}
\end{definition}

\begin{proposition}[Quasi hive structure]
\label{prop:quasi_hive_struct}
    The map $\Phi^C$ of Definition \ref{lem:def_phi_C} is bijective. Moreover, for any $z \in  \left( \frac{1}{N}  \Z \right)^D$ and edge $e \in E_{n,d}$ , $\Phi^C(z)(e)$ is an affine combination of $(z_1, \dots, z_D, \lambda, \mu, \nu, \frac{1}{N})$ with integer coefficients.
\end{proposition}

\begin{proof}
    The map $\Phi^C$ is the composition of two bijections : $z \mapsto \Phi^{C_0}(z) \in L^{C_0}(\lambda, \mu, \nu, N)$ and $L \mapsto Rot[C_0 \rightarrow C](L)$ which are both affine in $(z_1, \dots, z_D, \lambda, \mu, \nu, \frac{1}{N})$ by Lemma \ref{lem:rotation_linear} and Lemma \ref{lem:unique_label_map_c0}. Its inverse is given by $(\Phi^C)^{-1}(L) = Rot[C_0 \rightarrow C](L)_{\vert I_0}$ where to a label map $L \in L^{C_0}(\lambda, \mu, \nu, N)$, $L_{\vert I_0} = (L(e_1), \dots, L(e_D))$ are the labels of the edges in $I_0$.

\end{proof}

\noindent
So far we have defined the maps $Rot[C \rightarrow C']$ and $\Phi^C$ from $\widetilde{L}^C(\lambda, \mu, \nu, N) \rightarrow \widetilde{L}^{C'}(\lambda, \mu, \nu, N)$ and $\left( \frac{1}{N}  \Z \right)^D \rightarrow \widetilde{L}^C(\lambda, \mu, \nu, N)$ respectively. We will now extend their definitions to quasi hives.
\begin{definition}[Extension to dual hives]
    Let $C, C'$ be two color maps. We extend the maps $Rot[C \rightarrow C']$ and $\Phi^C$ of Definitions \ref{def:rotation_map} and \ref{lem:def_phi_C} to quasi hives by
    \begin{align}
        Rot[C \rightarrow C']: & \ \widetilde{H}^C(\lambda, \mu, \nu, N) \rightarrow \widetilde{H}^{C'}(\lambda, \mu, \nu, N) \\
        & \ H = (C, L) \mapsto Rot[C \rightarrow C'](H) = (C', Rot[C \rightarrow C'](L)),  
    \end{align}
    and 
    \begin{align}
        \Phi^C: \left( \frac{1}{N}  \Z \right)^D & \rightarrow \widetilde{H}^C(\lambda, \mu, \nu, N) \\
        z = (z_1, \dots, z_D) & \mapsto (C, \Phi^C(z)).
    \end{align}
\end{definition}

\section{Convergence to a volume of hives}\label{Sec:convergence}

Since $\diff \prob[\gamma+2r | \alpha+r, \beta+r]=\diff \prob[\gamma | \alpha, \beta]$ for $r>0$, assume in this section without loss of generality that $\alpha,\beta,\gamma \in \mathcal{H}_{reg}$ are such that $\alpha_n,\beta_n>0$ and $\gamma_1 < 1$. In particular, if $\lambda_N,\mu_N,\nu_N$ are such that $\frac{1}{N}\lambda_N\rightarrow \alpha, \frac{1}{N}\mu_N\rightarrow \beta, \frac{1}{N}\nu_N\rightarrow \gamma$, then $\min( (\lambda_N)_n, (\mu_N)_n, N-n-(\nu_N)_1 ) > d+1$ for $N$ large enough. Hence, we will assume throughout this section that the hives with boundary $(\lambda_N,\mu_N,\nu_N)$ are regular in the sense of \eqref{eq:boundary_value_from_triangle_to_hexagon}, see Figure \ref{fig:from_tri_to_hexa}.
%\noindent
%One can recover the edge colors in the sense of Definition \ref{def:two_col_hives} from a given height function via the following process. Orient the edges around north-pointing faces clockwise and edges around south-pointing faces counterclockwise. Given any oriented edge $e = (u, v)$, its color is given by
%\begin{equation}
%    C(e) = \begin{cases}
%        1 & \text{ if } H(v) - H(u) = 1 \ [3] \\
%        0 &\text{ if } H(v) - H(u) = 2 \ [3] \\
%        3 & \text{ if } H(v) - H(u) = 0 \ [3] \text{ and } F(e) \text{ has cyclic height differences } (1,2,0) \\
%        m & \text{ if } H(v) - H(u) = 0 \ [3] \text{ and } F(e) \text{ has cyclic height differences } (2,1,0).
%    \end{cases}
%\end{equation}

\subsection{Limit dual hives}

This section introduces limit dual hives which will be linked to the limit of quantum Littlewood-Richardson coefficients of Theorem \ref{th:density_limit_quantum_coef}.

\begin{definition}[Limit dual hive]
\label{def:limit_dual_hive}
    For $\alpha, \beta, \gamma \in (\R_{\geq 0})^3$, the limit dual hive $H(\alpha, \beta, \gamma, \infty)$ is the set of pairs $(C, L)$ on the hexagon edges $E_{n,d}$ such that :
    \begin{enumerate}
        \item $C: E_{n,d} \rightarrow \{ 0, 1, 3, m \}$ is a color map,
        \item $L: E_{n,d} \rightarrow \R_{\geq 0} $ is the label map satisfying 
        \begin{enumerate}
            \item $L(e_1) + L(e_2) + L(e_3) = 1 $ for every triangular face of $F_k$,
            \item if $e,e'$ are edges of same type on the boundary of a same lozenge $f$,
            \begin{enumerate}
                \item $L(e)=L(e')$ if the middle edge of $f$ is colored $m$,
                \item $L(e)\geq L(e')$ if $h(e)> h(e')$.
            \end{enumerate}
            \item The values of $L$ on $\partial E_{n,d}$ are given by $(\alpha, \beta, \gamma)$ so that, sorted in decreasing height of edges, see Figure \ref{fig:boundary_limit_hive} below.
            \begin{align*}
                \ell^{(0, 1)} &= (1 - \alpha_d, \dots, 1 - \alpha_1),
                &\ell^{(2, 2)} &= (\alpha_{d+1}, \dots, \alpha_{n}) \\
                \ell^{(2, 0)} &= (1 - \beta_d, \dots, 1 - \beta_1),
                &\ell^{(1, 1)} &= (\beta_{n}, \dots, \beta_{d+1}) \\
                \ell^{(1, 2)} &= (\gamma_n, \dots, \gamma_{n-d+1}), 
                &\ell^{(2, 2)} &= (1-\gamma_{n-d}, \dots, 1- \gamma_{1}).
            \end{align*}
        \end{enumerate}
    \end{enumerate}
\end{definition}

\begin{figure}[H]
    \centering
    \includegraphics[scale=0.75]{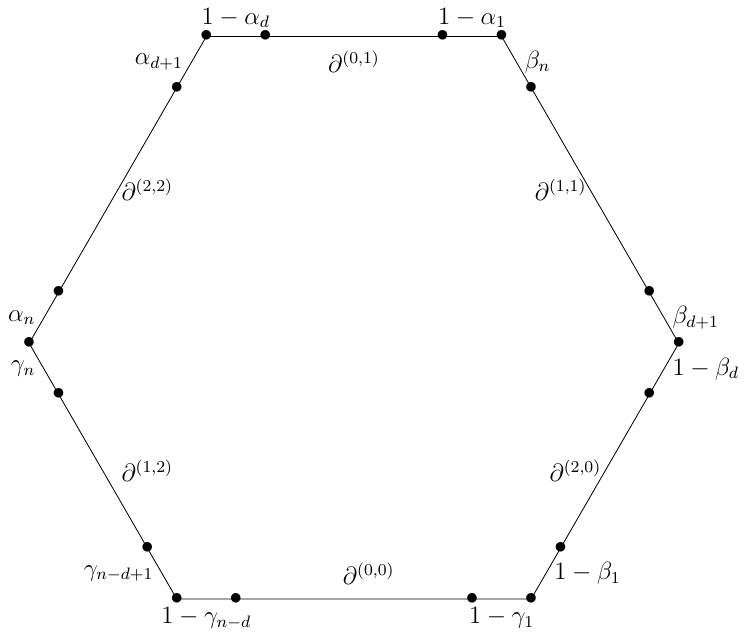}
    \caption{Boundary labels for limit hives in $H(\alpha, \beta, \gamma, \infty)$.}
    \label{fig:boundary_limit_hive}
\end{figure}

\noindent
As in the discrete case, if $C$ is a color map, we denote by $H^C(\lambda, \mu, \nu, \infty) \subset H(\lambda, \mu, \nu, \infty)$ the subset of limit dual hives of having color map $C$. As in the previous section, we denote by $\widetilde{H}(\lambda, \mu, \nu, \infty)$ the set of pairs $(C, L)$ as above where we remove the inequality conditions $(2.b.ii)$ on $L$.
\begin{remark}[Maps $\Phi^C$ and $Rot$ on limit dual hives]
    Note that Lemmas \ref{lem:boundary_values_det_interior}, \ref{lem:rotation_linear} and \ref{lem:unique_label_map_c0} hold for limit quasi dual hives $H = (C, L) \in \widetilde{H}(\lambda, \mu, \nu, \infty)$ extending $z \in \left( \frac{1}{N} \Z \right)^D$ to $z \in \R^D$. Using the same construction as in Section \ref{sec:color_swap}, we define $Rot[C \rightarrow C']: \widetilde{H}^C(\lambda, \mu, \nu, \infty) \rightarrow \widetilde{H}^{C'}(\lambda, \mu, \nu, \infty)$ and $\Phi^C: \R^D \rightarrow \widetilde{H}^C(\lambda, \mu, \nu, \infty)$ which are affine bijections with coefficients in $\mathbb{Z}[\alpha,\beta,\gamma]$ obtained by setting $\frac{1}{N}=0$.
\end{remark}

\subsection{Convergence to a volume}

The goal of this part is to prove Proposition \ref{prop:convergence_volume_dual} which expresses the limit quantum cohomology coefficients as the volume involving limit dual hives. The proof relies on Lemma \ref{lem:pt_convergence_indicator} and Lemma \ref{lem:bound_indicator} below.

\begin{proposition}[Convergence to volume of dual hives]
\label{prop:convergence_volume_dual}

    \begin{equation}
        \lim_{n \rightarrow \infty} N^{-D} c_{ \lambda_N, \mu_N}^{\nu_N, d} = \sum_{C} Vol \left( u \in \R^D, \Phi^C[u]\in H^C(\alpha, \beta, \gamma, \infty) \right).
    \end{equation} 
\end{proposition}

\noindent
Recall from Corollary \ref{cor:qLR_dual_hive} that 

\begin{align}
        N^{-D} c_{ \lambda_N, \mu_N}^{\nu_N, d} &= N^{-D} | H(\lambda_N, \mu_N, \nu_N, N) | = N^{-D} | \widetilde{H}(\lambda_N, \mu_N, \nu_N, N) \cap Ineq(n) | \\
        &= \sum_{C} \int_{\R^D} \sum_{z =(z_1, \dots, z_d) \in (\frac{1}{N} \Z)^D : \Phi^C(z) \in \widetilde{H}(\lambda_N, \mu_N, \nu_N, N) \cap Ineq(n) } \mathbf{1}(u)_{z + [-\frac{1}{N}, \frac{1}{N}[^D} du.
        \label{eq:discrete_volume}
\end{align}
where $Ineq(n)$ are the inequality constraints given in $(2b)$ and $(2c)$ of Definition \ref{def:two_col_hives}.

\begin{lemma}[Pointwise convergence]
\label{lem:pt_convergence_indicator}
    For any color map $C$ and $N \geq 1$, define 
    
    \begin{align*}
        f_N^C :  \R^D &\rightarrow \R \\
         u &\mapsto \sum_{z =(z_1, \dots, z_d) \in (\frac{1}{N} \Z)^D : \Phi^C(z) \in \widetilde{H}(\lambda_N, \mu_N, \nu_N, N) \cap Ineq(n) } \mathbf{1}(u)_{ \{ z + [-\frac{1}{N}, \frac{1}{N}[^D \}}.
    \end{align*}
    Recall that $\frac{1}{N} \lambda_N = \alpha + o(1)$, $\frac{1}{N} \mu_N = \beta + o(1)$ and $\frac{1}{N} \nu_N = \gamma + o(1)$ as $N \rightarrow + \infty$.
    Then, for almost all $u \in \R^D$ with respect to the Lebesgue measure:
    \begin{equation}
    \label{eq:convergence_indicator}
         \lim_{N \rightarrow \infty} f_N^C(u) = \mathbf{1}(u)_{ \{ \Phi^C[u] \in H^C(\alpha, \beta, \gamma, \infty) \} }.
    \end{equation}
\end{lemma}

\begin{remark}
    Note that a priori, $\Phi^C[u] \in \widetilde{H}^C(\alpha, \beta, \gamma, \infty)$ is a label map such that $(C, L)$ is a limit hive of Definition \ref{def:limit_dual_hive} without inequality constraints $(2c)$ and (2d). Here, the right hand side is more restrictive as it requires that $ (C, \Phi^C[u]) \in H^C(\alpha, \beta, \gamma, \infty) = \widetilde{H}^C(\alpha, \beta, \gamma, \infty) \cap Ineq(n)$.
\end{remark}

\begin{proof}
    Take $u$ such that $\Phi^C[u]$ in the interior of $H^C(\alpha, \beta, \gamma, \infty)$. We want to show that $f_N^C(u) = 1$ for $N \geq N_0$ which means that one can find a sequence $(z^{(N)}, N \geq N_0) = ( (z_1^{(N)}, \dots, z_D^{(N)}), N \geq N_0)$ such that for each $N\geq N_0$ : $\Phi^C(z^{(N)})\in\widetilde{H}(\lambda_N, \mu_N, \nu_N, N) \cap Ineq(n)$ and  $u \in z^{(N)} +[- 1/N,1/N[$. Let us take the label map
  \begin{equation}
        L^{(N)}: e \mapsto \Phi^C_N(\lfloor Nu \rfloor/N)(e) \in \frac{1}{N} \Z.
    \end{equation}
    associated to $z^{(N)} = \lfloor Nu \rfloor /N$ : $L^{(N)} = \Phi^C_N(z^{(N)})$ where $\Phi^C_N$ is associated to the boundaries $(\lambda_N, \mu_N, \nu_N)$. One has $|z^{(N)} - u | < 1/N$ by construction. We need to check that $L^{(N)} \in \widetilde{H}(\lambda_N, \mu_N, \nu_N, N) \cap Ineq(n)$. \\
    \\
    By definition, $\Phi^C_N$ is a quasi label map with boundary conditions $(\lambda_N, \mu_N, \nu_N)$ so that $L^{(N)} \in \widetilde{H}(\lambda_N, \mu_N, \nu_N, N) $.
    Let us check the inequality constraints of $Ineq(n)$. Take any pair of edges $(e, e')$ subject to an inequality. Since $\Phi^C[u] \in H^C(\alpha, \beta, \gamma, \infty)$ is in the interior, this equality is sharp for $\Phi^C[u]$ that is 

    \begin{equation}
        \exists \epsilon_{e, e'} >0 : \Phi^C[u](e) \leq \Phi^C[u](e') + \epsilon_{e, e'}. 
    \end{equation}

    \noindent
    Since $\lim_{N  \rightarrow +\infty} L^{(N)}(e) = \Phi^C[u](e)$ for any edge $e$, there exists $N_0(e, e')$ such that for $N \geq N_0(e, e')$:
    \begin{equation}
        L^{(N)}(e)  \leq L^{(N)}(e).
    \end{equation}
    Hence, $L^{(N)}$ satisfies all the inequality constraints for $N \geq N_0$, where $N_0$ is the largest of the thresholds $N_0(e, e')$ for $(e, e')$ related by an inequality constraint in Definition \ref{def:limit_dual_hive}. Therefore, 

    \begin{equation}
        \forall N \geq N_0: L^{(N)} \in \widetilde{H}(\lambda_N, \mu_N, \nu_N, N) \cap Ineq(n).
    \end{equation}
    so that $\lim_{N \rightarrow + \infty} f_N^C(u) = 1$ as desired. For $\Phi^C[u] \notin H^C(\alpha, \beta, \gamma, \infty)$, one of the inequalities in $(2.b.ii)$ of Definition \ref{def:limit_dual_hive} is violated, for all other conditions being satisfied by construction of $\Phi^C$. Let $(e, e')$ be a pair of edge such that $(2.b.ii)$ is not satisfied : $\Phi^C[u](e) < \Phi^C[u](e')$ while $h(e)>h(e')$ for some pair of edges of same type adjacent to a same lozenge. Using that $\lim_{N  \rightarrow +\infty} L^{(N)}(e) = \Phi^C[u](e)$, one has that for $N$ large enough $L^{(N)}(e) < L^{(N)}(e')$ so that $L^{(N)} \notin \widetilde{H}(\lambda_N, \mu_N, \nu_N, N) \cap Ineq(n)$. Hence, \eqref{eq:convergence_indicator} holds for almost all $u\in \mathbb{R}^D$ with respect to the Lebesgue measure.

\end{proof}

\begin{lemma}[Uniform bound]
\label{lem:bound_indicator}
    Let $f_N^C(u)$ be as in \eqref{eq:convergence_indicator}. Then, there exists a compact $K^{(C)} \subset \R^D$ such that for every $N \geq 1$,
    \begin{equation}
        |f_N^C(u)| \leq \mathbf{1}_{ \{ u \in K^{(C)} \} }.
    \end{equation}
\end{lemma}

\begin{proof}[Proof of Lemma \ref{lem:bound_indicator}]
    If $C = C_0$, the values $z \in \Z^D $ such that $\Phi_N^{C_0}(z) \in \widetilde{H}^{C_0}(\lambda_N, \mu_N, \nu_N, N) \cap Ineq(n) $ are in $[0, 1]^D$ since $(z_1, \dots, z_D)$ are the values of $(\Phi_N^{C_0}(z)(e_1), \dots, \Phi_N^{C_0}(z)(e_D) )$ for some horizontal edges $(e_1, \dots, e_D) \in E_{n,d}$ which are in $[0, 1]$ by construction.
    \\
    If $C \neq C_0$ by definition 
    \begin{equation}
        \Phi^C(z) = Rot[C_0 \rightarrow C] (\Phi^{C_0}(z))
    \end{equation}
    where $\Phi^{C_0}(z)$ is the label map of the quasi hive with simple color map $C_0$ having horizontal edge labeled $z$. For $z \in \Z^D $ such that $\Phi_N^{C}(z) \in \widetilde{H}^{C}(\lambda_N, \mu_N, \nu_N, N) \cap Ineq(n) $, we know from $Ineq(n)$ that values $ \{ \Phi_N^{C}(z)(e), e \in E_{n,d} \}$ are in $[0, 1]^{E_{n,d}}$. Applying the affine hence continuous map $Rot[C\rightarrow C_0]$, we get that $Rot[C\rightarrow C_0]\Phi^C(z)=\Phi^{C_0}(z)\in Rot[C\rightarrow C_0]([0,1]^{E_{n,d}})$ which is compact. In particular, the labels $(z_i = \Phi^{C_0}(z)(e_i), 1 \leq i \leq D)$ of horizontal edges $e_i$ in $I$ are in compact sets hence bounded.

    %\textcolor{green}{On sait que $\Phi_C(z)\in[0,1]^{E_{n,d}}$ si $\Phi_C(z)\in H^{C}(\lambda_N, \mu_N, \nu_N, N)$. En appliquant, on en déduit que $Rot[C\rightarrow C_0]\Phi_C(z)=\Phi_{C_0}(z)\in Rot[C\rightarrow C_0]([0,1]^{E_{n,d}})$ qui est compact car $Rot[C\rightarrow C_0]$ est continue. Donc $\Phi_C^{-1}(H^{C_0}(\lambda_N, \mu_N, \nu_N, N))\subset \Phi_{C_0}^{-1}(Rot[C\rightarrow C_0]([0,1]^{E_{n,d}})$ qui est donc compact par le cas $C=C_0$.}

\end{proof}

\begin{proof}[Proof of Proposition \ref{prop:convergence_volume_dual}]

By Lemma \ref{lem:pt_convergence_indicator} and Lemma \ref{lem:bound_indicator}, using dominated convergence theorem in \eqref{eq:discrete_volume}:

\begin{equation}
    \lim_{n \rightarrow \infty} N^{-D} c_{ \lambda_N, \mu_N}^{\nu_N, d} = \sum_{C} Vol \left( u \in \R^D, \Phi^C[u]\in H^C(\alpha, \beta, \gamma, \infty) \right).
\end{equation}
    
\end{proof}

\subsection{Volume preserving map}

This subsection aims at proving that there is a volume preserving map between dual hives $H(\alpha, \beta, \gamma, \infty)$ and hives $P_{\alpha,\beta,\gamma}^g$ of Definition \ref{def:p_g_alpha_beta_gamma}. Refer to the notations of Section \ref{Sec:Notation_statement} for the hives and related notions. \\
\\
\noindent
Let $H = (C, L) \in \tilde{H}^C(\alpha, \beta, \gamma, \infty)$. We assign to $H$ an function $\Psi^C: R_{d, n} \rightarrow \R$ constructed as follows. Set $\Psi^C[L](v) = d$ where $v$ is the south-east vertex of $R_{d,n}$. For any other vertex $v \in R_{d, n}$ such that $e=(u, v) \in E_{n, d}$ and $\Psi^C(u)$ has been set, the value $\Psi^C(v)$ is given by 
\begin{equation}
\label{eq:Psi_values_edges}
    \Psi^C(v) =
    \begin{cases}
        \Psi^C(u) + L(e) &\text{ if $e$ is of type $1$ or $2$ } \\
        \Psi^C(u) + 1 - L(e) &\text{ if $e$ is of type $0$. }
    \end{cases} 
\end{equation}
See Figure \ref{fig:edge_to_vertices} for a picture of the recursive construction of $\Phi^C$ along edges.
    \begin{figure}[H]
        \centering
        \includegraphics{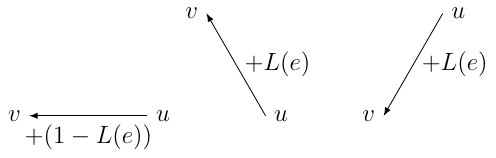}
        \caption{Values at vertices when traversing an edge $e = (u, v)$.}
        \label{fig:edge_to_vertices}
    \end{figure}

\begin{definition}[Dual hive to hive]
\label{def:dual_map}
    Let $C$ be a color map. Define 
    \begin{equation}
        S^C = \{ v_4 : l=(v_1, v_2, v_3, v_4) \text{ is a rigid lozenge} \} \subset R_{d,n}
    \end{equation}
    and 
    \begin{equation}
        \tilde{P}^C_{\alpha,\beta,\gamma}:=\{f: R_{d,n} \setminus S^C \rightarrow \mathbb{R},f_{\partial R_{d,n}} \text{ given by } \alpha,\beta,\gamma \}.
    \end{equation}
    Moreover, define
        \begin{align}
            \Psi^C: \ & \tilde{H}^C(\alpha, \beta, \gamma, \infty) \rightarrow \tilde{P}^C_{\alpha,\beta,\gamma} \\
            & (C, L) \longmapsto \Psi^C[L]
        \end{align}
    where $\Psi^C$ is given by \eqref{eq:Psi_values_edges} in the above construction.
\end{definition}

\noindent
Remark that the choice of $v_4$ and the coloring on the boundary ensures that $v_4$ is never on the boundary of $R_{d,n}$ in the above definition.
\noindent
That the map $\Psi^C$ is well defined is due to the face summation constraint $L(e_1) + L(e_2) = 1 - L(e_0)$ around every face $f \in F_k$ having edges $(e_0, e_1, e_2)$ of respective types $0, 1, 2$ for the label maps $L$ of dual hives in $\tilde{H}^C(\alpha, \beta, \gamma, \infty)$. Remark that $\Psi^C$ depends on $C$ only through its domain and target space.

\begin{remark}[Extension of functions of $\tilde{P}^C_{\alpha,\beta,\gamma}$ ]
\label{rem:extension_functions}
    Let $f \in \tilde{P}^C_{\alpha,\beta,\gamma}$. Then $f$ can be uniquely extended to a map $f: R_{d,n} \rightarrow \mathbb{R}$ by setting $f(v_4)=f(v_3)+f(v_1)-f(v_2)$ for any $v_4 \in S^C$.
\end{remark}

\noindent
Let $C: E_{n,d} \rightarrow \{0, 1, 3, m \}$ be a regular color map. Define 
\begin{align*}
    g[C]: R_{d, n} &\rightarrow \Z_3 \\
    v &\mapsto g[C](v)  
\end{align*}
where the value at vertex $v \in R_{d,n}$ is set as follows. If $v = A_0$ the south-east most vertex in $R_{d, n}$, set $g[C](v) = 0$. Orient the edges in $E_{n, d}$ around direct triangles clockwise and edges around reversed triangles counterclockwise. For any oriented edge $e = (u,v)$, set 
\begin{equation}
    g[C](v) = \begin{cases}
     g[C](u) + 1 & \text{ if } C((u,v)) = 1 \\
     g[C](u) + 2 & \text{ if } C((u,v)) = 0 \\
     g[C](u) & \text{ if } C((u,v)) \in \{ 3, m \} .
\end{cases} 
\end{equation}
\begin{proposition}[From color maps to regular labelings]
\label{prop:regular_label_color}
    The above map $C \mapsto g[C]$ is a bijection between color maps on $E_{n,d}$ and regular labelings on $R_{d, n}$. For any regular labeling $g$, its inverse is given by 
    \begin{align*}
        C[g]: E_{n, d} &\rightarrow  \{0, 1, 3, m \} \\
        e = (u,v) &\mapsto C[g](e)  
    \end{align*}
    where, if $w \in R_{d, n}$ denotes the third vertex so that $(u, v, w)$ is a direct triangular face, 
    \begin{equation}
        C[g](e)  = \begin{cases}
        1 & \text{ if } g(v) - g(u) = 1 \\
        0 & \text{ if } g(v) - g(u) = 2 \\
        3 & \text{ if } g(v) = g(u) \text{ and } g(w) = g(u) - 1 = g(v) - 1 \\
        m & \text{ if } g(v) = g(u) \text{ and } g(w) = g(u) + 1 = g(v) + 1.
    \end{cases} 
    \end{equation}
\end{proposition}

\begin{proof}
    Let us show that $g[C]$ is well defined. Since $C$ is a color map, the only colors around a triangular face in $E_{n, d}$ are up to cyclic permutations $(0,0,0), (1,1,1), (1,0,3) $ and $(0,1,m)$. One checks that summing the clockwise differences of values of $g$ going from a vertex to itself around any such color triple gives a zero contribution in $\Z_3$. Therefore, the value of $g[C](v)$ does not depend on the choice of the path from $A_0$ to $v$. 
    That $(g[C]^A, g[C]^B, g[C]^C)$ has the right boundary conditions is due to the fact that $C$ is regular. It remains to check the lozenge condition on $g[C]$ from Definition \ref{def:regular_labeling}. Take any lozenge $l = (v^1, v^2, v^3, v^4)$ and suppose that $g[C](v^2) = g[C](v^4)$. Note that from Figure \ref{Fig:possible_lozenges}, the edge between $v^2$ and $v^4$ is always oriented from $v^4$ to $v^2$. The edge $e = (v^4, v^2)$ has color either $3$ or $m$. Since $C$ is a color map, the two faces adjacent to $e$ have either $(1,0,3)$ or $(0,1,m)$ colors. The face with vertices $(v^1, v^2, v^4)$, respectively $(v^3, v^2, v^4)$ is always direct, respectively reversed, see Figure \ref{Fig:possible_lozenges}. If $C(e) = 3$, $g[C](v^1) = g[C](v^2) + 1$ and $g[C](v^3) = g[C](v^2) + 2$ whereas if $C(e) = m$, $g[C](v^1) = g[C](v^2) + 2$ and $g[C](v^3) = g[C](v^2) + 1$. In both cases, $\{ g[C](v^1), g[C](v^3) \} = \{g[C](v^2) + 1, g[C](v^2) + 2  \} $ and thus $g[C]$ is a regular labeling. \\
    \\
    Let us show that $g \mapsto C[g]$ maps a regular labeling $g$ to a color map. Since $g$ is regular, $C[g]$ also is by the same argument as above. Let us show that the only cyclic colors triples around any triangular face are $(0,0,0), (1,1,1), (1,0,3)$ and $(0,1,m)$. Take any triangular face and denote $X, Y, Z$ the clockwise differences of values of $g$. Then, $X+Y+Z = 0 [3]$ so that $(X, Y, Z) \in \{ (1,1,1), (2,2,2), (0,1,2), (0,2,1) \}$ up to cyclic rotation. Note that we exclude $(0,0,0)$ by the lozenge condition in Definition \ref{def:regular_labeling} since no lozenge can have three vertices with equal values, for otherwise the second condition of Definition \ref{def:regular_labeling} is violated. These possible height differences give the clockwise colors $\{ (1,1,1), (0,0,0), (3,1,0), (m,0,1) \}$ respectively. Therefore, $C[g]$ is a color map and by construction $g \mapsto C[g]$ is the inverse of $C \mapsto g[C]$. 
\end{proof}

\begin{lemma}[Image of limit dual hives are toric concave functions]
\label{lem:img_limit_dual_hive_is_toric}
    For any regular color map $C$,
    \begin{equation}
    \label{eq:img_conversion}
    \Psi^C( H^C(\alpha, \beta, \gamma, \infty) ) = P^{g[C]}_{\alpha,\beta,\gamma}. 
\end{equation}
    where $g[C]$ is the regular labeling associated to $C$ as in Proposition \ref{prop:regular_label_color} and $P^{g[C]}_{\alpha,\beta,\gamma}$ is the polytope defined in Definition \ref{def:p_g_alpha_beta_gamma}.
\end{lemma}

\begin{proof}

    The image $\Psi^C[(C, L)]$ of any limit hive $(C, L) \in H^C(\alpha, \beta, \gamma, \infty) $ can be extended by Remark \ref{rem:extension_functions} to a function $f: R_{d,n} \rightarrow \mathbb{R}$ such that by construction $\Psi^C[(C, L)] = f_{\vert R_{d, n} \setminus S^C} = f_{\vert Supp(g[C])}$. Let us check that $f \in P^{g[C]}_{\alpha,\beta,\gamma}$. By construction, the values of $f$ on $\partial R_{d, n}$ are as in Definition \ref{def:p_g_alpha_beta_gamma}, see Figure \ref{fig:bound_p_alpha_beta_gamma}. By definition of the extension, $f$ satisfies the equality constraints over any rigid lozenge in $R_{d,n}$. For any other lozenge $l = (v_1, \dots, v_4)$, the inequality $f(v_2) + f(v_4) \geq f(v_1) + f(v_3)$ is equivalent to the inequality (2c) of Definition \ref{def:limit_dual_hive}. \\
    \\
    Conversely, to any function $f \in P^{g[C]}_{\alpha,\beta,\gamma}$, associate the label map $L[f]: E_{n,d} \rightarrow \R_{\geq 0}$, $e = (u,v) \mapsto L[f](e) = f(v) - f(u)$ if $e$ has type $1$ or $2$ and $L[f](e) = 1- (f(v) - f(u))$ if $e$ has type $0$. We have that $\Psi^C(C, L) = f$. Moreover, the equality and inequality constraints in Definition \ref{def:limit_dual_hive} are equivalent to the rhombus concavity of $f$ so that $(C, L) \in H^C(\alpha, \beta, \gamma, \infty)$.
\end{proof}

\begin{figure}
    \centering
    \includegraphics[scale=0.6]{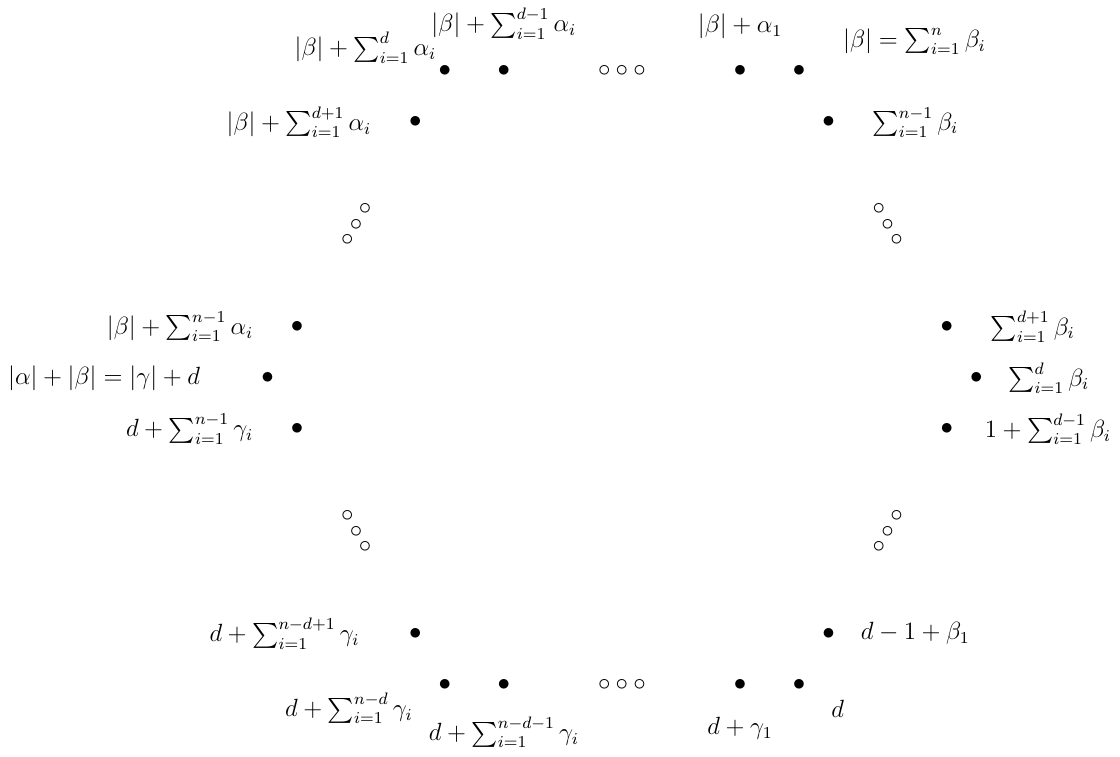}
        \caption{Boundary conditions on $P^{g[C]}_{\alpha,\beta,\gamma}$ induced by $\Psi^C$ on boundary conditions in Figure \ref{fig:boundary_limit_hive}.}
        \label{fig:bound_p_alpha_beta_gamma}
\end{figure}

\noindent
With the definitions above, we have an affine map $\Psi^C \circ \Phi^C: \R^D \rightarrow \tilde{P}^C_{\alpha,\beta,\gamma}$. In the rest of this section, we write $\det(\Psi^C \circ \Phi^C)$ for the determinant of the linear part of this application.

\begin{proposition}[Volume preservation by duality]
\label{prop:volume_preserving}

Let $C$ be a color map. Then, the map
\begin{equation}
    \Psi^C \circ \Phi^C: \R^D \rightarrow \tilde{P}^C_{\alpha,\beta,\gamma}
\end{equation}
satisfies 
\begin{equation}
    \vert \det(\Psi^C \circ \Phi^C) \vert=1
\end{equation}
and thus 
\begin{equation}
\label{eq:vol_equality}
    Vol(u\in\mathbb{R}^D,\Phi^C(u) \in  H^C(\lambda,\mu,\nu,\infty)) = Vol(u\in\mathbb{R}^D,\Psi^C\circ\Phi^C[u]\in P^{g[C]}_{\alpha,\beta,\gamma}) = Vol(P^{g[C]}_{\alpha,\beta,\gamma}).
\end{equation}

\end{proposition}

\begin{proof}

For $C = C_0$, enumerate $e_1,\ldots,e_D$ the horizontal edges in $C_0$ as in Lemma \ref{lem:unique_label_map_c0}. Then, for $u \in \R^D$ and $v \in R_{d, n}$,
\begin{equation*}
    [\Psi^{C_0} \circ \Phi^{C_0}(u)](v) = \sum_{e_i}(1-u_i) + (d-v_2)^+ + \sum_{i=1}^{v_2} \beta_i
\end{equation*}
where the sum is over edges $e_i$ of type $0$ connecting $v$ to the east boundary of $E_{n,d}$
with inverse
$$[\Psi^{C_0}\circ\Phi^{C_0}(f)]^{-1}(i)=1-(f(v)-f(v'))$$
where $v,v'$ the both endpoint of $e_i$ with the correct orientation.
Since $\Psi^{C_0} \circ \Phi^{C_0}$ and $[\Psi^{C_0} \circ \Phi^{C_0}(f)]^{-1}$ have integer coefficients, $\det(\Psi^{C_0} \circ \Phi^{C_0})=1$. \\
\\
If $C$ is general, introduce for each hexagon rotation $C \rightarrow C'$ given by an hexagon $h$ the map $\tilde{R}_{C \rightarrow C'}:\tilde{P}^C_{\alpha,\beta,\gamma}\rightarrow \tilde{P}^{C'}_{\alpha,\beta,\gamma}$ by 

\begin{enumerate}
\item For $f \in \tilde{P}^c_{\alpha,\beta,\gamma}$, extend $f$ uniquely to a function  $f: R_{d,n} \rightarrow \R$,

\item Let us describe how the hexagon rotation $C \rightarrow C'$ maps $f$ to another function $f': R_{d,n} \rightarrow \R$. The value of the center vertex $c$ of $h$ is uniquely determined by the position of the rigid lozenge in $h$ and the values of $f$ on $\partial h$. Indeed, if $(v, v', v'')$ are the three other vertices of the rigid lozenge such that $C((c, v'')) = m$, then $f(c) = f(v) + f(v') - f(v'')$. Note that $v, v', v'' \in \partial h$. For every vertex $u \in R_{d,n}$ other that the center vertex $c$ of $h$, we set $f'(u) = f(u)$. In the hexagon rotation $C \rightarrow C'$, the position of the the rigid lozenge changes and we set $f'(c) = f(w) + f(w') - f(w'')$ where $w, w', w'' \in \partial h$ are the new vertices of the rigid lozenge which are located on $ \partial h = \partial h$. 

\item The map $\tilde{R}_{C \rightarrow C'}(f)$ is defined as the restriction of $f'$ of the previous step to $R_{d,n} \setminus S^{C'}$.
\end{enumerate}

Note that the map $\tilde{R}_{C \rightarrow C'}$ is an affine bijection with integer coefficients whose inverse is given by $\tilde{R}_{C' \rightarrow C}$. Let us check that the following diagram is commutative
\begin{equation}
\label{eq:diagram}
    \begin{tikzcd}
\tilde{P}^{C}_{\alpha,\beta,\gamma} \arrow{r}{\tilde{R}_{C \rightarrow  C'}} & \tilde{P}^{C'}_{\alpha,\beta,\gamma}  \\
\tilde{H}^C(\alpha, \beta, \gamma, \infty) \arrow{u}{\Psi^C}  \arrow{r}{Rot[C, C']} & \tilde{H}^{C'}(\alpha, \beta, \gamma, \infty) \arrow[swap]{u}{\Psi^{C'}}
    \end{tikzcd}
\end{equation}
Let $H = (C, L) \in \tilde{H}^C(\alpha, \beta, \gamma, \infty)$ having an ABC hexagon $h$ with center vertex $c$. Denote by $C'$ the color map obtained after any rotation $h \mapsto h'$ and set $(C,L')=Rot[C, C'](C,L)$. For any vertex $u \neq c \in R_{d,n}$, one has $\tilde{R}_{C \rightarrow C'}( \Psi^C) (u) = \Psi^C(u)$. Moreover, if $u \neq c$ then one can find a path of edges from the south-east most vertex of $R_{d,n}$ to $u$ without any edge incident to $c$. Since the labels of these edges are not changed by $Rot[C \rightarrow C']$, $\Psi^{C'}(L')(u) = \Psi^{C}(L)(u) = \tilde{R}_{C \rightarrow C'}( \Psi^C(L)) (u) $ as desired. It remains to check that the same property holds for $u = c$. Denote $(v,v',v'')$, respectively $(w, w', w'')$ the vertices on $\partial h$ such that up to cyclic rotation $(v,c, v', v'' )$, respectively  $(w,c, w', w'')$, are vertices of the rigid lozenge in $h$, respectively $h'$, and that $C((u, v'')) = C'((u, w'')) = m$. Then, 

\begin{equation*}
    \tilde{R}_{C \rightarrow C'}( \Psi^C(L)) (u) = \Psi^C(L)(w) + \Psi^C(L)(w') - \Psi^C(L)(w'')
\end{equation*}
and 
\begin{equation*}
    \Psi^{C'}(Rot[C, C'](L)) (u) = \Psi^{C'}(L')(u)
\end{equation*}
Since the values of $\Psi$ do not depend on the chosen path, let us choose the following four paths. Take any path $p = (e_1, \dots, e_r) \in (E_{n,d})^r$ from the south-east vertex $A_0$ to $w$ such that for each $1 \leq i \leq r$, $e_i$ is not an interior edge of $h$. then,

\begin{enumerate}
    \item To evaluate $\Psi^C(L)(w)$, we choose the path $p$,
    \item To evaluate $\Psi^C(L)(w'')$, we append the edge $(w, w'') \in \partial h$ to $p$,
    \item To evaluate $\Psi^C(L)(w')$, we append edges $(w, w''), (w'', w') \in (\partial h)^2$ to $p$,
    \item To evaluate $\Psi^{C'}(L')(u)$, we append the edge $(w, u)$ to $p$,
\end{enumerate}
which gives
\begin{align*}
    \Psi^C(L)(w) &= \sum_{1 \leq i \leq r} L(e_i) + d \\
    \Psi^C(L)(w'') &= \sum_{1 \leq i \leq r} L(e_i) + L((w, w'')) +d \\
    \Psi^C(L)(w') &= \sum_{1 \leq i \leq r} L(e_i) +  L((w, w'')) + L((w'', w')) +d \\
    \Psi^{C'}(L')(u) &= \sum_{1 \leq i \leq r} L'(e_i) + L'((w, u)) +d.
\end{align*}
Since we have chosen edges $e_i \in p$ not interior to $h$, $L(e_i) = L'(e_i)$ for $1 \leq i \leq r$. The commutativity of the diagram is thus equivalent to 
\begin{equation*}
     L'((w, u)) = L((w, w'')) + L((w'', w')) - L((w, w'')) = L((w'', w')),
\end{equation*}
i.e 
\begin{equation*}
    L'((w, u)) = L((w'', w')).
\end{equation*}
Notice that $(w, u), (w'', w')$ are two edges of the same type in the rigid lozenge in $C'$ which implies that $L'((w, u)) = L'((w'', w')) = L((w'', w'))$, where the last equality is due to the fact that $(w', w'') \in \partial h$ so that its label value is unchanged by $Rot[C, C']$. The commutativity of \eqref{eq:diagram} is thus showed. \\
\\
Using \eqref{eq:diagram}, we have for any sequence of hexagon rotations $C_0\rightarrow C_1 \rightarrow \dots \rightarrow C$,

$$\prod \tilde{R}_{C_i\rightarrow C_{i+1}} \Psi^{C_0} \circ \Phi^{C_0} = \Psi^C \prod Rot[C_i \rightarrow C_{i+1}] \Phi^{C_0} = \Psi^C \circ \Phi^C.$$
On the left hand side, every map is affine with integer coefficient and with inverse having integer coefficients, so the same is true on the right hand-side, and thus 
$$\det(\Psi^C \circ \Phi^C)=1.$$
\noindent
The first equality of \eqref{eq:vol_equality} is due to Lemma \ref{lem:img_limit_dual_hive_is_toric} and the second is a consequence of $\det(\Psi^C \circ \Phi^C)=1.$

\end{proof}

\subsection{Proof of Theorem \ref{thm:main} and Corollary \ref{cor:volume_flat_connection}}

\begin{proof}[Proof of Theorem \ref{thm:main}]

By Theorem \ref{th:density_limit_quantum_coef} we have 
    
\begin{equation}
     \diff \prob[\gamma | \alpha, \beta] = \frac{ sf(n-1)(2\pi)^{(n-1)(n-2)/2} \Delta(\ed^{2i \pi \gamma})} {n!\Delta(\ed^{2i \pi \alpha})\Delta(\ed^{2i \pi \beta}) } \lim_{N \rightarrow \infty} N^{-(n-1)(n-2)/2}c_{ \lambda_N, \mu_N}^{\nu_N, d}.
\end{equation}
By Proposition \ref{prop:convergence_volume_dual} and Proposition \ref{prop:volume_preserving},

\begin{align}
     \diff \prob[\gamma | \alpha, \beta] &= \frac{ sf(n-1)(2\pi)^{(n-1)(n-2)/2}\Delta(\ed^{2i \pi \gamma})} {n!\Delta(\ed^{2i \pi \alpha})\Delta(\ed^{2i \pi \beta}) }  \sum_{C} Vol[u\in\mathbb{R}^D,\Phi^C[u]\in H^C(\alpha, \beta, \gamma, \infty)] \\
     &= \frac{sf(n-1)(2\pi)^{(n-1)(n-2)/2} \Delta(\ed^{2i \pi \gamma})} {n!\Delta(\ed^{2i \pi \alpha})\Delta(\ed^{2i \pi \beta}) }  \sum_{C} Vol(P^{g[C]}_{\alpha,\beta,\gamma}) \\
     &=\frac{sf(n-1)(2\pi)^{(n-1)(n-2)/2} \Delta(\ed^{2i \pi \gamma})} {n!\Delta(\ed^{2i \pi \alpha})\Delta(\ed^{2i \pi \beta}) }  \sum_{g:R_{d,n}\rightarrow \mathbb{Z}_3 \text{ regular}}Vol_g(P^g_{\alpha,\beta,\gamma}).
\end{align}

\end{proof}

\begin{proof}[Proof of Corollary \ref{cor:volume_flat_connection}]
From the expression \cite[Eq. (4.116)]{Witten_quantum_gauge} proven in \cite{jeffrey1998intersection}, we have 
$$Vol\left[\mathcal{M}(\Sigma_0^3,\alpha,\beta,\gamma)\right]=\frac{\#Z(SU_n)Vol(SU_n)}{Vol((\mathbb{R}/2\pi\mathbb{Z})^{n-1})^3}\sum_{\lambda\in \mathbb{Z}^n_{\geq 0}}\frac{1}{\dim V_{\lambda}}\chi_\lambda(e^{2i\pi\alpha})\chi_\lambda(e^{2i\pi\beta})\chi_\lambda(e^{2i\pi\gamma}),$$
where $Z(SU_n)$ is the center of $SU_n$. From \eqref{eq:density_horn_1}, we deduce that 
$$Vol\left[\mathcal{M}(\Sigma_0^3,\alpha,\beta,\gamma)\right]=\frac{\#Z(SU_n)Vol(SU_n)(2 \pi)^{n-1}n!}{Vol((\mathbb{R}/2\pi\mathbb{Z})^{n-1})^3|\Delta(\ed^{2i \pi \gamma})|^2}\diff \prob[-\gamma | \alpha, \beta] .$$
Corollary \ref{cor:volume_flat_connection} is then deduced from Theorem \ref{thm:main} and the fact that $Z(SU_n)=2^{(n+1)[2]}$, $Vol(SU_n)=\frac{(2\pi)^{n(n+1)/2-1}}{\prod_{k=1}^nk!}$ and $Vol((\mathbb{R}/2\pi\mathbb{Z})^{n-1})=(2\pi)^{n-1}$.
\end{proof}

\bibliographystyle{plain}
\bibliography{mybib}

\end{document}